\documentclass[12pt]{amsart}

\RequirePackage[OT1]{fontenc}
\RequirePackage{amsthm,amsmath}
\RequirePackage[numbers]{natbib}
\RequirePackage[colorlinks,citecolor=blue,urlcolor=blue]{hyperref}

\usepackage{geometry}
\usepackage{amsmath, amsthm, amssymb,  color, graphicx, hyperref, amsfonts,  caption, subcaption, tikz, bbm, algorithm, url, capt-of}
\RequirePackage{bm} 
 \usepackage{setspace}
\usepackage{jbradic_definitions}
\usepackage[T1]{fontenc}
\usepackage[utf8]{inputenc}
\usepackage{booktabs}
\usepackage{dcolumn}
\usepackage{units}
\usepackage{array}
\usepackage{booktabs}
\usepackage{graphicx}

\makeatletter

\hypersetup{
    bookmarks=true,         
    unicode=false,          
    pdftoolbar=true,        
    pdfmenubar=true,        
    pdffitwindow=false,     
    pdfstartview={FitH},    
    pdftitle={My title},    
    pdfauthor={Author},     
    pdfsubject={Subject},   
    pdfcreator={Creator},   
    pdfproducer={Producer}, 
    pdfkeywords={keyword1} {key2} {key3}, 
    pdfnewwindow=true,      
    colorlinks=true,       
    linkcolor=red,          
    citecolor=blue,        
    filecolor=magenta,      
    urlcolor=cyan           
}

\usepackage{color}

\newcolumntype{f}{>{$}l<{$}}
\newcolumntype{n}{l}
\newcolumntype{N}{>{\scriptsize}l}
\newcolumntype{v}[1]{>{\raggedright\hspace{0pt}}p{#1}}
\newcolumntype{V}[1]{>{\scriptsize\raggedright\hspace{0pt}}p{#1}}
\newcolumntype{B}[1]{>{\boldmath\DC@{.}{,}{#1}}l<{\DC@end}}
\newcolumntype{d}[1]{>{\DC@{.}{.}{#1}}l<{\DC@end}}
\newcolumntype{i}[1]{>{\DC@{.}{.}{#1}\mathnormal\bgroup}l<{\egroup\DC@end}}
\newcolumntype{s}[1]{>{\DC@{.}{.}{#1}\mathsf\bgroup}l<{\egroup\DC@end}}
\newcolumntype{R}[1]{%
  >{\begin{turn}{90}\begin{minipage}{#1}\scriptsize\raggedright\hspace{0pt}}l%
  <{\end{minipage}\end{turn}}%
}
\newcolumntype{x}{>{\scriptsize\raggedright\hspace{0pt}}X}
\makeatother

\usepackage{amsmath}
\usepackage{amsfonts}
\usepackage{amssymb}
\usepackage{graphicx}

\usepackage{verbatim}
\usepackage{bbm}
\usepackage{indentfirst}
\usepackage{graphics}
\usepackage{color}
\usepackage{amsthm}
\usepackage[section]{placeins}
\makeatletter
\newcommand{\rmnum}[1]{\romannumeral #1}
\newcommand{\Rmnum}[1]{\expandafter\@slowromancap\romannumeral #1@}
\newcommand\given[1][]{\:#1\vert\:}

\usepackage{tikz}

\newcommand{\var}{\mathrm{Var}}
\newcommand{\smallocal}{\mbox{\scriptsize$\ocal$}}

\newcommand{\jiaqi}[1]{\textcolor{brown}{Jiaqi: #1}}

\makeatother

\title[{Semiparametric Inference for Left-Censored   Models}]{\normalsize \sc Robust Confidence Intervals  in \\[0.5ex] {\normalsize  \sc High-Dimensional  Left-Censored Regression}}
 
\author{Jelena Bradic and Jiaqi Guo}
\address{Department of Mathematics\\
University of California San Diego\\
La Jolla, CA, 92093}

\begin{document}

\maketitle

\begin{abstract}
This paper  develops   robust  confidence intervals in high-dimensional and left-censored regression.   Type-I censored regression models are extremely  common in practice, where a competing event   makes the variable of interest   unobservable.  However, techniques developed for entirely observed data do not directly apply to the censored observations.   In this paper, we develop smoothed estimating equations   that  augment the   de-biasing method, such that
   the resulting  estimator  is adaptive to   censoring and is more robust to the misspecification  of the error distribution.     We propose a unified class of robust estimators, including Mallow's, Schweppe's and Hill-Ryan's  one-step  estimator. 
      In the ultra-high-dimensional setting, where the dimensionality
can grow exponentially with the sample size, we show that as long as the preliminary estimator converges faster than $n^{-1/4}$, the one-step estimator inherits asymptotic distribution of  fully iterated version. 
Moreover, we    show   that the size  of the residuals of the Bahadur representation matches those of  the simple linear models, $s^{3/4 } (\log (p \vee n))^{3/4} / n^{1/4}$  -- that is,  the effects of censoring asymptotically disappear. 
Simulation studies demonstrate that    our method is adaptive to
the censoring level  and asymmetry in the error distribution, and does not lose efficiency when the  errors are from symmetric distributions. Finally,  we  apply the developed method to  a real data set from the MAQC-II repository that is related to the  HIV-1 study.
\end{abstract}
%
%

 \section{Introduction}
 
 Left-censored data is a characteristic of many datasets. In physical science applications, observations can be censored due to
limit of detection and quantification in the measurements. For example, if a measurement device has a value limit on the lower end, the observations is recorded with the minimum value, even though the actual result is below the measurement range. In fact, many of the HIV studies have to deal with difficulties due to the lower quantification and detection limits of viral load assays \citep{S14}. In social science studies, censoring may be implied in the nonnegative nature or defined through human actions. Economic policies such as minimum wage and minimum transaction fee result in left-censored data, as quantities below the thresholds will never be observed.  With advances in modern data collection, high-dimensional data where the number of variables, p, exceeds the number of observations, n, are becoming more and more commonplace. HIV studies are usually complemented with observations about genetic signature of each patient, making the problem of finding the association between the number of viral loads and the gene expression values extremely high dimensional.  Hence, it is important   to develop inferential methods for left-censored and high-dimensional data.

A  general approach to
estimation of the unknown parameter $\betab$ in high dimensional settings, is given by the penalized M-estimator
\[
\hat \betab = \arg \min _{\betab \in \RR^p} \left\{  l(\betab) + P_{\lambda}{(\betab)} \right\},
\]
where  $l(\betab)$ is a loss function (e.g., the negative log-likelihood) and $P_{\lambda}(\betab)$ is
a penalty function with a tuning parameter $\lambda$. Examples include but are not limited to the Lasso, SCAD, MCP, etc. 
Significant progress has been made towards understanding the
estimation theory of penalized M-estimators with recent breakthroughs in 
quantifying the uncertainty of the obtained results.
However, no general theory exists for  high-dimensional estimation in the setting of left-censored data, not to mention  for understanding their uncertainty. 
 A few challenges of left-censored data are particularly difficult even in low-dimensional settings. Left-censored models rarely obey particular distributional forms, preventing the use of likelihood theory and  demanding for estimators that are semi-parametric in nature. For the same reasons, the estimators need  to be robust to  the presence of outliers in the design or model error. Lastly, theoretical results cannot be obtained using naive Taylor expansions and require the development of novel concentration of measure results.
 
 To bridge this gap, this
paper proposes a new mechanism, named as {\it smoothed estimating equations} (SEE) and {\it smoothed robust estimating equations} (SREE), for construction of  confidence intervals 
  for low-dimensional
components in high-dimensional left-censored models. For a high-dimensional parameter of interest $\betab \in \RR^p$, we aim to provide confidence intervals $(L_j, U_j)$ for any of its coordinates $j \in \{1,\dots, p \}$ while adapting to the left-censored nature of the problem. No distributional assumption will be made on the model error.  T The proposed estimators and confidence intervals are thus semiparametric. The main challenge in such setting is the non-differentiability of many of semiparametric loss functions, e.g, the least absolute deviation (LAD) loss. To handle this challenge, we apply
a  smoothing operation on the high dimensional  estimating equations, so that
the obtained SEE become smooth in the underlying $\beta_j$. Moreover, SEE are designed to handle high-dimensional model parameters and hence differ from the classical approaches of estimating equations. Although we consider left-censored models, the proposed SEE equations are  quite general and can apply to any non-differentiable loss function even with fully observed data. For example, they can  provide valid confidence sets using penalized rank estimator with both convex and non-convex penalties. 

We establish theoretical asymptotic coverage for confidence intervals
while allowing left-censoring and $p \gg n$. Moreover, for the estimators resulting from the SEE and SREE equations, we provide delicate Bahadur representation  and establish the order of the residual term. Under mild conditions, we show that the effects of censoring asymptotically disappear,  a result that is novel and of independent interest even in low-dimensional setting. Additionally, we establish a number of new uniform concentration of measure results particularly useful for many left-censored models. 

To further broaden our framework we   formally develop   robust Mallow's, Schweppe's and Hill-Ryan's estimators that adapt to the unknown censoring. We believe these estimators to be novel even in low-dimensional setting. This generalizes the classical
robust   theory developed by \cite{hampel_1986}.
We point out that the SEE framework can be viewed as an  extension of the de-biasing framework of \citep{zhang_zhang_2014}. In particular, the confidence intervals resulting from the SEE estimator   are asymptotically equivalent to  the confidence intervals of de-biasing methods in the case of a smooth loss function and non-censored observations.    However, SREE confidence sets provide robust alternative to the naive de-biasing as the resulting inference procedures are robust to the distributional misspecifications, and most appropriate for applications with extremely skewed observations.

 \subsection{Related Work}
 Given the prevalence of left-censored data, a large body of work in model estimation and inference has been dedicated to the topic. Estimation in the left-censored models has been studied since the 1950's. \cite{tobin_1958} first proposed the model with a nonnegative constraint on the response variable, which is also known as the Tobit-I model. Later, \cite{Amemiya_1973} proposed a maximum likelihood estimator  where a data transformation model is considered, and then impose a class of distributions for the resulting model error. However, as Zellner has noted \citep{Z96}, knowledge of the underlying data generating mechanism is seldom available, and thus models with parametric distributions may be subject to the  distributional misspecification. 
\cite{powell_1986a}, \cite{powell_1986b}, and \cite{NP90} pioneered   the development of robust inference procedures for the left-censored data, and relieved the assumption on model error distribution in prior work. \cite{powell_1984} introduced a LAD estimator, whereas \cite{P97}  introduced robust estimators and inference based on maximum entropy principles. \cite{Z15} proposed an alternative robust two-step estimator, while \cite{S11} and \cite{Z14} developed  distribution free and  rank-based  tests. For these models, the common assumption is that $p \leq n$.
 
For high-dimensional models,  and with Lasso being the cornerstone of achieving sparse estimators  \citep{T96}, numerous efforts have been made on establishing finite sample risk oracle inequalities of penalized estimators; examples include   \cite{greenshtein_ritov_2004}, \cite{bunea_tsybakov_wegkamp_2007}, \cite{candes_tao_2007}, \cite{vdgeer_2008}, \cite{zhang_huang_2008}, \cite{bickel_ritov_tsybakov_2009}, \cite{meinshausen_yu_2009} and\cite{negahban_ravikumar_wainwright_yu_2012}. Regarding censored data, \cite{muller_vdgeer_2014} offered a penalized version of   Powell's estimator. 
However, substantially smaller efforts have been made toward high-dimensional inference, namely  confidence interval construction and statistical testing in the uncensored high-dimensional setting, not to mention in the censored high-dimensional setting. Recently, \cite{javanmard_montanari_2014}, \cite{vdgeer_buhlmann_ritov_dezeure_2014} and \cite{zhang_zhang_2014}  have corrected the bias of high-dimensional regularized estimators by projecting its residual to a direction close to that of the efficient score. Such technique, named de-biasing, is parallel to
 the bias correction of the nonparametric estimators in the semiparametric inference  literature \citep{bickel_klaassen_ritov_wellner_1998}.   \cite{vdgeer_buhlmann_ritov_dezeure_2014}  considered an extension of this technique to  generalized linear model, while \cite{sun_zhang_2012b} and \cite{r15} considered extensions to  graphical models. 
\cite{B14} developed a three-step bias correction technique for quantile estimation. 
For inference in censored high-dimensional linear models, to the best of our knowledge, there has been no prior work.
   It is worth pointing out that the main contribution of this paper is in understanding fundamental limits of semiparametric inference for left-censored models.

 \subsection{Organization of the Paper}
 
 In Section 2, we propose the  smoothed estimating equations (SEE) for left-censored linear models. In Section 3, we establish general results for   confidence regions and the  Bahadur representation of  the SEE estimator. We also emphasize on  the new concentration of measure results, the building blocks of the  main theorems.  In Section 4, we  develop robust and left-censored Mallow's, Schweppe's and Hill-Ryan's estimators and present their theoretical analysis.  Section 5 provides numerical results on simulated and real data sets.
We defer technical details to the Supplementary Materials.

     \section{Smoothed Estimating Equations for Left-Censored \\ High-Dimensional Models}
     
     We begin by introducing
a general modeling framework followed by    highlighting the
difficulty for directly applying  existing inferential methods (such are de-biasing, score, Wald, etc.) to the models with  left-censored observations. Finally, we propose a new mechanism, named  smoothed estimating equations, to construct  semi-parametric confidence regions in high-dimensions.
     
     \subsection{Left-Censored Linear Model}
     We consider the problem of confidence interval construction where we observe a 
  vector of responses $Y=(y_1,\dots,y_n)$ and their  censoring level $c=(c_1,\dots, c_n)$ together with covariates $X_1,\dots X_p$. 
 The type of statistical inference under consideration   is regular in the sense that it
does not require model selection consistency. A characterization of such inference
is that it does not require a uniform signal strength  in the model.
Since ultra-high dimensional data often display heterogeneity,
we advocate  a robust confidence interval framework.
 We  begin with the following
latent regression model: 
\begin{align*}
y_i= \max\left\{c_i,x_i \betab^*+\varepsilon_i\right\},
\end{align*}
where the  response, $Y$, and the  censoring level, $c$, are observed and the vector $\betab^* \in \RR^p$ is unknown.  This model is often called  the semi-parametric censored regression model, whenever the distribution of the error, $\varepsilon$, is not specified.  
We assume that $\{\varepsilon_i\}_{i=1}^n$ are independent
across $i$  and are independent of $x_i$. Matrix $X=\left[X_1,\cdots,X_p\right]$ is the $n\times p$ design matrix, with $x_i$ being the $i^{th}$ row.    We also denote $S_{\betab}:=\{j|\betab_j\neq0\}$ as the active set of variables and its cardinality by $s_{\betab}:=|S_{\betab}|$.   We restrict our study to constant-censored model, also called Type-I Tobit model, where each entry of the censoring vector $c$ is the same. Without loss of generality, we focus on the zero-censored model  
\begin{align}\label{model}
y_i = \max\left\{0,x_i\betab^*+\varepsilon_i\right\}.
\end{align}
       For the censored model \eqref{model} but when $p \leq n$, Powell introduced a censored least absolute deviation loss (CLAD), where 
     \[
     l (\betab, y_i,x_i) = |y_i - \max \{ 0, x_i \betab\}|.
     \]

     \subsection{Challenges of Existing High-Dimensional Methods}
     
Although great progress has already been made in understanding the hypothesis testing in high-dimensions, directly applying existing methods to the case of left-censored observations might present a challenge.
      Inference for robust losses in the presence of  censoring is particularly difficult \citep{WF09} even in low-dimensional setting, and it is well known that left-censoring results do not extend from the results of fully observed data.
     A similar paradigm exists in high-dimensions.   
  Several problems are  immediately evident. First, if  observations are censored, there  will hardly be a model error that belongs to the family of unimodal distributions. Thus, it is necessary to make a method that works equally well with symmetric and asymmetric distributions. 
  In other words, a robust method is preferred over   maximum likelihood or least squares approaches. Second, the optimal inference function depends on the model censoring.   In particular, population Hessian matrix for the left-censored data does not have the simple form irrespective of the left-censoring.  Therefore,  methods that ignore censoring will not be efficient; vanilla de-biasing \citep{javanmard_montanari_2014} and \citep{zhang_zhang_2014} can produce biased and conservative confidence intervals with much larger width.  Third, the model itself is non-linear and is not well approximated by an additive linear model.  Therefore, additive models, although very flexible do not apply to the problem we consider.    
Although inference in high-dimensions that addresses the first concern has already been proposed and include an LAD-based inferential procedure \citep{Belloni2013Uniform}, a score based procedure  \citep{YL16} and a quantile-based procedure \citep{Z16}, neither of these address the second challenge arriving from the left-censored nature of the data, hence can be highly inefficient.

     \subsection{Smoothed Estimating Equations (SEE)} 
 Our estimator is  motivated by the principles of estimating equations. 
We begin by observing that 
the true parameter vector $\betab^*$ satisfies the  population system of equations
     \begin{equation} \label{eq:naive}
      \EE \Bigl[ \Psi(\betab^*) \Bigl]=0.
     \end{equation}
     where $\Psi (\betab) = n^{-1} \sum_{i=1}^n \psi_i(\betab)$ for a class of suitable functions $\psi_i$. For the CLAD loss 
     \[
\psi_i(\betab)= \mbox{sign}\left(y_i-\max\{0,x_i\betab\}\right)   w_i^\top (\betab) 
 \]
 and  $w_i (\betab) = x_i \ind\{ x_i \betab>0\}$.
     In high-dimensional setting, where $p \geq n$ solving estimating equations $\Psi(\betab)=0$ has several drawbacks. In particular,  for semi-parametric estimation and inference in model \eqref{model},  the function $\Psi$ is non-monotone as the loss is non-differentiable or non-convex. Hence, the system above   has multiple roots resulting in an estimator that is ill-posed.   
     Instead of solving the system \eqref{eq:naive} directly, we augment it
       by observing that, for a suitable choice of the matrix  ${\boldsymbol {\Upsilon}} \in \RR^{p \times p}$,  $\betab^*$ also satisfies  the system of equations 
     \begin{equation}\label{see1}
     \EE[\Psi(\betab^*)] + {\boldsymbol {\Upsilon}} [\betab -\betab^*] = 0.
     \end{equation}
     To avoid difficulties with non-smoothness of $\Psi$, we propose to consider with a matrix   
     ${\boldsymbol {\Upsilon}} = {\boldsymbol {\Upsilon}}(\betab^*)$, where the matrix ${\boldsymbol {\Upsilon}} (\betab^*)$ is defined as 
     \[
     {\boldsymbol {\Upsilon}}(\betab)= \EE_X  \left[\nabla_{\betab} S (\betab ) \right]  ,
     \]
for a smoothed vector 
     $S (\betab )$  defined as 
\[
S (\betab )    =\int_{-\infty} ^ \infty   \Phi(\betab, x ) f_\varepsilon(x)dx.
\]
In the above display $\Psi(\betab^*) = \Phi (\betab^*, \varepsilon)$, for a suitable function $\Phi = n^{-1} \sum_{i=1}^n \phi_i$ and  $\phi_i: \RR^p \times \RR \to \RR $ whereas, $f_\varepsilon$ denotes the density of the model error \eqref{model}. Additionally,
 $\EE_X$ denotes expectation with respect to the random measure generated by the vectors $X_1,\dots, X_n$.  
    For the  Powel's CLAD loss, we observe that the smoothed  vector takes the form
   \begin{align}\label{psi_subgradient}
     S(\betab) = n^{-1} \sum_{i=1}^{n}\left[1-2P_\varepsilon\left(y_i-x_i\betab^*\leq0\right)\right]\left(w_i(\betab^*)\right)^T,
  \end{align}
  where $P_{\epsilon}$ denotes the probability measure generated by the errors $\varepsilon$  \eqref{model}.
 This   leads to  
$
\nabla_{\betab^*}  S(\betab^*) 
=  2f_\varepsilon(0) n^{-1} \sum_{i=1}^{n}$ $w_i(\betab^*)^Tw_i(\betab^*) $
and a matrix
 \begin{equation}\label{eq:ipsilon}
{\boldsymbol {\Upsilon}}(\betab^*) = 2f_\varepsilon(0) \EE_X \left[ n^{-1} \sum_{i=1}^{n}w_i(\betab^*)^Tw_i(\betab^*) \right] : = 2f_\varepsilon(0) \Sigmab(\betab^*).
 \end{equation}
   To infer the parameter $\betab^*$, we adapt a one-step approach. We can observe that  solving  SEE equations \eqref{see1} requires inverting the matrix ${\boldsymbol {\Upsilon}}(\betab^*) $, as we are looking for a solution $\betab$ that satisfies
\[
     {\boldsymbol {\Upsilon}}(\betab^*)  \betab = {\boldsymbol {\Upsilon}}(\betab^*)   \betab^*   -  \EE \Psi( \betab^* ). 
\]
 For low-dimensional problems, with $p \ll n$, this can efficiently be done by considering an initial estimate $\hat \betab$
 and 
a sample plug-in estimate, ${\boldsymbol {\Upsilon}}(\hat \betab) $, of ${\boldsymbol {\Upsilon}}(\betab^*)$,
 \begin{equation}\label{eq:sigma}
 {\boldsymbol {\Upsilon}}(\hat \betab) = 2   n^{-1}  \hat f_\varepsilon(0)\sum_{i=1}^{n}w_i(\hat \betab )^Tw_i(\hat \betab )  : = 2 \hat f_\varepsilon(0) \hat \Sigmab(\hat \betab)
 \end{equation}
and sample estimate of  $\EE \Psi( \betab^* ) $, denoted with $  \Psi( \hat \betab  ) $.
 However,   when $p \gg n$ this is highly ineffective. Instead, it is more efficient to directly estimate ${\boldsymbol {\Upsilon}}^{-1}(\betab^*) =\Sigmab^{-1}(\betab^*)/2f_\varepsilon(0)$.  Let $\tilde \Omegab(\hat \betab ) $ be an estimate of ${\boldsymbol {\Upsilon}}^{-1}(\betab^*) $. Then, the SEE estimator is defined as 
 \[
    \tilde  \betab = \hat \betab  -  \tilde \Omegab(\hat \betab)  \Psi(\hat \betab ).
\]

Proposed SEE can be viewed as a high-dimensional extension of  inference from estimating equations.

     \begin{remark}
Although we consider a left-censored linear model, the proposed  SEE methodology  applies more broadly. For example,  our framework includes  loss functions based on ranks  or non-convex loss functions for the fully observed data.  For instance, the
method in  \cite{vdgeer_buhlmann_ritov_dezeure_2014}  is based on  inverting KKT conditions might not directly apply   for the non-convex loss functions (e.g., Cauchy loss) or rank loss functions (e.g., log-rank loss).
  \end{remark}

\subsubsection{Estimation of  the Scale in Left-Censored Models}
We will introduce the methodology for estimating each row of the matrix $\Sigmab^{-1}(\betab^*)$. For further analysis it is useful to define  $W(\betab)$ as a matrix composed of row vectors $w_i(\betab)$;  
$ W(\betab) =  A(\betab)X,$ 
where 
 $A(\betab) = \mbox{diag}\left(\ind\left(X {\betab} > 0\right)\right) \in \RR^n \times \RR^n.$ The methodology is motivated by the following simple observation:
\begin{align*}
\tau_j^{-2} \Gamma_{(j)}^\top \Sigmab(\betab^*)   = \eb_j,
\end{align*}
where 
$\Gammab_{(j)}(\betab^*) = \left[
- \gammab_{(j)}^*(\betab^*)_1, \cdots,  - \gammab_{(j)}^*(\betab^*)_{j-1},  1,  - \gammab_{(j)}^*(\betab^*)_{j+1},  \cdots,  - \gammab_{(j)}^*(\betab^*)_p
\right]$ with
\begin{align}\label{eq:optimization_pop}
\gammab_{(j)}^*(\betab) &:= \underset{\gammab  \in \RR^{p-1}}{\argmin} \  \mathbbm{E} \left\| W_j(\betab) - W_{-j}(\betab)\gammab \right\|_2^2/n
\end{align}
and 
$$ \tau_j^2 := n^{-1}\mathbbm{E} \left\| W_j(\betab^*) - W_{-j}(\betab^*)\gammab_{(j)}^*(\betab^*) \right\|_2^2.$$
This motivates us to consider the following as an estimator for the inverse $\Sigmab^{-1}(\betab^*)$. Let    $\hat{\gammab}_{(j)}(\hat \betab)$ and $\hat{\tau}_j^2$  denote the estimators of $\gammab_{(j)}^*(\betab^*)$ and $\tau_j^2$, respectively. 
We will show that a simple plug-in Lasso type estimator is sufficiently good for construction of confidence intervals. We propose to estimate $\gammab_{(j)}^*(\betab^*)$, with the following $l_1$ penalized plug-in least squares regression,
\begin{align}\label{nodewise_lasso}
\hat{\gammab}_{(j)}(\hat\betab) = \underset{\gammab  \in \RR^{p-1}}{\argmin} \left\{ n^{-1} \left\| W_j(\hat\betab) - W_{-j}(\hat\betab)\gammab \right\|_2^2 + 2\lambda_j \|\gammab\|_1 \right\}.
\end{align}
Notice that this regression does not trivially share all the nice properties of the penalized least squares, as in this case the rows of the design matrix are not independent and identically distributed. An estimate of  $\tau_j^2 $ can then be defined through the estimate of the residuals 
\begin{align} \label{eqn:zeta*}
\zetab_j^* := W_j(\betab^*) - W_{-j}(\betab^*)\gammab_{(j)}^*(\betab^*).
\end{align}
We propose the plug-in estimate for $\zetab_j^*$ as 
$
 \hat{\zetab}_j = W_j(\hat\betab) - W_{-j}(\hat\betab)\hat{\gammab}_{(j)}(\hat\betab),
$
 and a bias corrected  estimate  of $\tau_j^2 $ defined as
\begin{align}\label{tau_hat}
  \hat{\tau}_j^2 (\lambda_j)&= n^{-1}{\hat \zetab_j}^\top \hat \zetab_j +\lambda_j\left\|\hat{\gammab}_{(j)}(\hat{\betab})\right\|_{1}. 
\end{align}
Observe that the naive estimate $n^{-1}{\hat \zetab_j}^\top \hat \zetab_j $ does not suffice due to the bias carried over by the penalized estimate $\hat \gammab_{(j)}(\hat \betab)$.
Lastly, the matrix estimate of $\Sigmab^{-1}(\betab^*)$, much in the same spirit as \cite{zhang_zhang_2014}  is defined with
\begin{equation}\label{eq:theta}
\Omega_{jj}(\hat{\betab}) =  \hat{\tau}_j^{-2} ,
\qquad 
\Omega_{j,-j}(\hat{\betab}) = - \hat{\tau}_j^{-2} \hat{\gammab}_{(j)}(\hat \betab), \qquad j=1,\dots,p.
\end{equation}

 \begin{remark}
 The proposed scale estimate   can be considered as the censoring adaptive extension of the graphical lasso estimate of \cite{vdgeer_buhlmann_ritov_dezeure_2014}. 
Certainly,  there are alternative procedures for estimating  $\Sigmab^{-1}(\betab^*)$ with examples parallel to the Dantzig selector.
However, we believe, the  choice of tuning parameters for such estimates will depend on the unknown sparsity of $\betab^*$, thus   will be especially difficult to choose in practice.  
 \end{remark}

\subsubsection{Density Estimation}

Whenever the model considered is homoscedastic, i.e., $\varepsilon_i$ are identically distributed with a density function $f_{\varepsilon}$ (denoted whenever possible with $f$), we propose a novel density estimator designed to be adaptive to the left-censoring in the observations.
 For a positive   bandwidth sequence $\hat h_n$, we define the density estimator  of $f_{\varepsilon}(0)$ as 
\begin{equation}\label{eq:density}
\hat f(0) =
\hat h_n^{-1}\sum_{i=1}^n \frac
{ \ind(x_i\hat{\betab}>0)\ind(0\leq y_i-x_i\hat{\betab}\leq \hat h_n)}
{\sum_{i=1}^n\ind(x_i\hat{\betab}>0)} .
\end{equation}
 Of course, more elaborate smoothing schemes for the estimation of $f(0)$
could be devised for this problem, but there seems to be no a priori reason to
prefer an alternate estimator. 
\begin{remark}\label{remark4}
We will show that a choice of the bandwidth sequence satisfying $h_n^{-1} =  \smallocal ( \sqrt{n /(s \log p)})$ suffices. 
However, we also  propose an adaptive choice of the bandwidth sequence and consider $\hat h_n=\smallocal(1)$  such that  
 \[
 \hat h_n =  c \left\{s_{\hat \betab} \log p /n \right\}^{-1/3}   \mbox{median} \left\{ y_i > x_i \hat \betab + \sqrt{\log p/n}, \ x_i \hat \betab >0 \right\},
 \]
 for a constant $c>0$.   Here, $s_{\hat \betab}$ denotes the size of the estimated set of the non-zero elements of the initial estimator $\hat \betab$, i.e., $s_{\hat \betab} = \| \hat \betab\|_0 $.
 \end{remark}

     \subsection{Confidence Intervals}
     
Following the SEE principles, the one-step solution is defined as an estimator,
\begin{align}\label{beta_os}
\tilde{\betab} =   \hat \betab   -   \Omegab(\hat \betab) \Psi(\hat \betab )  /{2\hat f(0)}   .
\end{align}

For the presentation of our coverage rates of the confidence interval \eqref{eq:ci} and \eqref{eq:pi}, we start with  the   Bahadur representation.  Lemmas 1-6 (presented below) enable us to  
 establish the following decomposition for the introduced  one-step estimator $\tilde\betab$,
\begin{align}\label{eq:short_bahadur}
&\sqrt{n}\left(\tilde{\betab}-\betab^*\right)  = \frac{1}{2f(0)}\Sigmab^{-1}(\betab^*)\frac{1}{\sqrt{n}}\sum_{i=1}^n\psi_i(\betab^*) + \Delta.
\end{align}
where the vector $\Delta$ represents the residual component. 
We show that the residual vector's size is small uniformly and that the leading term is asymptotically normal.
The theoretical guarantees required from an initial estimator $\hat \betab$ is presented below.

 \vskip 7pt
 
{\bf  Condition (I)}: \label{condition_i}
{ \it An initial estimate $\hat \betab$ is such that for the left-censored model, irrespective of the density assumptions, the following three properties hold. There exists a sequence of positive numbers $r_n$ and $d_n$ such that $r_n,d_n \to 0$ when $n \to \infty$ and  $\| \hat \betab - \betab^*\|_2= \Ocal_P(r_n)$, $\|\hat \betab - \betab^* \|_1 = \Ocal_P(d_n)$ and $\| \hat \betab\|_0 = t=\Ocal_P(1)$.
}
 
 \vskip 7pt

$l_1$ penalized CLAD estimator studied in \cite{muller_vdgeer_2014}, under suitable conditions and a choice of the tuning parameter $\lambda > C \sqrt{\log p/n}$, satisfies the Condition (I) with $d_n = s_{\betab^*}\sqrt{ {\log p}/{n}}$. Results of \cite{muller_vdgeer_2014} can be extended to guarantee   that $r_n^2 =s_{\betab^*} \log p /n$ and $\| \hat \betab\|_0 = \Ocal_P( s_{\betab^*}$ \ $  \times\lambda_{\max}(X^\top X)/n)$, under the same conditions (proof is trivial extension of \cite{bickel_ritov_tsybakov_2009} and is hence not provided).  It is worth noting that the above condition does not assume model selection consistency of the initial estimator.
 
 \vskip 2pt
With the normality result of the proposed estimator $\tilde \betab$ (as shown in Theorem \ref{cor:ci}, Section \ref{sec:theory}), we are now ready to present the confidence intervals. Fix  $\alpha$ to be in the interval $(0,1)$ and let $z_{\alpha}$ denote the $(1-\alpha)$th standard normal percentile point. 
Let $\cbb$ be a fixed vector in $\RR^p$. Based on the results of Section \ref{sec:theory}, the standard studentized approach leads to a $(1-2\alpha) 100\%$ confidence interval for  $\cbb^\top \betab^* $ of the form 
\begin{equation}\label{eq:ci}
\mbox{ I}_{n}= \biggl( \cbb^\top \tilde\betab  -a_n,\cbb^\top  \tilde\betab +a_n \biggl),
\end{equation}
where $\tilde \betab $ is defined in   \eqref{beta_os} and 
\begin{equation}\label{eq:pi}
a_n =  z_{\alpha} \sqrt{  \cbb^\top  \Omegab (\hat\betab) \hat \Sigmab (\hat \betab) \Omegab(\hat\betab) \cbb} \Bigl/ \Bigl. 2 \sqrt{n}   \hat f(0)
\end{equation}
with $\Omegab (\hat\betab) $ defined in \eqref{eq:theta},  $\hat \Sigmab(\hat \betab)$ defined in \eqref{eq:sigma} 
and $\hat f(0)$ as defined in \eqref{eq:density}.
In the above, for $\cbb=\eb_j$, the above confidence interval provides a coordinate-wise confidence interval for each $\beta_j $, $1 \leq j \leq p$. Notice that the above confidence interval is robust in a sense that it is asymptotically valid irrespective of the distribution of the error term $\varepsilon$. 
     
     \section{Theoretical Results}\label{sec:theory}

We begin theoretical analysis with the following decomposition of \eqref{beta_os}
\begin{align}\label{betaos_decomp1}
\sqrt{n}&\left(\tilde{\betab}-\betab^*\right) \nonumber \\
&=  \frac{1}{2f(0)}\Sigmab^{-1}(\betab^*)\frac{1}{\sqrt{n}}\sum_{i=1}^n\psi_i(\betab^*) +\frac{1}{2f(0)}\left(\Omegab(\hat{\betab})-\Sigmab^{-1}(\betab^*)\right)\frac{1}{\sqrt{n}}\sum_{i=1}^n\psi_i(\betab^*)  \nonumber\\
& \qquad +\sqrt{n}\left(\hat{\betab}-\betab^*\right) +\frac{1}{2f(0)}\Omegab(\hat{\betab})\sqrt{n}\left(n^{-1}\sum_{i=1}^n\psi_i(\hat{\betab})-n^{-1}\sum_{i=1}^n\psi_i(\betab^*)\right).
\end{align} 
We can further decompose the last factor of the last term in \eqref{betaos_decomp1} as
$
n^{-1}\sum_{i=1}^n\psi_i(\hat{\betab})-n^{-1}\sum_{i=1}^n\psi_i(\betab^*) \nonumber
=  \mathbb G_n(\hat \betab) -  \mathbb G_n (\betab^*) + n^{-1}\sum_{i=1}^n\EE \left[ \psi_i(\hat{\betab})- \psi_i(\betab^*) \right] ,
$
where  
\begin{equation}\label{eq:Gn}
\mathbb G_n( \betab) = n^{-1}\sum_{i=1}^n \left[ \psi_i( {\betab}) -  \EE\psi_i(\betab) \right].
\end{equation}

To characterize the behavior of individual terms in the decomposition above, we develop a sequence of results presented below that rely on a number of conditions that we explain below. 
 We begin with a simple design assumption.
 
 \vskip 7pt 
 
{\bf  Condition (X)}: \label{condition_x}
{ \it There exists a bounded constant $K$, such that $ \max|X_{ij} | \leq K$, for all $i,j$. Moreover, $x_i$'s are i.i.d. random variables with  $\EE[X_{ij}^2]=1$, for all $i=1,\cdots,n$ and $j=1,\cdots,p$.  For some constant $K_0$, $x_i(\betab-\betab^*)$ and $\max\{0,x_i\betab^*\}-\max\{0,x_i\betab\}$ take value in interval $[-K_0,K_0]$, for all $i=1,\cdots,n$ and all $ \betab\in \mathcal{B}$ for a bounded set $\mathcal{B}$. 
}
 
 \vskip 7pt

The bound on $X_{ij}$ is  quite standard in high-dimensions \citep{vdgeer_buhlmann_ritov_dezeure_2014}.  However, in many cases, if $X$ follows an unbounded distribution, we can approximate its distribution with a truncated one.      Next, we rely on a set of very mild model error assumptions.

 \vskip 7pt

{\bf  Condition (E)}: \label{condition_e}
{ \it The error distribution $F$ has median 0, and is everywhere continuously differentiable, with density $f$, which is bounded above, $f_{\max}<\infty$, and  below, $f_{\min}>0$. Also, the density $f$ is bounded away from 0 at the origin, $f(0+)>0$.  Furthermore, $f(\cdot)$ is also Lipschitz continuous,
$
|f(t_1)-f(t_2)|\leq L_0\cdot|t_1-t_2|, \text{ for some } L_0>0.
$
Moreover, the function  $G_i(z,\betab,r) = \EE \left[\ind(|x_i\betab|\leq\|x_i\|\cdot z)\|x_i\|^r\right] \leq K_1 z$ and is $\smallocal(z)$ for $z$ near zero  and $r=0,1,2$ uniformly in $i$. 

}

 \vskip 7pt

The above assumption is the only condition we assign to the error distribution    \citep{powell_1984}. We require  the error density function to be with bounded first derivative.
 This excludes densities with unbounded first moment, but includes a   class of  distributions much larger than the Gaussian.
Moreover, this assumption implies that $x_i\betab$ are distributed much like the error $\varepsilon_i$, for $\betab$ close to $\betab^*$ and $x_i\betab$ close to the censoring level $0$.

\begin{lemma}
\label{lemma0}
  Suppose that the Conditions {\bf{(X)}}, {\bf{(E)}} hold.  
 Consider the class of parameter spaces modeling sparse vectors with at most $t$ non-zero elements,
$
\Ccal(r, t)  = \{ 
\wb \in \RR^{p} \mid \norm{\wb}_2 \leq r_n,
 \sum_{j =1}^ p\ind\{w_j \neq 0\} \leq t \}
$
where $r_n$ is a sequence of positive numbers. Then, there exists a fixed
  constant $C$ (independent of $p$ and  $n$), such that 
  the process  $\mu_i(\deltab) = \ind\{x_i\deltab \geq x_i \betab^* \}$ satisfies with probability $1-\delta$.
\[
\sup_{\deltab \in \Ccal(r_n, t)} n^{-1} \left| \sum_{i=1}^n  \mu_i(\deltab) -  \EE [\mu_i(\deltab)] \right| \leq C \left( 
\sqrt{\frac{   r_n   t  \sqrt{t} 
\log(n p /\delta)}{n}} 
\bigvee
\frac{  t \log(2n p /\delta)}{n} \right).
\]
\end{lemma}

The preceding Lemma immediately implies  strong approximation of the empirical process with its expected process,  as long as $r_n$, the estimation error, and $t$, the size of the estimated set  of the initial estimator, are sufficiently small.  The power of the Lemma \ref{lemma0} is that it holds uniformly for a class of parameter vectors enabling a wide range of choices for the initial estimator.   

Apart from the condition on the design matrix $X$ and the error distribution, we need conditions on the censoring level of the model \eqref{model} for further analysis.

 \vskip 7pt

{\bf  Condition (C)}: \label{condition_c}
{ \it There exists some constant $C_2>0$, such that for all $\betab$ satisfying $\|(\betab-\betab^*)_{S_{\betab^*}^C}\|_1 \leq 3\|(\betab-\betab^*)_{S_{\betab^*}}\|_1$, 
$ \left\|\max\{0,X\betab^*\}-\max\{0,X\betab\}\right\|_2^2 \geq C_2\|X(\betab-\betab^*)\|_2^2,$  where the $\max$ operation is entry-wise  maximization. 
}

\vskip 7pt
 
The censoring level $c$ has a direct influence on the constant $C_2$.  In general, higher values for $c_i$ increase the number of censored data. The bounds for the coverage probability  (see Theorem \ref{cor:fixed:ci}) do not depend on the  censoring level $c$. The fact that the censoring level does not directly appear in the results should be understood in the sense that the percentage
of the censored data is important, not the censoring level. 

\vskip 7pt

{\bf  Condition (CC)}: \label{condition_cc}
{\it For  some compatibility constant $\phi_0>0$   and all $\betab$ satisfying $\|(\betab-\betab^*)_{S_{\betab^*}^C}\|_1$$ \leq$
 $ 3\|(\betab-\betab^*)_{S_{\betab^*}}\|_1$, the following holds 
$n \phi_0^2 \|(\betab-\betab^*)_{S_{\betab^*}}\|_1^2 \leq (\betab-\betab^*)^\top \EE[X^\top  X](\betab-\betab^*)  s_{\betab^*} .$ Let $v_n$ be the smallest eigenvalue of $\Sigmab(\betab^*) $. Then $v_{n_1}>v_{n_2}$, for  
$n_1>n_2$. Additionally, $v_n$ is also strictly positive, with $1/v_n=\mathcal{O}(1)$
and assume $\underset{j}{\max}\ \Sigmab_{jj}(\betab^*)=\mathcal{O}(1)$.
}

\vskip 7pt

Note that this compatibility factor does not impose any restrictions on the censoring of the model, i.e., it is the same as the one introduced for linear models \citep{bickel_ritov_tsybakov_2009}.  
 Observe that  this condition does not impose distribution of $W$ to be Gaussian or continuous. However,  it  requires that $\Sigmab(\betab^*)$, the population covariance matrix, is at least invertible, a condition unavoidable even in linear models.

Next, we present a linearization result useful for further decomposition of the Bahadur representation \eqref{betaos_decomp1}.

\begin{lemma}\label{lemma1}
  Suppose that the conditions {\bf{(X)}, {\bf{(E)}}} hold. For all $\betab$, such that $\|\betab-\betab^*\|_1<\xi$, the following representation holds
\begin{align*}
n^{-1}\sum_{i=1}^n\EE\psi_i(\betab) = 2f(0) \Sigmab( {\betab}^*)(\betab^*-\betab) + \mathcal{O}(\|\betab-\betab^*\|_1)(\betab^*-\betab).
\end{align*} where $ \Sigmab(\betab^*)$ is defined in \eqref{eq:ipsilon}.
\end{lemma}
 
 Once the properties of the initial estimator are provided, such is Condition {\bf{(I)}},
 Lemma \ref{lemma1} can be used to  linearize   the population level difference of the  functions   $\psi_i(\hat{\betab}) $ and $\psi_i( {\betab}^*) $. Together with Lemma \ref{lemma0},  Lemma \ref{lemma1} allows us to overpass the original highly discontinuous  and non-convex loss function.
Utilizing  Lemma \ref{lemma1}, Conditions {\bf{(I)}}-{\bf{(CC)}} and representation \eqref{betaos_decomp1},  the  Bahadur representation of $\tilde \betab$ becomes 
\begin{align}\label{delta_1}
&\sqrt{n}\left(\tilde{\betab}-\betab^*\right)  = \frac{1}{2f(0)}\Sigmab^{-1}(\betab^*)\frac{1}{\sqrt{n}}\sum_{i=1}^n\psi_i(\betab^*) 
+ I_1 + I_2 + I_3 + I_4
\end{align}
where
\begin{align*}
 I_1 =\sqrt{n}\left(I-\Omegab(\hat{\betab})   \Sigmab(\betab^*)  \right)\left(\hat{\betab}-\betab^*\right),  \ \ 
 I_2 ={ -} \frac{1}{2f(0)}\Omegab(\hat{\betab})\sqrt{n}\cdot \mathcal{O}_P(\|\hat{\betab}-\betab^*\|_1)(\hat{\betab}-\betab^*)\ \  \ \  \ \   \\
 I_3=\frac{1}{2f(0)}\left(\Omegab(\hat{\betab})-\Sigmab^{-1}(\betab^*)\right)\frac{1}{\sqrt{n}}\sum_{i=1}^n\psi_i(\betab^*), \ 
I_4=\frac{1}{2f(0)}\Omegab(\hat{\betab})\sqrt{n}
\left[  \mathbb G_n(\hat \betab) -  \mathbb G_n (\betab^*)\right].\ \ \ \ \ \ 
\end{align*}

 We show that the last four terms of the right hand side above, each converges to $0$ asymptotically at a faster rate than the first term on the right hand side of \eqref{delta_1}.  In order to establish such result, we need to control the scale estimator. We begin by introducing a basic condition.

\vskip 5pt 

{\bf  Condition ($\boldsymbol \Gamma$)}:\label{condition_pc}
{\it  Parameters $\gammab_{(j)}^*(\betab^*)$ for all $j=1,\dots, p$ are bounded by $K_\gamma$,  
and such that 
$\|  \gammab_{(j)}^*(\betab^*)\|_0  \leq s_j$, and  $s_j \leq s$, for all $j$.  Moreover,  $ \zetab_j^*$ is sub-exponential random vector, and $K \| \zetab_{1,i}^* \|_{\psi_j}  := \tilde K< \infty$.
}
 \vskip 5pt 

The preceding  condition can be traced back to   \cite{vdgeer_buhlmann_ritov_dezeure_2014}. It restricts   the conditional  mean  of the column $W_j(\betab^*)$  to be  a function of at most $s_j$ other columns of the design matrix $W(\betab^*)$.  However,   the condition does not impose a particular
 distributional assumptions.

 The following two lemmas help  to establish $l_1$ column bound  of the corresponding precision matrix estimator. The first one provides properties of the estimator $\hat \gammab_{(j)}(\hat \betab)$ as defined  in \eqref{nodewise_lasso}. Although this estimator is obtained via Lasso-type procedure, significant challenges arise in its analysis due to dependencies  in the plug-in loss function.   The design matrix of this problem does not have independent and identically distributed rows.
 We overcome these challenges by approximating the solution to the oracle one  and without imposing any new conditioning of the design matrix.  
 
\begin{lemma}\label{cor:2}
Let  $\lambda_j = C \left( (\log p / n)^{1/2} \bigvee \left(r_n^{1/2} \bigvee t^{1/4} (\log p / n)^{1/2} \right) t^{3/4} s_j(\log p / n)^{1/2}\right)$ for a constant $C >1$ and let 
Conditions {\bf{(I)}}, {\bf(X)}, {\bf(E)}, {\bf(C)}, {\bf(CC)} and {\bf($\Gammab$)} hold.  
Then,
\begin{align}
&\left\|\hat\gammab_{(j)}(\hat\betab)-\gammab_{(j)}^*(\betab^*)\right\|_1 = \ocal_P \left(  \frac{ K_\gamma }{\phi_0^2 C_2 }  s_j
\lambda_j \right).\nonumber
\end{align}
\end{lemma}

\begin{remark}
 This Lemma implies that the  precision matrix estimator   has distinct limiting behaviors
in terms of the magnitude of  the censoring level.  In particular, Lemma \ref{cor:2} implies that $\left\|\hat\gammab_{(j)}(\hat\betab)-\gammab_{(j)}^*(\betab^*)\right\|_1$  inherits the rates available for fully observed linear models whenever $C_2 $ is bounded away from zero. Additionally, if all the data is censored, i.e., whenever $C_2$ converges to zero at a rate faster than $\lambda_j$, the estimation error will explode.  These results agree with the asymptotic results on consistency in left-censored and low-dimensional models; however, they provide additional details through the exact rates of censoring that is allowed. For example, if the initial estimator is such that $r_n $ is of the order of $s_{\betab^*} \sqrt{\log p/n}$, then the asymptotic result above matches those of linear models (see Theorem \ref{thm:clad_temp}).
\end{remark}

\begin{remark}
The choice of the tuning parameter $\lambda_j$ depends on the $l_2$ convergence rate of the initial estimator $r_n$, and the size of its estimated non-zero set. However, we observe that whenever $r_n$ is such that $r_n \leq t^{-3/8} s_j^{-1/2}$ and the sparsity of the initial estimator is such that $t s_j \sqrt{\log p/n} <1$, then the optimal choice of the tuning parameter is of the order of $\sqrt{\log p/n}$. In particular, any initial estimator that satisfies $r_n < n^{-1/4}$ is sufficient for optimal rates at estimator in a model where $t \leq n^{1/8}$ and $s_j \leq n^{1/8}$.
 \end{remark}

The next result gives a  bound on the variance of our $\hat \gammab_{(j)}(\hat \betab)$ estimator.

\begin{lemma}\label{lem:temp1}
Let  $\lambda_j = C \left( (\log p / n)^{1/2} \bigvee \left(r_n^{1/2} \bigvee t^{1/4} (\log p / n)^{1/2} \right) t^{3/4} s_j(\log p / n)^{1/2}\right)$ for a constant $C >1$ and let 
Conditions {\bf{(I)}}, {\bf(X)}, {\bf(E)}, {\bf(C)}, {\bf(CC)} and {\bf($\Gammab$)} hold.
Then, for  $j=1,\dots, p$ and $\zetab^*$ and $\hat \zetab$ defined in \eqref{eqn:zeta*}  
\begin{align*}
\left| \hat \tau_j ^2(\lambda_j)-  \tau_j^2  \right|  = \ocal_P\left( K^2 K_\gamma s_j^2 \lambda_j \right).
\end{align*}
\end{lemma}

 Next is the main result on the properties of the proposed matrix estimator $\Omegab(\hat\betab)$.

\begin{lemma}\label{lemma2}
Let the setup of Lemma \ref{lem:temp1} hold.
Let $\Omegab(\hat \betab)$ be the estimator as in \eqref{eq:theta}.
Then, 
for $\hat \tau_j^{2}$ as in \eqref{tau_hat}, we have 
$
\hat \tau_j^{-2} = \mathcal{O}_P(1).
$
Moreover, 
\begin{align*}
&\left\|\Omegab(\hat{\betab})_j-\Sigmab^{-1}(\betab^*)_j\right\|_1   = \ocal_P\left( K^2 K_\gamma^2 s_j^3 \lambda_j \right).
\end{align*}
\end{lemma}

Lemma \ref{lemma2}  provides   easy to verify sufficient conditions  for the consistency   of a class of semiparametric  estimators  of the precision matrix for censored regression models.  Even in low-dimensional setting, this result appears to be new and highlights specific rate of convergence (see Theorem \ref{thm:clad_temp} for more details). 

 The   one-step estimator $\tilde \betab$ relies crucially on the bias correction step that carefully  projects the residual  vector in the direction close to  the  most efficient score. The next result measures the uniform distance of such projection.

\begin{lemma}
  \label{lemma4}
  Let the setup of Lemma \ref{lem:temp1} hold.
  There exists a fixed
  constant $C$ (independent of $p$ and  $n$), such that the process 
  $\mathbb V_n(\deltab) = \Omegab( \deltab + \betab^*) \left[\mathbb G_n(\deltab + \betab^*) - \mathbb G_n(\betab^*) \right] $ satisfies 
\begin{align*}
&\sup_{\deltab \in \Ccal(r_n, t)}
\left\|   
  \mathbb V_n(\deltab )  \right\|_\infty 
  \leq C \rbr{
\sqrt{\frac{  (r_n \vee r_n^2K_1^2) t 
\log(n p /\delta)}{n}} 
\bigvee
\frac{  t \log(2n p /\delta)}{n} },
\end{align*}
with probability $1-\delta$ and a constant $K_1$ defined in Condition {\rm ({\bf E})}.
\end{lemma}

 Lemma \ref{lemma4} establishes  a  uniform tail probability bound for a  growing supremum  of an empirical process $\mathbb{V}_n(  \deltab)$. It   is uniform in $\deltab$  and it is growing as supremum is taken over $p$, possibly growing ($p=p(n)$) coordinates of the process.
The proof of Lemma~\ref{lemma4} is further challenged by the non-smooth components  of the process $\mathbb{V}_n(  \deltab)$ itself and the multiplicative nature of the factors within it. It proceeds in two
steps. First, we show that for a fixed  $\deltab$ the 
term $\norm{ \mathbb{V}_n(\deltab)}_\infty$ is small. In the second step,
 we devise a new  epsilon net argument to control the  non-smooth  and multiplicative terms uniformly for all
 $\deltab$ simultaneously. This is established by devising new  representations of the process that allow for small  size of the covering numbers.
In conclusion, Lemma~\ref{lemma4} establishes a uniform bound
$\| I_4\|_\infty = \ocal_P\left( (r_n^{1/2} \vee r_n K_1) t^{1/2} (\log p)^{1/2} \bigvee t \log p / n^{1/2} \right)$ in \eqref{delta_1}.

\subsection{Size of the Remainder Term}

 Size of the remainder term in \eqref{eq:short_bahadur} is controlled by the results of Lemmas 1-6 and we provide details below.
 
\begin{theorem}\label{cor:fixed:ci} 
Let  $\lambda_j = C \left( (\log p / n)^{1/2} \bigvee \left(r_n^{1/2} \bigvee t^{1/4} (\log p / n)^{1/2} \right) t^{3/4} s_j(\log p / n)^{1/2}\right)$ for a constant $C >1$ and let Conditions {\bf{(I)}}, {\bf{(X)}}, {\bf{(E)}}, {\bf{(C)}}, {\bf{(CC)}} and {\bf{($\Gammab$)}} hold. With $s_{\Omega} = \max_j s_j $,
 \begin{align*}
\| \Delta\|_\infty &= \ocal_P\left( (1 \vee d_n n^{1/2} ) s_\Omega^3 \lambda_j \bigvee d_n^2 n^{1/2}  \bigvee r_n^{1/2} t^{1/2} (\log (p \vee n))^{1/2} \bigvee t \log (p \vee n) / n^{1/2} \right). 
 \end{align*}
 \end{theorem}

We first notice that the expression above requires $t = \smallocal(n^{1/2} / \log (p \vee n))$, a condition frequently  imposed  in high-dimensional inference (see \cite{zhang_zhang_2014} for example). Then, in the case of low-dimensional problems with $s= \ocal(1)$ and $p = \ocal(1)$, we observe that whenever the initial estimator of rate $r_n$, is in the order of $n^{-1/4 - \epsilon}$, for a small constant $\epsilon >0$, then $\| \Delta\|_\infty = \Ocal_P(n^{-2\epsilon})$. In particular, for a consistent initial estimator, i.e. $r_n = \ocal(n^{-1/2})$ we obtain that $\| \Delta\|_\infty = \Ocal_P(n^{-1/2})$. For high-dimensional problems with $s$ and $p$ growing with $n$, for all initial estimators of the order $r_n$ such that $r_n = \ocal(s_{\betab^*}^{a} (\log p)^b /n^c)$ and $t =\Ocal(s_{\betab^*})$ we obtain that $\| \Delta\| _\infty = \Ocal_P(\bar s^{(a+1)/2} (\log p)^{(b+1)/2}/n^{c/2})$  whenever $\bar s^{(2a+7)/4} (\log p)^{b} /n^{c} = \smallocal(1)$, where $\bar s = t \vee s_\Omega$.
Further discussion is relegated to the comments following Theorem \ref{thm:clad_temp}.

\subsection{Asymptotic Normality of the Leading Term
}
Next, we present the result on the asymptotic normality of the leading term of the Bahadur representation \eqref{eq:short_bahadur}.
\begin{theorem}\label{normality}
Let  $\lambda_j = C \left( (\log p / n)^{1/2} \bigvee \left(r_n^{1/2} \bigvee t^{1/4} (\log p / n)^{1/2} \right) t^{3/4} s_j(\log p / n)^{1/2}\right)$ for a constant $C >1$ and let Conditions {\bf{(I)}}, {\bf{(X)}}, {\bf{(E)}}, {\bf{(C)}}, {\bf{(CC)}} and {\bf{($\Gammab$)}} hold. 
Define $U := \frac{1}{2f(0)}\Sigmab^{-1}(\betab^*)\frac{1}{\sqrt{n}}\sum_{i=1}^n\psi_i(\betab^*) = {\mbox{\scriptsize$\ocal$}}_P(\sqrt{n})$. Furthermore, assume
$$(1 \vee d_n n^{1/2} ) s_\Omega^3 \lambda_j \bigvee d_n^2 n^{1/2} \bigvee r_n^{1/2} t^{1/2} (\log p)^{1/2} \bigvee t \log p / n^{1/2} = \smallocal(1).$$
Denote $\bar s = t \vee s_\Omega$. If $f(0)$, the density of $\varepsilon$ at 0 is known, 
\begin{align*}
\left[\Omegab(\hat{\betab}) \hat\Sigmab(\hat{\betab})\Omegab (\hat{\betab})\right]_{jj}^{-\frac{1}{2}}U_j \xrightarrow[n,p, \bar s\rightarrow\infty]{d} \mathcal{N}\left(0,\frac{1}{4f(0)^2}\right).
\end{align*}
\end{theorem}
 
\begin{remark}
\rm
A few remarks are in order. Theorem \ref{normality} implies that the effects of censoring asymptotically disappear. Namely, the limiting distribution only becomes degenerate when the censoring rate asymptotically explodes, implying that  no data is fully observed. However, in all other cases the limiting distribution is fixed and does not depend on the censoring level.
\end{remark}

%

\subsection{Quality of Density Estimation}

Density estimation is a necessary step in the semiparametric inference for left-censored models. Below we present the result guaranteeing good qualities of density estimator proposed in \eqref{eq:density}.

\begin{theorem} \label{thm:f}
There exists a sequence $h_n$ such that  $h_n =  \smallocal (1)$ and $ \lim _{n \to \infty }\hat h_n / h_n =1$ and  $h_n^{-1} (d_n \vee r_n^{1/2} t^{3/4} (\log p / n)^{1/2} \vee t\log p / n) = \smallocal(1)$. Assume Conditions {\bf{(I)}}, {\bf{(X)}} and {\bf{(E)}} hold, 
then 
$$\left|\widehat f(0) - {f(0)}\right| = \smallocal_P(1) .$$
\end{theorem}
 
 Together with Theorem \ref{normality} we can provide the next result.

\begin{corollary}\label{cor1}
With the choice of density estimator as in \eqref{eq:density}, under conditions of Theorem \ref{normality} and \ref{thm:f}, the results of Theorem \ref{normality} continue to hold unchanged, i.e.,
\begin{align*}
\left[\Omegab(\hat{\betab}) \hat\Sigmab(\hat{\betab})\Omegab (\hat{\betab})\right]_{jj}^{-\frac{1}{2}}U_j \cdot 2 \widehat f(0) \xrightarrow[n,p,\bar s\rightarrow\infty]{d} \mathcal{N}\left(0,1\right).
\end{align*}
\end{corollary}

\begin{remark}
Observe that the result above is robust   in the sense that the result holds regardless of the particular distribution of the model error \eqref{model}. Condition {\bf{(E)}} only assumes minimal regularity conditions on the existence and smoothness of the density of the model errors. In the presence of censoring, our result is unique as it allows $p \gg n$, and yet it successfully estimates the variance of the estimation error.
\end{remark}

\subsection{Confidence Regions}     

Combining all the results obtained in previous sections we arrive at the main conclusions.

\begin{theorem}\label{cor:ci} 
Let  $\lambda_j = C \left( (\log p / n)^{1/2} \bigvee \left(r_n^{1/2} \bigvee t^{1/4} (\log p / n)^{1/2} \right) t^{3/4} s_j(\log p / n)^{1/2}\right)$ for a constant $C >1$ and let Conditions {\bf{(I)}}, {\bf{(X)}}, {\bf{(E)}}, {\bf{(C)}}, {\bf{(CC)}} and {\bf{($\Gammab$)}} hold.
Furthermore, assume 
\begin{align*}
(1 \vee d_n n^{1/2} ) s_\Omega^3 \lambda_j \bigvee d_n^2 n^{1/2} \bigvee r_n^{1/2} t^{1/2} (\log p)^{1/2} \bigvee t \log p / n^{1/2} = \smallocal(1),
\end{align*}
for $s_{\Omega} = \max_j s_j$. Denote $\bar s = t \vee s_\Omega$.
Let $I_{n}$ and   $a_n$  be defined in  \eqref{eq:ci} and \eqref{eq:pi}.  Then,  for all vectors $\cbb=\eb_j$ and any $j \in \{ 1,\dots, p\}$, when $n,p, \bar s \to \infty$ we have 
  \begin{align*}  
&  \PP_{\betab} \left( \cbb^\top  \betab^*  \in I_{n} \right)   = 1- 2 \alpha 
  \end{align*}
\end{theorem}

 The statements of Theorems \ref{cor:fixed:ci}  and \ref{normality} also hold in
a uniform sense, and thus the confidence intervals   are honest. In particular, the  confidence interval $I_n$ does not suffer from 
the problems arising from the non--uniqueness of $\betab^*$.  
We consider the set of parameters
$
\mathcal{B} = \{\betab \in \RR^p: \# \{j: \betab_j \neq 0\} \leq \bar s \}.
$
 Let $\PP_{\betab^*}$ be the distribution of the data under the  model \eqref{model}. Then
the following   holds.

\begin{theorem}\label{cor:uniform:ci} 
Under the setup and assumptions of Theorem \ref{cor:ci}   when $n,p,\bar s \to \infty$
  \begin{align*}  
&  \sup_{\betab \in \mathcal{B}}\PP_{\betab} \left( \cbb^\top  \betab^*  \in I_{n} \right)   = 1- 2 \alpha.
  \end{align*}
\end{theorem}

Previous results depend on a set of  high-level conditions imposed on the initial estimate. Moreover, rates depend on the initial estimator precisely and to better understand them we present here their summary when the initial estimator   $\hat \betab$ is chosen to be penalized CLAD estimator of   \cite{muller_vdgeer_2014}.
 
\begin{theorem} \label{thm:clad_temp}
Let $\hat \betab$ be defined as in \cite{muller_vdgeer_2014}
with a  choice of the tuning parameter $\lambda =A_2 K \left(\sqrt{ {2\log(2p)}/{n}} + \sqrt{ {\log p}/{n}} \right)$ for a constant $A_2>16$ and independent of $n$ and $p$.
 Assume that $\bar s^2 (\log p)^{1/4} / n^{1/4} =\smallocal(1)$, for $\bar s = s_{\betab^*} \vee s_{\Omega}$ with   $s_{\Omega} = \max_j s_j $. 

(i) Suppose that conditions {\bf{(X)}}, {\bf{(E)}}, {\bf{(C)}}, {\bf{(CC)}} and {\bf{($\Gammab$)}} hold. Moreover, let  $\lambda_j = C \sqrt{\log p / n} $ for a constant $C >1$.
Then
\begin{align}\label{eq:gamma}
\left\|\hat\gammab_{(j)}(\hat\betab)-\gammab_{(j)}^*(\betab^*)\right\|_1 = \ocal_P \left(  \frac{ K_\gamma }{\phi_0^2 C_2 }   s_j \sqrt{\log p/n}\right).
\end{align}
(ii)  For   $j=1,\dots, p$ and $\zetab^*$ and $\hat \zetab$ defined in \eqref{eqn:zeta*}  
\begin{align*}
\left| \hat \tau_j^2(\lambda_j)- \tau_j^2  \right| =  \mathcal{O}_P \left( K^2 K_\gamma s_j^2 \sqrt{\log (p \vee n)/n}   \right).
\end{align*}

(iii) Let $\Omegab (\hat\betab) $ defined in \eqref{eq:theta}.
Then, 
for $\hat \tau_j^{2}$ as in \eqref{tau_hat}, we have 
$
\hat \tau_j^{-2} = \mathcal{O}_P(1).
$
Moreover, 
\begin{align*}
\left\|\Omegab(\hat{\betab})_j-\Sigmab^{-1}(\betab^*)_j\right\|_1 = \ocal_P \left(K^2 K_\gamma^2 s_j^3 \sqrt{\log (p \vee n)/n} \right)
\end{align*}

(iv) Let $\tilde \betab$ be defined as in  \eqref{beta_os} with 
$\Omegab (\hat\betab) $ defined in \eqref{eq:theta},  $\hat \Sigmab(\hat \betab)$ defined in \eqref{eq:sigma} 
and $\hat f(0)$ as defined in \eqref{eq:density}. Then, for $\bar s = s_{\betab^*} \vee s_{\Omega}$ with   $s_{\Omega} = \max_j s_j $, the size of the residual term in \eqref{eq:short_bahadur} is 
  \begin{align*}
  \| \Delta\|_\infty =
 \Ocal_P\left( \frac{\bar s^4 \log (p \vee n)}{n^{1/2}} \bigvee \frac{s_{\betab^*}^{3/4} (\log (p \vee n))^{3/4}}{n^{1/4}} \right).
 \end{align*}
 
(v) Assume that  $\bar s^{3/4} (\log p)^{3/4} / n^{1/4} =\smallocal(1)$,   for $\bar s = s_{\betab^*}  \vee s_{\Omega}$ with   $s_{\Omega} = \max_j s_j$.
Let $I_{n}$ and   $a_n$  be defined in  \eqref{eq:ci} and \eqref{eq:pi}.  Then,  for all vectors $\cbb=\eb_j$ and any $j \in \{ 1,\dots, p\}$, when $\bar s, n,p \to \infty$ we have 
  \begin{align*}  
&  \PP_{\betab} \left( \cbb^\top  \betab^*  \in I_{n} \right)   = 1- 2 \alpha .
  \end{align*}
  
\end{theorem}

\begin{remark}
\rm 
Result (i) suggests that the rates of estimation match those of simple linear model as long as proportion of censored data is not equal 1.  In that sense, our results are also efficient. Moreover, result (ii) implies that the rates of estimation of the variance  are slower by a factor of $s_j^{3/2}$ compared to the least squares method. This is also apparent in the result (iii) where the rate of convergence of the precision matrix is slower by a factor of $s_j^{5/2}$, due to the non-standard dependency issues in the plug-in Lasso estimator \eqref{cor:2}.

Lastly, results (iv) and (v)  suggest that 
the confidence interval $I_n$ is  asymptotically valid and that   the coverage errors  are of the order  of 
$ \mathcal{O}\left(   { s_{\betab^*}^{3/4} \rbr{\log p}^{3/4}   }/{n^{1/4} } \right) $ whenever $\bar s^{3/4} (\log p)^{3/4} / n^{1/4} =\smallocal(1)$. Classical results on inference for left-censored data, with $p \ll n$,  only  imply that the error rates of the confidence interval is $\smallocal_P(1) $; instead, we obtain a precise characterization of the size of the residual term.  Moreover, with $p \gg n$ the rates above match the optimal rates of inference for the absolute deviation loss (see e.g. \cite{Belloni2013Uniform}), indicating that our estimator is asymptotically efficient  in the sense that the censoring asymptotically disappears. However,  we impose  slightly stronger dimensionality restrictions as for fully observed data $\bar s \log p/ \sqrt{n}$ is a sufficient condition. The additional condition $\bar s^{3/4} (\log p)^{3/4} / n^{1/4} =\smallocal(1)$  can be thought of as a penalty to pay for being adaptive to left-censoring.   This implies that a larger sample size needs to be employed for the results to be valid. However, this is not unexpected as censoring typically reduces the  effective sample size. 

 \end{remark}


\section{Mallow's, Schweppe's and Hill-Ryan's Estimators for High-Dimensional  Left-Censored Models}     
Statistical models are seldom believed to be complete
descriptions of how real data are generated; rather, the model is an approximation
that is useful, if it captures essential features of the data. Good robust
methods perform well even if the data deviates from the theoretical distributional assumptions.
The best known example of this behavior is the outlier resistance and transformation
invariance of the median.
Several authors have proposed one-step and k-step estimators to combine
local and global stability, as well as a degree of efficiency under the target
linear model \citep{bickel_1975}.
There have been considerable challenges
in developing good robust methods for more general problems.
To the best of our knowledge, there is no prior work that discusses robust one-step estimators for the case of  left-censored models (for either  high or low dimensions).

     We propose here a family of doubly robust estimators that stabilize estimation in the presence of ``unusual'' design or model error distributions.  Observe that  \eqref{model} rarely follows  distribution with light tail. Namely, model \eqref{model} can be reparametrized as $y_i =z_i(\betab^*) \betab^* + \xi_i$ where, $z_i(\betab^*)=x_i  \ind\{x_i\betab^* + \varepsilon_i \geq 0\} $ and $\xi_i =\varepsilon_i \ind\{x_i\betab^* + \varepsilon_i \geq 0\}$. 
Hence $\xi_i$ will rarely follow light tailed distribution and it is in this regard very important to design estimators that are robust.   We introduce Mallow's, Schweppe's and Hill-Ryan's estimators for left-censored models.
     \subsection{Smoothed Robust Estimating Equations (SREE)}

 In this section we propose  a  doubly robust population system of equations
 \begin{equation}\label{eq:robust}
 \EE[\Psi^r(\betab)]=0
 \end{equation}
 with $\Psi^r=n^{-1} \sum_{i=1}^n \psi_i^r(\betab)$
  and
\begin{equation}\label{score}
 \psi_i^r(\betab) =   - n^{-1} \sum_{i=1}^{n}  {q_{i\mbox{\scriptsize }}}  w_i^\top( \betab) \ \psi\biggl( v_i \bigl(y_i-\max\{0,x_i \betab\}\bigl)  \biggl)  ,
\end{equation}
where $\psi$ is an odd, nondecreasing and bounded function. Throughout we assume that the function $\psi$ either has finitely many jumps or is differentiable with bounded first derivative.
Notice that when $q_i = 1$ and $v_i = 1$, with $\psi$ being the sign function, we have $\psi_i^r = \psi_i$ of previous section. 
Moreover, observe that for the weight functions
 $q_i =q(x_i)$ and $v_i=v(x_i)$, both functions of $\RR^p \rightarrow \RR^+$,
 the true parameter vector $\betab^*$ satisfies the robust population system of equations above.
 Appropriate  weight functions $q$ and $v$ are chosen for particular efficiency considerations. Points which have high leverage are considered ``dangerous'', and should be downweighted by the appropriate choice of the weights $v_i$. Additionally, if the design has ``unusual'' points, the weights $q_i$ serve to downweight their effect in the final estimator.

      We augment the system above similarly as before  and consider     
       the system of equations 
     \begin{equation}
     \EE[\Psi^r(\betab^*)] + {\boldsymbol {\Upsilon}}^r [\betab -\betab^*] = 0,
     \end{equation}
     for a suitable choice of the robust matrix  ${\boldsymbol {\Upsilon}}^r \in \RR^{p \times p}$.
     Ideally, most efficient estimation can be achieved when the matrix   ${\boldsymbol {\Upsilon}}^r $ is close to the influence function of the robust equations \eqref{eq:robust}.

     To avoid difficulties with non-smoothness of $\psi$, we propose to work with a matrix  ${\boldsymbol {\Upsilon}}^r$  that is smooth enough and is robust simultaneously. To that end, 
    observe $\Psi^r(\betab^*) = \Phi ^r(\betab^*, \varepsilon)$ for a suitable function $\Phi ^r= n^{-1} \sum_{i=1}^n \phi_i^r$ and $\phi_i^r: \RR^p \times \RR \to \RR $.  
 We consider a smoothed version of the Hessian matrix and work with     ${\boldsymbol {\Upsilon}}^r = {\boldsymbol {\Upsilon}}^r(\betab^*)$ for     \[
     {\boldsymbol {\Upsilon}}^r(\betab^*)=\EE_X \left[ \nabla_{\betab^*} \int_{-\infty} ^ \infty \Phi^r(\betab^*,\varepsilon)  f_\varepsilon(x)dx \right]
     \]
     where $f_\varepsilon$ denotes the density of the model error \eqref{model}.            
     To infer the parameter $\betab^*$, we adapt a one-step approach in solving the empirical counterpart of the population equations above. We name the empirical equations as   {\it Smoothed Robust Estimating Equations}  or SREE in short.   For a preliminary estimate    we solve an approximation of the robust system of equations above and search for the $\betab$ that solves
      $
     \Psi^r(\hat \betab ) +     {\boldsymbol {\Upsilon}}^r(\hat \betab) (\betab - \hat \betab) = 0.
     $
           
     The particular form of the matrix ${\boldsymbol {\Upsilon}}^r(\betab^*)$ depends on the choice of the weight functions $q$ and $v$ and the function $\psi$. In particular, for the left-censored model \eqref{model}
     \begin{align}\label{psi_subgradient_1}
\nabla_{\betab^*} \EE_\varepsilon[\Psi^r(\betab^*) ] 
&=  n^{-1} \sum_{i=1}^{n} q_i \nabla_{\betab^*} \EE_\varepsilon \left[\psi \left( v_i( y_i- \max\{0,x_i\betab^*\})\right)\right]\ 
\end{align}
leading to the following form
\[
{\boldsymbol {\Upsilon}}^r(\betab^*) =   \EE_X \left[ n^{-1} \sum_{i=1}^{n}{q_i}  v_i  \psi' ( v_i  \varepsilon_i  )x _i ^\top  w_i(\betab^*) \right] 
 \]
 whenever the function $\psi$ is differentiable. Here, we denote $\psi'(xy)$ as $\partial \psi(x y) / \partial y$. In case of non-smooth $\psi$, $\psi '$ should be interpreted as $g'$ for  $g(\varepsilon_i) =  \EE_{\varepsilon}[\psi(v_i \varepsilon_i)]$. For example, if $\psi=\mbox{sign}$ then $g ( \varepsilon_i)$ is equal to $1 - 2P(\varepsilon_i \leq 0)$ and $g' (\varepsilon_i) =  -2 f_{\varepsilon_i}(0)$. 
 

\subsection{Left-censored Mallow's, Hill-Ryan's and  Schweppe's estimator}

Here we provide specific definitions of new robust one-step estimates. We begin by defining a robust estimate of the precision matrix 
  i.e., $\{{\boldsymbol {\Upsilon}}^{r} \}^{-1}(\betab^*) $. We design a robust estimator that preserves the ``downweight'' functions $q$ and $v$ as to stabilize the estimation in the presence of contaminated observations. 
For further analysis, it is useful to define the matrix $\tilde W(\betab) = Q^{1/2} W(\betab)$ and 
$$Q = \mbox{diag}(\mathbf q \circ \db) \in \RR^{n \times n},$$ $\mathbf q  \in \RR^n$ with 
$\mathbf q = \left[
q(x_1), q(x_2), \cdots, q(x_n)
\right]^\top$
and $\mathbf d \in \RR^n$ with 
$$\mathbf d = \begin{bmatrix}
\psi' ( v_1 \hat \varepsilon_1 ), & \psi' ( v_2 \hat \varepsilon_2  ), & \cdots, & \psi' ( v_n  \hat \varepsilon_n  )
\end{bmatrix}^\top$$  for  $\hat \varepsilon_i = y_i-\max\{0,x_i\hat \betab \}$.  When function $\psi$ does not have first derivative, we replace $\psi' ( v_i \hat \varepsilon_i )$ with $n^{-1} \sum_{i=1}^n [\EE \psi ( v_i \hat \varepsilon_i )]'$. With this notation, we have $\tilde W_{j}( {\betab}^*) = Q^{1/2} A(\betab^*)X_{j}$ and ${\boldsymbol {\Upsilon}}^r(\betab^*)= n^{-1}\mathbb{E} \left[ \tilde W(\betab^*)^\top \tilde W(\betab^*) \right]$ takes the form of a weighted covariance matrix. Hence, to estimate the inverse $\{{\boldsymbol {\Upsilon}}^r \}^{-1}(\betab^*)$, we project columns one onto the space spanned by the remaining columns.
For $j = 1, \dots, p$, we define the vector $\tilde{\thetab}_{(j)}(\betab)$ as follows,
\begin{align} \label{step2_model_r}
\tilde{\thetab}_{(j)}(\betab)&= \underset{\thetab \in \RR^{p-1}}{\argmin} \ \EE  \left\| \tilde W_j(\betab) - \tilde W_{-j}(\betab)\thetab \right\|_2^2 /n.
\end{align}
Also, we assume the vector $\tilde{\thetab}_{(j)}(\betab^*)$ is sparse with $\tilde s_j := \| \tilde{\thetab}_{(j)} (\betab^*) \|_0 \leq s_{\Omega}$.  
Thus, we propose the following  as a robust estimate of the scale
\begin{equation}\label{eq:theta_1}
\tilde\Omega_{jj}(\hat{\betab}) =\tilde{\mbox{\scriptsize$\mathcal{J}$}}_j^{-2} ,
 \qquad
\tilde\Omega_{j,-j}(\hat{\betab}) = - \tilde{\mbox{\scriptsize$\mathcal{J}$}}_j^{-2} \tilde{\thetab}_{(j)}(\hat\betab). 
\end{equation}
with 
\begin{align*} 
\tilde{\thetab}_{(j)}(\hat{\betab}) &= \underset{\thetab \in \RR^{p-1}}{\argmin} \left\{n^{-1}\left\| \tilde W_j(\hat{\betab}) - \tilde W_{-j}(\hat{\betab})\thetab \right\|_2^2+2\lambda_j\|\thetab \|_{1} \right\}.
\end{align*}
and the normalizing factor 
\begin{align} 
\tilde{\mbox{\scriptsize$\mathcal{J}$}}_j^2 &=
 n^{-1}\left\|  \tilde W_j(\hat{\betab}) - \tilde W_{-j}(\hat{\betab})\tilde{\thetab}_{(j)}(\hat{\betab}) \right\|_2^2 + \lambda_j\|\tilde{\thetab}_{(j)}(\hat{\betab}) \|_{1} .\nonumber
\end{align}
\begin{remark}
 Estimator \eqref{eq:theta_1} is a high-dimensional extension of Hampel's ideas of  approximating the inverse of the Hessian matrix in a robust way, by allowing data specific weights to trim down the effects of the outliers. Such weights can be  stabilizing estimation in the presence of high proportion of censoring.
   \cite{H77}   compared the efficiency of the
Mallow's and Schweppe's estimators  to several others and found that they dominate in the case of linear models in low-dimensions. 
  \end{remark}
  
Lastly, we arrive at a class of  doubly robust one-step estimators,
\begin{align}\label{beta_r}
\check {\betab} = \hat{\betab} + {\tilde \Omegab(\hat{\betab})} \left(     n^{-1} \sum_{i=1}^{n}  {q_{i\mbox{\scriptsize }}}  w_i^\top(\hat\betab) \ \psi\biggl( v_i \bigl(y_i-\max\{0,x_i\hat \betab\}\bigl)  \biggl) \right).
\end{align}

We  propose a one-step left-censored Mallow's estimator for left-censored high-dimensional regression by setting the weights to be 
$v_i=1$, and
$$q_i=\min \left\{ 1,  {b^{\alpha/2}} \left({  \left(w_{i, \hat S}(\hat \betab) - \bar w_{\hat S} (\hat \betab) \right)^\top  \tilde \Omegab _{\hat S, \hat S}(\hat \betab) \left(w_{i, \hat S}(\hat \betab) - \bar w_{\hat S} (\hat \betab) \right) } \right)^{-\alpha/2} \right\},$$ for  constants $b>0$ and $\alpha\geq 1$, 
and with 
$$\bar w_{  \hat S}(\hat \betab) = n^{-1} \sum_{i=1}^n w_{i, \hat S}(\hat \betab)$$ and $\hat S =\{j: \hat \betab_j \neq 0\}$. 
Extending the work of \cite{CH93}, it is easy to see that    Mallow's 
one-step estimator  with $\alpha=1$  and $b= \chi^2_{\hat s,0.95}$ quantile of chi-squared distribution with $\hat s =|\hat S|$ improves  a breakdown point of the initial estimator to nearly   $0.5$, by providing local stability of the precision matrix estimate.

Similarly, the  one-step  left-censored Hill-Ryan estimator  is defined with $v_i =q_i$ and the one-step left-censored Schweppe's estimator with 
$$v_i =1/  q_i , \qquad \mbox{and} \qquad q_i = 1/ \left\| \tilde \Omegab_{\hat S, \hat S}(\hat \betab) (w_{i, \hat S}(\hat \betab) - \bar w_{\hat S}(\hat \betab))\right\|_2.$$

\subsection{Theoretical Results}
      
      Similar to the concise version of Bahadur representation presented in \eqref{eq:short_bahadur} for the standard one-step estimator with $q_i = 1$ and $v_i = 1$, we also have the expression for doubly robust estimator, 
 \begin{align}\label{eq:bahadur2}
\sqrt{n}&\left(\breve{\betab}-\betab^*\right)   \nonumber
\\
&= \frac{1}{2f(0)}\{\Sigmab^{\mbox{\scriptsize r}}\}^{-1}(\betab^*)\frac{1}{\sqrt{n}}\sum_{i=1}^n q_i \psi\left(  v_i (y_i-\max\{0,x_i {\betab^*}\} ) \right)(w_i({\betab^*}))^\top + \Delta^{\mbox{r}}.
\end{align}

Next, we show that the leading component has asymptotically normal distribution and that the residual term is of smaller order. For simplicity of presentation we present results below with an initial estimator being  penalized CLAD  estimator with the choice of tuning parameter as presented in Theorem \ref{thm:clad_temp}. We  introduce the following condition.

\vskip 5pt

{\bf  Condition (r$\boldsymbol \Gamma$)}:\label{condition_rpc}
{\it  Parameters $\thetab_{(j)}^*(\betab^*) $ for all $j=1,\dots, p$ are bounded, and such that 
$\left|\{k : \theta_{(j),k}^*(\betab^*) \neq 0\}\right| \leq \tilde s_j$ for some $\tilde s_j \leq n$.  Moreover,  $\tilde \eta_j$  are sub-exponential random vectors. Let $q_i$ and $v_i$ be functions such that $\max_i |q_i| \leq M_1$ and $\max_i |v_i| \leq M_2$ for positive constants $M_1$ and $M_2$ and $\EE[\psi(\varepsilon_i v_i)]=0$.  Moreover, let $\psi$ be such that $\psi(z) < \infty$ and $0<\psi'(z) < \infty$. }

\begin{theorem}\label{normality_r}
Assume that  $\bar s^2 \log^{1/4}(p) / n^{1/4} =\smallocal(1)$  , with $\bar s=s_{\betab^*} \vee \tilde s_{\Omega}$ and $\tilde s_{\Omega} = \max_j \tilde s_j$.  Define $U^{\mbox{r}} := \frac{1}{2f(0)}\{\Sigmab^{\mbox{\scriptsize r}}\}^{-1}(\betab^*)\frac{1}{\sqrt{n}}\sum_{i=1}^n q_i \psi\left(  v_i (y_i-\max\{0,x_i {\betab^*}\} ) \right)(w_i({\betab^*}))^\top$.  Let  Conditions {\bf{(X)}}, {\bf{(C)}}, {\bf{(CC)}}, {\bf (r}$\boldsymbol \Gamma${\bf )} and {\bf{(E)}} hold and let    $\lambda_j = C \sqrt{\log p / n}$ for a constant $C >1$. Then,
\begin{align*}
\left[\tilde\Omegab(\hat{\betab}) \hat  { {\boldsymbol {\Upsilon}}^r} (\hat \betab)  \tilde\Omegab (\hat{\betab})\right]_{jj}^{-\frac{1}{2}}U_{j}^{\mbox{r}} \xrightarrow[n,p,s_{\betab^*}\rightarrow\infty]{d} \mathcal{N}\left(0, 1\right).
\end{align*}
\end{theorem}

For the residual term we obtain the following statement.

\begin{theorem}\label{cor:fixed:ci:b} 
Let  Conditions {\bf{(X)}}, {\bf{(C)}}, {\bf{(CC)}}, {\bf (r}$\boldsymbol \Gamma${\bf )} and {\bf{(E)}} hold and let    $\lambda_j = C \sqrt{\log p / n}$ for a constant $C >1$.
Assume that  $\bar s^2 \log^{1/4}(p) / n^{1/4} =\smallocal(1)$, for $\bar s = s_{\betab^*} \vee s_{\tilde \Omega}$ with $s_{\tilde \Omega} = \max_j \tilde s_j  $. 
  Let $q_i$ and $v_i$ be functions such that $\max_i |q_i| \leq M_1$ and $\max_i |v_i| \leq M_2$ for positive constants $M_1$ and $M_2$.  Then, 
   \begin{align*}
 \|\Delta^{\mbox{r}} \|_\infty =
 \Ocal_P\left( \frac{\bar s^4 \log (p \vee n)}{n^{1/2}} \bigvee \frac{s_{\betab^*}^{3/4} (\log (p \vee n))^{3/4}}{n^{1/4}}  \right).
 \end{align*}
\end{theorem}

 \begin{remark}
 \rm
The estimation procedure described above is based on the initial estimator $\hat\betab$  taken to be penalized CLAD. 
However, it is possible to show that a large family of  sparsity encouraging estimator suffices. In particular, suppose that the initial estimator $\bar \betab$ is such that 
$\| \bar \betab - \betab^*\|_1 \leq \delta_n$ and let for simplicity $s_{\betab^*}=s$. Then results of Theorem \ref{cor:fixed:ci:b} extend to hold for  the confidence interval  defined as 
$
\bar{I}_n =  ( \cbb^\top \tilde\betab  -a_n,\cbb^\top  \tilde\betab +a_n  )
$
with  $a_n$ as in \eqref{eq:pi_1}. In particular, the error rates are of the order of  
\[
  (1 \vee \delta_n \sqrt{ n} ) s_\Omega^3 \lambda_j  + \delta_n^2 \sqrt{ n}  +  \delta_n^{1/2} \sqrt{s }  \sqrt{\log (p \vee n)}  +  s  \log (p \vee n) / \sqrt{n}   
\]
 
  When $s=\ocal(1)$ and $s_j=\ocal(1)$,  and all $\sqrt{n} \lambda_j = \ocal(1)$, previous result implies   that the initial estimator need only to converge at a rate of  $\smallocal(n^{-1/4 - \epsilon})$ for a small $\epsilon >0$.
 \end{remark}  
  
With the   results above, we can now construct a $(1-2\alpha) 100\%$ confidence interval for  $\cbb^\top \betab $ of the form 
\begin{equation}\label{eq:ci_1}
\mbox{ I}_{n}^{\mbox{\small r}}= \biggl( \cbb^\top \breve\betab  - \breve a_n,\cbb^\top  \breve\betab +  \breve a_n \biggl),
\end{equation}
where $\breve \betab $ is defined in   \eqref{beta_r}, $\mathbf{c} = \mathbf e_j$ for some $j \in \{1,2,\dots, p\}$
\begin{equation}\label{eq:pi_1}
 {\breve{a}_n} =  z_{\alpha} \sqrt{  \cbb^\top  \tilde \Omegab (\hat\betab)    \hat  { {\boldsymbol {\Upsilon}}^r} (\hat \betab) \tilde \Omegab(\hat\betab) \cbb}  \Bigl/ \Bigl.\sqrt{n}
\end{equation}
and
\[
  \hat  { {\boldsymbol {\Upsilon}}^r} (\hat \betab)=   n^{-1} \sum_{i=1}^{n}{q_i}  v_i  \psi' ( v_i   (y_i - x_i^\top \hat \betab) )x _i ^\top  w_i(\hat \betab)  .
 \] 
 
   \begin{remark}
 \rm
 Constants $M_1$ and $M_2$  change with  a   choice of the robust estimator. For the Mallow's  and Hill-Ryan's, by Lemma \ref{lemma2},
   \begin{align*}
  \left(w_{i,\hat S}(\hat \betab) - \bar w_{\hat S} (\hat \betab)\right)^\top  \tilde \Omegab_{\hat S, \hat S}(\hat \betab) \left(w_{i,\hat S}(\hat \betab) - \bar w_{\hat S} (\hat \betab) \right)
 >  C \left\|w_{i,\hat S}(\hat \betab) - \bar w_{\hat S} (\hat \betab) \right\|_2^2 \geq 0.
 \end{align*}
 Thus, the  coverage probability of Mallow's  and Hill-Ryan's estimator is the same as that of the M-estimator. 
 
 However, the  coverage of the Schweppe's estimator   is slightly slower, as  result of Lemma \ref{lemma0} and Lemma  \ref{lemma2} imply
   \begin{align*}
&   \left(w_{i,\hat S}(\hat \betab) - \bar w_{\hat S} (\hat \betab)\right)^\top  \tilde \Omegab_{\hat S, \hat S}(\hat \betab) \left(w_{i,\hat S}(\hat \betab) - \bar w_{\hat S} (\hat \betab) \right)\\
& \qquad \leq    \left(w_{i,\hat S}(\hat \betab) - \bar w_{\hat S} (\hat \betab)\right)^\top    \Sigmab^{-1}(  \betab^*) \left(w_{i,\hat S}(\hat \betab) - \bar w_{\hat S} (\hat \betab) \right)+ \smallocal_P(1)  
\\ & \qquad  
\leq \left\| x_{i, \hat S} \right\|_2^2 \  / \lambda_{\min } \left(  \Sigmab ( \betab^*)  \right) = \Ocal(s_{\betab^*}).
 \end{align*}
 Together with Theorem \ref{cor:fixed:ci}, part (b), we observe now a rate that is slower by a factor of $s_{\betab^*}$, i.e., the leading term is  of the order of $\Ocal\left(  {s_{\betab^*}^{7/4} (\log (p \vee n))^{3/4} }{n^{-1/4}}\right)$.
 \end{remark}

 The statements of Theorem \ref{cor:fixed:ci:b} also hold in
a uniform sense.   
\begin{theorem}\label{cor:uniform:ci:b} 
  Under Conditions of Theorems \ref{normality_r} and \ref{cor:fixed:ci:b},    we have  
for  Mallow's and Hill-Ryan's estimator
  \begin{align*}  
     \|\Delta^{\mbox{r}} \|_\infty = 
 \ocal_P\left(  \frac{s_{\betab^*}^{3/4} (\log (p \vee n))^{3/4} }{n^{1/4}}\bigvee \frac{\bar s^4\sqrt{\log (p\vee n)} }{  n^{1/2}}  \right),
  \end{align*}
  whereas for the  Schweppe's estimator  
  \begin{align*}  
  \|\Delta^{\mbox{r}} \|_\infty =  \ocal_P\left(  \frac{s_{\betab^*}^{7/4} (\log (p \vee n))^{3/4} }{n^{1/4}}\bigvee \frac{\bar s^6\sqrt{\log (p\vee n)} }{  n^{1/2}}   \right).
  \end{align*}
\end{theorem}
%
%
%

 \begin{remark}
 \rm 
This result implies that the residual term sizes depend on  the  type of weight functions  chosen.  Due to the particular left-censoring, the ideal weights  measuring concentration in the error or design depend on the unknown censoring. Hence, we approximate these ideal weights with a plug-in estimators, and therefore obtain rates of convergence that are slightly slower than those of non-robust estimators. This implies that the robust confidence intervals require larger sample size to achieve the nominal level.
 \end{remark}

\begin{corollary}\label{cor:fixed:ci:b:1} 
 Under Conditions of Theorem \ref{normality_r} and \ref{cor:fixed:ci:b},   for all vectors $\cbb=\eb_j$ and any $j \in \{ 1,\dots, p\}$, when $\bar s, n,p \to \infty$  and all $\alpha \in (0,1)$ we have that
   (i) whenever the interval is constructed using Mallow's or Hill-Ryan's estimator and  ${s_{\betab^*}^{3/4} (\log (p \vee n))^{3/4} }/{n^{1/4}}=o(1)$, the respective confidence intervals have asymptotic coverage $1-\alpha$;
 (ii) 
   whenever the interval is constructed using Schweppe's estimator and   ${s_{\betab^*}^{7/4} (\log (p \vee n))^{3/4} }/{n^{1/4}}=o(1)$, the 
    respective confidence intervals have asymptotic coverage of $1-\alpha$.
 \end{corollary}

     \section{Numerical Results}
     
In this section, we present a number of numerical experiments from both high-dimensional, $p \gg n$, and  low-dimensional, $p \ll n$, simulated  settings.

We  implemented  the proposed estimator  in a number of different model settings. Specifically, we vary the following parameters of the model. The number of observations, $n$, is  taken to be  $300$, while  $p$, the number of parameters, is taken to be $40$ or $400$. The error of the model, $\varepsilon$, is generated from a number of  distributions including: standard normal, Student's $t$ with $4$ degrees of freedom, Beta distribution with parameters $(2,3)$ and Weibull distribution with parameters $(1/2,1/5)$. In the case of the non-zero mean distributions, we center the observations before generating the model. 
The parameter $s_{\betab^*}$, the sparsity of $\betab^*$, $\#\{j:\betab^*_j=0\}$, is taken to be 3, with all signal parameters taken to be $1$ and located as the first three coordinates.  The $n\times p$ design matrix, $X$,  is generated from a multivariate Normal distribution $\mathcal{N}\left(\mu,\Sigmab\right)$. The mean $\mu$ is chosen to be vector of zero, and the censoring level $c$ is chosen to fix censoring proportion at $25\%$. The covariance matrix, $\Sigmab$,  of the distribution that $X$  follows, is taken to be the identity matrix or the Toeplitz matrix  such that $\Sigmab_{ij} =\rho^{|i-j|}$ for $\rho=0.4$. In each case, we generated 100 samples from one of the settings described above  and for each sample we calculated the 95\%  confidence interval obtained by using the algorithm described in Steps 1-4 below. We also note that the optimization problem required to obtain the CLAD estimator is not convex. Linear programming techniques used to obtain the solution is described in the following,
\begin{align*}
& \underset{\substack{\ub^+, \ub^- \geq 0 \\ \vb^+, \vb^- \geq 0 \\ \betab^+, \betab^- \geq 0}}{\text{minimize}}
& & \left\{ n^{-1} \sum_{i=1}^n \left(\ub_i^+ + \ub_i^-\right) + \lambda \sum_{j=1}^p \left(\betab_j^+ + \betab_j^-\right)\right\} \\
& \text{subject to} 
& & \ub_i^+ - \ub_i^- = y_i - \vb_i^+, \text{ for $1 \leq i \leq n$} \\
&&& \vb_i^+ - \vb_i^- = \sum_{j=1}^p X_{ij}\left(\betab_j^+ - \betab_j^-\right), \text{ for $1 \leq i \leq n$}.
\end{align*}

\begin{enumerate}
\item The penalization factor $\lambda$ is chosen by  the one-standard deviation rule of the cross validation,
$\hat \lambda = \arg\min_{ \lambda \in \{ \lambda^1,\dots, \lambda^m\}} \mbox{CV} (\lambda).$
 We move $\lambda$ in the direction of  decreasing regularization until it
ceases to be true that
$\mbox{CV} (\lambda) \leq  \mbox{CV} (\hat \lambda)+ \mbox{SE}(\hat \lambda)$.
Standard error for the cross-validation curve, $\mbox{SE}(\hat \lambda)$, is defined as a sample standard error of the  $K$ fold cross-validation statistics $\mbox{CV}_1(\lambda), \dots, \mbox{CV}_K(\lambda)$. They are calibrated using the censored LAD loss as 
\[
\mbox{CV}_k(\lambda) = n_k^{-1} \sum_{i \in F_k} \left|y_i - \max\{ 0, x_i \hat \betab^{-k} (\lambda)\} \right|,
\]
with $\hat \betab^{-k} (\lambda)$ denoting the CLAD estimator computed on all but the $k$-th fold of the data. 

\item The tuning parameter $\lambda_j$ in each penalized $l_2$ regression, is chosen by the one standard deviation rule (as described above). In more details, $\lambda_j$ is 
 in the direction of  decreasing regularization until it
ceases to be true that
$\mbox{CV}^j (\lambda_j) \leq  \mbox{CV}^j (\hat \lambda_j)+ \mbox{SE}^j(\hat \lambda_j)$
for $\hat \lambda_j$ as the cross-validation parameter value. The cross-validation statistic
is here defined as 
\[
\mbox{CV}_k^j(\lambda) = n_k^{-1} \sum_{i \in F_k} \left(W_{ij}(\hat \betab) - W_{ij}(\hat\betab) \hat \gamma^{-k}_{(j)} (\lambda_j)\right)^2,
\]
with $\hat \gamma^{-k}_j (\lambda_j)$ denoting estimators \eqref{nodewise_lasso} computed on all but the $k$-th fold of the data. This choice leads to the conservative confidence intervals with wider than the optimal length.  Theoretically guided optimal choice is highly complicated and depends on both design distribution and censoring level concurrently. Nevertheless,  we show that  one-standard deviation choice is  very reasonable. 
  
  \item Whenever the density of the error term is unknown, we estimate $f(0)$,  using  the proposed  estimator \eqref{eq:density}, with a constant  $c=10$.
We compute the above estimator by splitting the sample into two parts: the first sample is used for computing $\hat \betab$ and $\tilde \betab$ and the other sample is  to compute the estimate $\hat f(0)$.   Optimal value of $h$ is of special independent interest; however, it is not the main objective  of this work.
  
  \item Obtain $\tilde{\betab}$ by plugging $\Omegab(\hat{\betab})$ and $\hat f(0)$ into \eqref{beta_os} with $\lambda$ and $\lambda_j$ as specified in the steps above.
\end{enumerate}

\subsection{Finite Sample Comparisons}\label{Finite_Sample_Comparisons}

 
The summary of the results is presented across dimensionality of the parameter vector. The
{\it Low-Dimensional Regime} are summarized in Table \ref{tab:lowdim} and Figures \ref{fig:1} and \ref{fig:2},
whereas the {\it High-Dimensional Regime} are summarized in Table \ref{tab:highdim} and  Figures \ref{fig:3} and \ref{fig:4}.
   We report average coverage probability across the signal and noise variables independently, as the signal variables are more difficult to cover when compared to the noise variables.
   
We consider a number of challenging settings. Specifically, the censoring proportion is kept relatively high at $25\%$, and our parameter space is large with $p=400$ and $n=300$. In addition, we consider the case of error distribution being Student with $4$ degrees of freedom, which is notoriously difficult to deal with in left-censored problems. In Figures \ref{fig:3} and \ref{fig:4}, we illustrate boxplots of the width of the $95\%$ level confidence intervals across  the simulated repetitions. We showcase  the signal  and the noise variables separately. Table \ref{tab:lowdim} and \ref{tab:highdim} summarize average coverage probabilities of the constructed $95\%$ level confidence intervals for both low-dimensional and high-dimensional regime respectively. For the four error distributions, the observed coverage probabilities are approximately
the same. However, we observe that our method is not insensitive to the heavy-tailed distributions (Student's $t_4$), due to the large bias of the initial estimator. This bias results in larger interval widths especially in the signal variables. Nevertheless, the coverage probability is not affected.

  The biggest advantage of our method is most clearly seen   when the errors are
asymmetric   (Beta and Weibull). In this case, our method has smaller interval width and smaller variance.  Symmetric distributions are very difficult to handle in left-censored models. However, when errors were symmetric (Normal), the coverage probabilities were extremely close to the nominal ones. The above  cases evidently show  that  our method is robust to asymmetric distributions   and does not lose efficiency when the errors are symmetric.


\begin{figure}[h]
\centering
\caption{Comparative boxplots of the average Interval length of Signal (left) and Noise (right) variables. Case of  $p \ll n$ and Toeplitz Design with $\rho=0.4$. }
\label{fig:1}
\includegraphics[width=7cm,height=5cm,keepaspectratio]{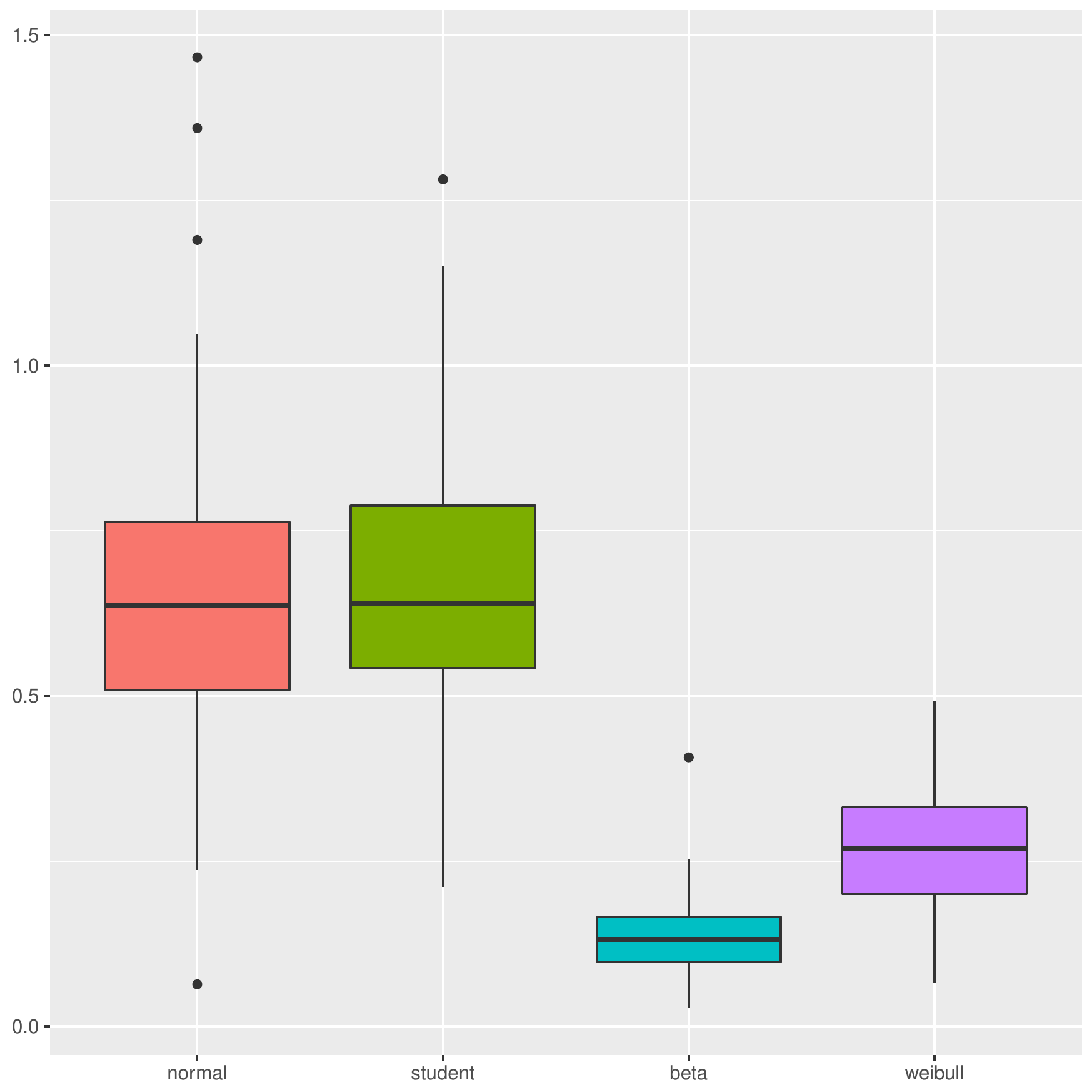}
\includegraphics[width=7cm,height=5cm,keepaspectratio]{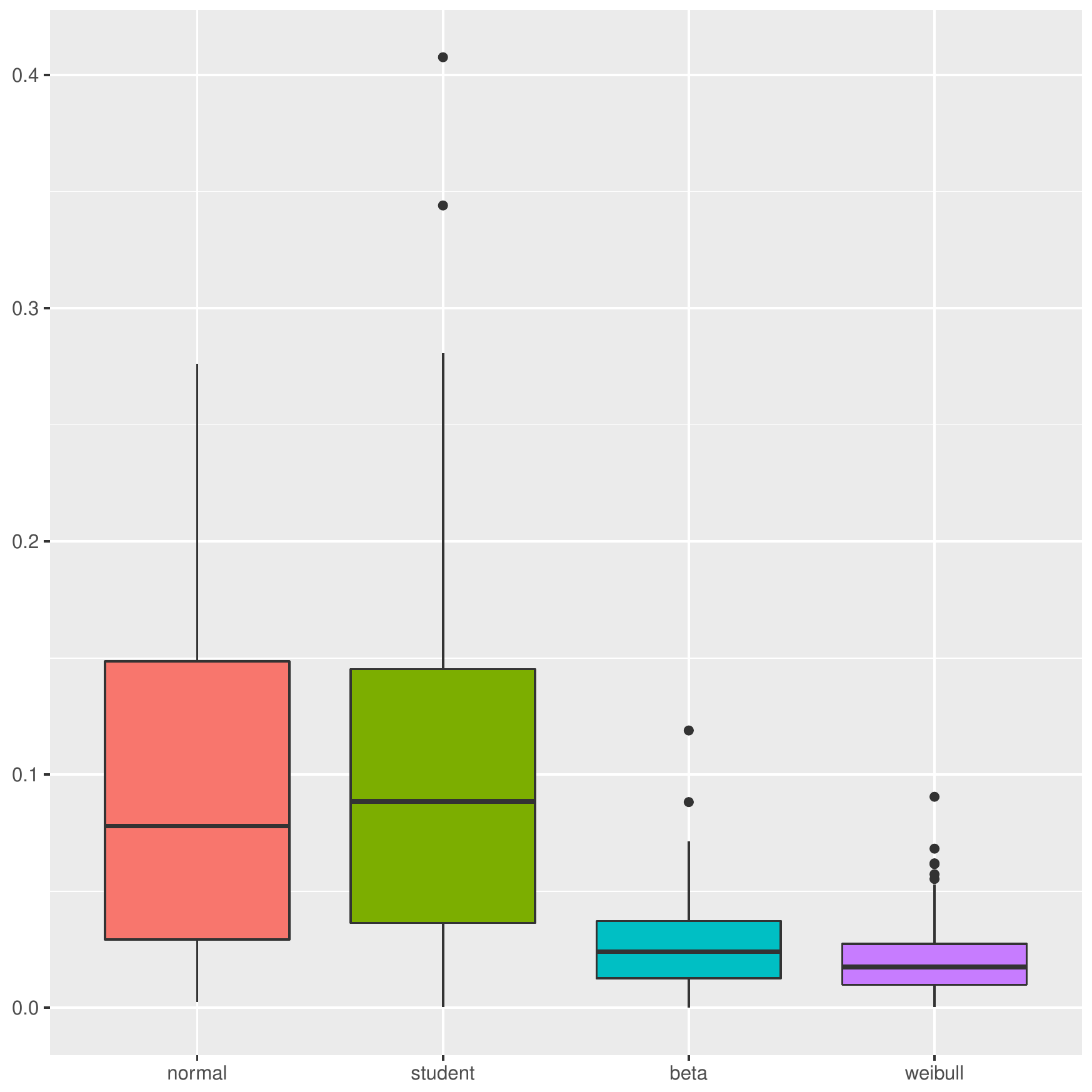}
\end{figure}

\begin{figure}[h]
\centering
\caption{Comparative boxplots of the average Interval length of Signal (left) and Noise (right) variables. Case of  $p \ll n$ and Identity Design with $\rho=0.4$. }
\label{fig:2}
\includegraphics[width=7cm,height=5cm,keepaspectratio]{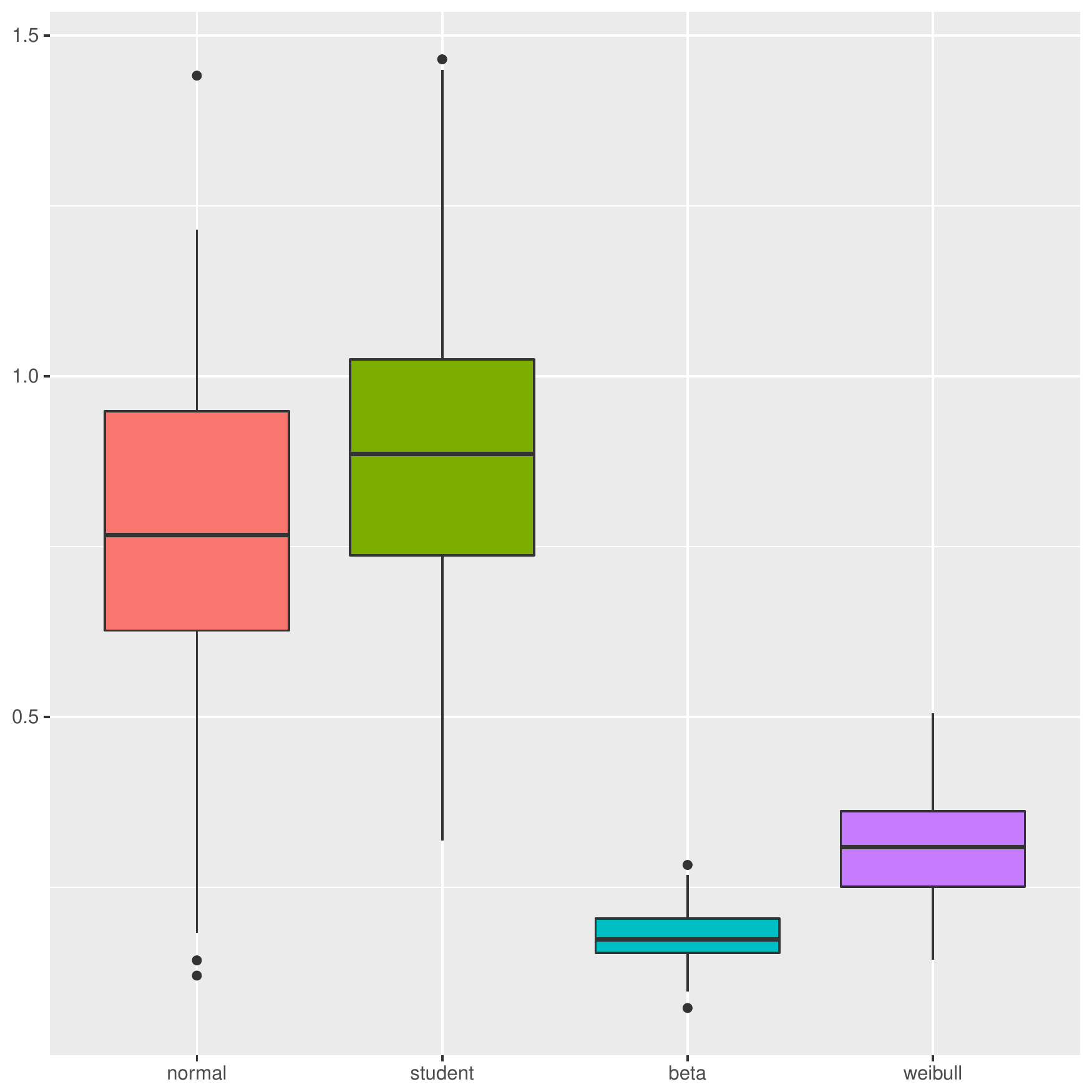}
\includegraphics[width=7cm,height=5cm,keepaspectratio]{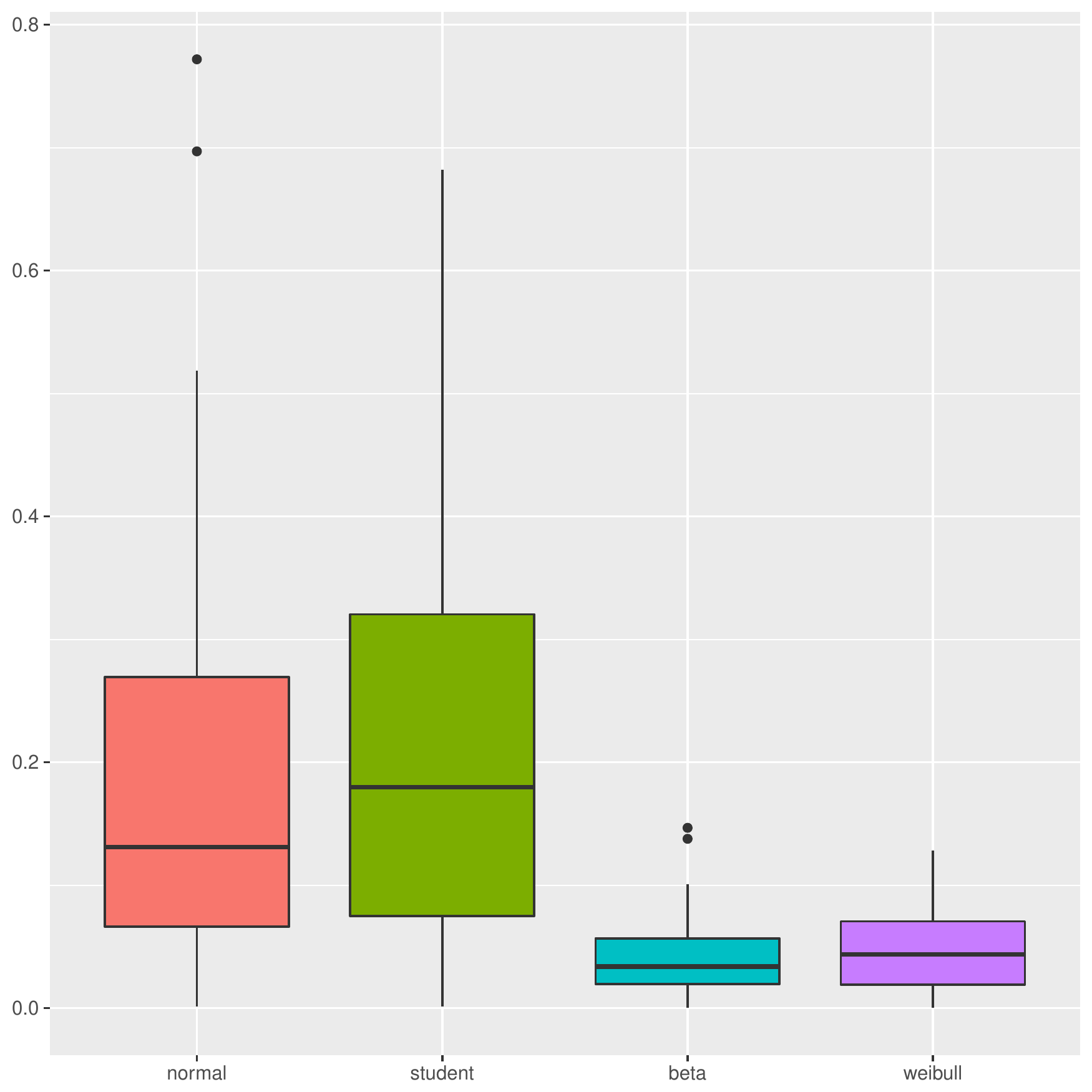}
\end{figure}

\begin{figure}[h]
\centering
\caption{Comparative boxplots of the average Interval length of Signal (left) and Noise (right) variables. Case of  $p \gg n$ and Toeplitz Design with $\rho=0.4$. }
\label{fig:3}
\includegraphics[width=7cm,height=5cm,keepaspectratio]{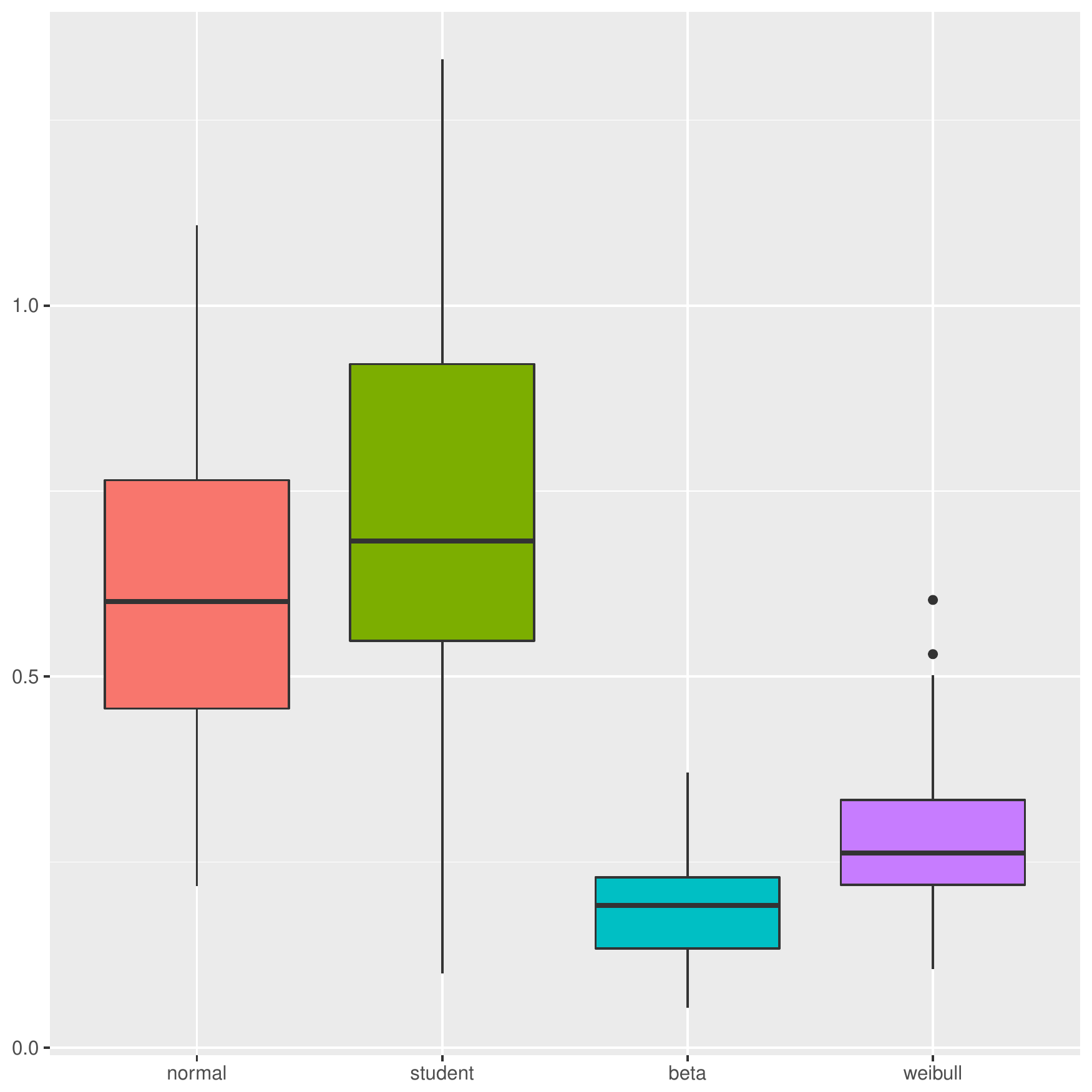}
\includegraphics[width=7cm,height=5cm,keepaspectratio]{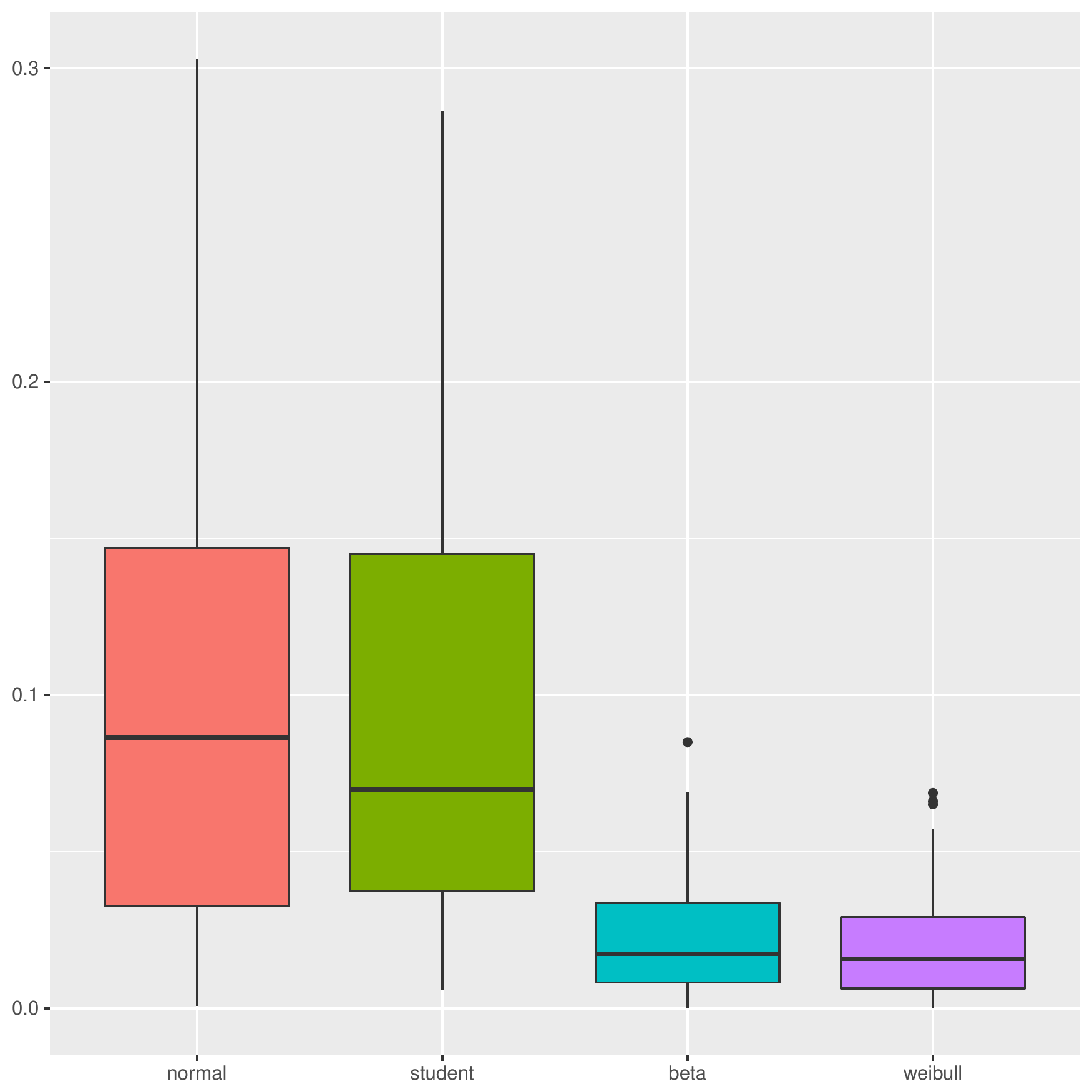}
\end{figure}

\begin{figure}[h]
\centering
\caption{Comparative boxplots of the average Interval length of Signal (left) and Noise (right) variables. Case of  $p \gg n$ and Identity Design. }
\label{fig:4}
\includegraphics[width=7cm,height=5cm,keepaspectratio]{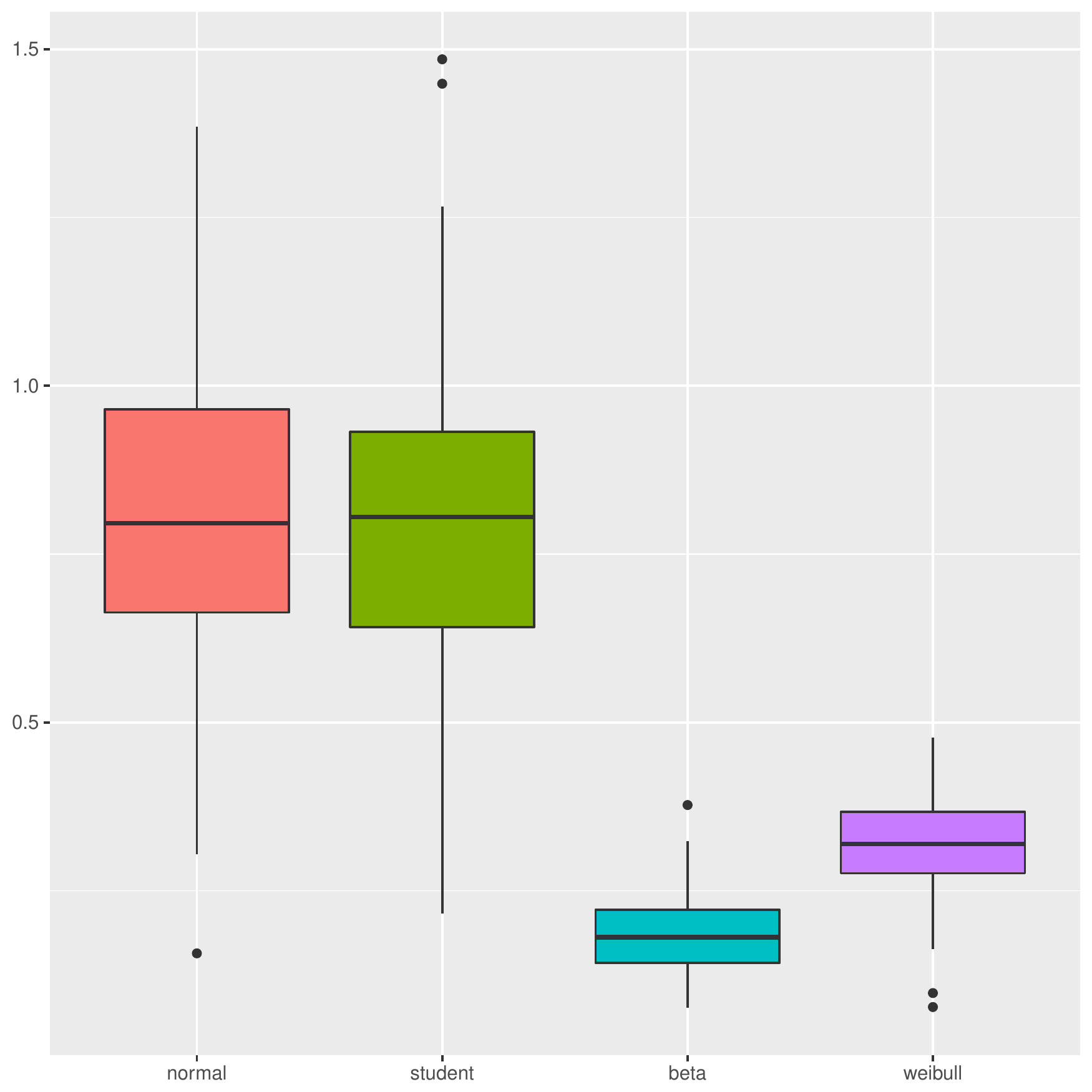}
\includegraphics[width=7cm,height=5cm,keepaspectratio]{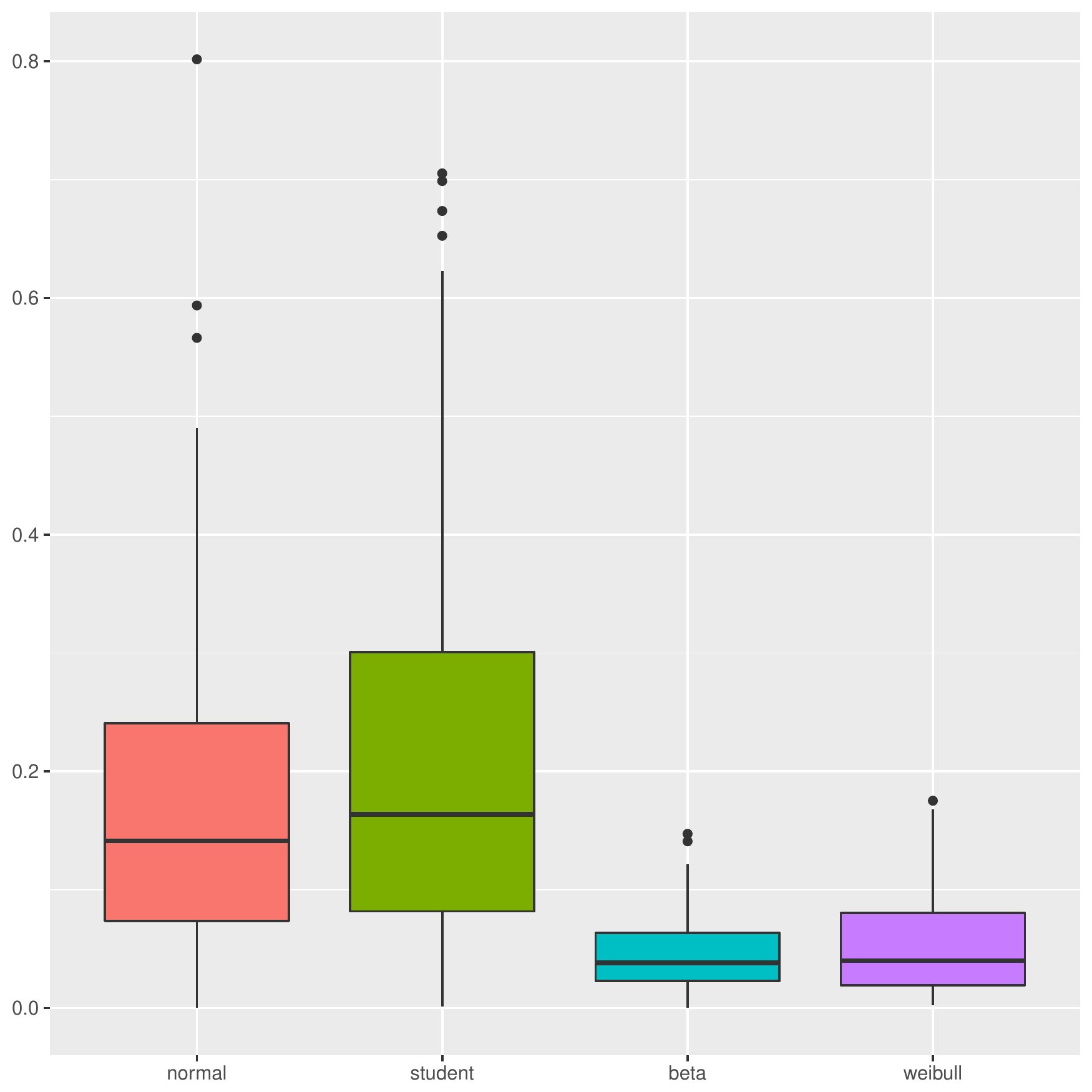}
\end{figure}
     
    \begin{table}[H]
    \centering
    \footnotesize
    \caption{Coverage Probability for Low-Dimensional Regime}
    \label{tab:lowdim}
    \begin{tabular}{@{}nd{1.1}*{3}{d{1.2}}d{1.1}d{3.2}@{}}
     \toprule 
        \multicolumn{1}{@{}N}{Distribution of the error term} &
        \multicolumn{3}{N}{Simulation Setting} &
        \multicolumn{1}{N@{}}{ } \\
      \cmidrule(lr){2-5}
        &
        \multicolumn{2}{V{6.5em}}{Toeplitz design} &
        \multicolumn{2}{V{6.5em}}{Identity design} \\
      \cmidrule(lr){2-3}\cmidrule(lr){4-5}          &
        \multicolumn{1}{V{4.5em}}{Signal Variable } &
        \multicolumn{1}{V{4.5em}}{Noise Variable} &
        \multicolumn{1}{V{4.5em}}{Signal Variable } &
        \multicolumn{1}{V{4.5em}}{Noise Variable} \\
              \cmidrule(r){1-1}\cmidrule(lr){2-2}\cmidrule(lr){3-3} 
        \cmidrule(lr){4-4}\cmidrule(lr){5-5}
        Normal    & 0.97   & 0.98 & 0.95 & 0.94 \\
        Student   & 0.97  &  1   &  0.97   & 0.98  \\
        Beta   &  0.94    &  1   &  0.98   &  0.97   \\
        Weibull   &  0.98   &  0.98   & 0.94   &  0.98   \\
      \bottomrule
    \end{tabular}
  \end{table}
    
  \begin{table}[H]
    \centering
    \footnotesize
    \caption{Coverage Probability for High-Dimensional Regime}
    \label{tab:highdim}
    \begin{tabular}{@{}nd{1.1}*{3}{d{1.2}}d{1.1}d{3.2}@{}}
     \toprule 
        \multicolumn{1}{@{}N}{Distribution of the error term} &
        \multicolumn{3}{N}{Simulation Setting} &
        \multicolumn{1}{N@{}}{ } \\
      \cmidrule(lr){2-5}
        &
        \multicolumn{2}{V{6.5em}}{Toeplitz design} &
        \multicolumn{2}{V{6.5em}}{Identity design} \\
      \cmidrule(lr){2-3}\cmidrule(lr){4-5}          &
        \multicolumn{1}{V{4.5em}}{Signal Variable } &
        \multicolumn{1}{V{4.5em}}{Noise Variable} &
        \multicolumn{1}{V{4.5em}}{Signal Variable } &
        \multicolumn{1}{V{4.5em}}{Noise Variable} \\
              \cmidrule(r){1-1}\cmidrule(lr){2-2}\cmidrule(lr){3-3} 
        \cmidrule(lr){4-4}\cmidrule(lr){5-5}
        Normal    & 0.92   & 0.96 & 0.97 & 0.95 \\
        Student   & 0.96  &  0.98   &  0.96   & 0.98  \\
        Beta   &  1    &  1   &  0.96   &  0.97   \\
        Weibull   &  0.95   &  1   & 0.87   &  0.97   \\
      \bottomrule
    \end{tabular}
  \end{table}

\subsection{Whole Blood Transcriptional HIV
 Data}

  The objective of this study is to  illustrate the performance of the proposed two-step estimator in characterizing the transcriptional signature of an early acute HIV infection. Researchers have recently shown great interest in modeling viral load (plasma HIV-1 RNA copies) data after initiation of a potent antiretroviral (ARV) treatment. Viral load is a measure of the amount of actively replicating virus and is used as a marker of disease progression among HIV-infected patients. However, the
extent of viral expression and the underlying mechanisms of the
persistence of HIV-1 in this viral reservoir have not been fully
recovered.

Moreover, viral load measurements are often subject to left censoring due to a lower limit of quantification.
We aim to find  a  pattern describing the
interaction between the HIV virus and the gene expression values and can be  useful for understanding
the pathogenesis of HIV infection and for developing effective vaccines \citep{F15}.   
We evaluated    48803 of  Illumina BreadArray based gene expressions identified through a   whole blood transcriptional,   genome-wide   analysis  for association with acute HIV infection.  Each array on the HumanHT-12 v4 Expression BeadChip targets more than 31,000 annotated genes with more than 47,000 probes derived from the National Center for Biotechnology Information Reference Sequence. 
This data set is part of the ``MicroArray quality control II'' project, which is available from the
gene expression omnibus database with accession number GSE29429.

   \begin{table}[h]
    \centering
    \footnotesize
    \caption{Transcriptional Signatures}
    \label{tab:4a}
    \begin{tabular}{@{}cd{13.1}*{2}{cd{1.2}*{2}}ld{1.1}rd{1.2}@{}}
     \toprule 
        \multicolumn{2}{@{}N}{} &
        \multicolumn{5}{N}{ } &
        \multicolumn{1}{N@{}}{ } \\
        &
        \multicolumn{2}{V{4.5em}}{Week 1} &
        \multicolumn{1}{V{0.1em}}{ } &
        \multicolumn{2}{V{4.5em}}{Week 4} &
        \multicolumn{1}{V{0.1em}}{ } \\
      \cmidrule(lr){2-3}  \cmidrule(lr){5-6} 
        &
        \multicolumn{1}{V{6.5em}}{Gene Symbol} &
        \multicolumn{1}{V{6.5em}}{Confidence Interval} &
         \multicolumn{1}{V{0.5em}}{ } &
        \multicolumn{1}{V{6.5em}}{Gene Symbol} &
        \multicolumn{1}{V{6.5em}}{Confidence Interval}  \\
              \cmidrule(lr){2-2}\cmidrule(lr){3-3} %
        \cmidrule(lr){5-5}\cmidrule(lr){6-6}  
          &  \mbox{\scriptsize MKL1} &   (-7.449,   -7.365)   & & \mbox{\scriptsize ABCD4} &   (-5.718, -3.020)  &   \\
          &  \mbox{\scriptsize MAGEC1} &   (-0.432, -0.345)   & & \mbox{\scriptsize LSP1} &  (-0.365,  -0.164)    &  \\
        & \mbox{\scriptsize PKD1L1} &  (-7.556,  -4.262)  &  & \mbox{\scriptsize PRDM16} &  (-1.252,   -0.388)  &   \\
         & \mbox{\scriptsize  PNOC} &  (-2.234, -2.146)   &   &\mbox{\scriptsize LOC728343} &  (-1.532,  -0.599)  &   \\
       & \mbox{\scriptsize SYNE2} &  (-1.725,  -1.235)    &  & \mbox{\scriptsize PES1} &  (-2.217,   -1.973)  &   \\
          &  \mbox{\scriptsize  CLK1} &   (-0.898,   -0.732)   &  & \mbox{\scriptsize FIBP} &  ( -7.563, -0.200)    &   \\
        & \mbox{\scriptsize CRB2} &  (-0.765,  -0.133 )  &   & \mbox{\scriptsize GPBP1L1} &  ( -5.267, -1.025)  &   \\
         & \mbox{\scriptsize RBM4} &  (-0.651,   -0.424)     &   & \mbox{\scriptsize CYorf15A} &  ( -1.023,  -0.787)  &   \\
       &  \mbox{\scriptsize LOC651287} &  (-2.654,   -0.116)  &   & \mbox{\scriptsize C5orf13} &  (-0.456,   -0.098)  &   \\
        & \mbox{\scriptsize MKLN1} &  (-4.901,   -2.457)   &  &  \mbox{\scriptsize REG1B} &  (-0.955,  -0.191)  &  \\
        &\mbox{\scriptsize   DBH} &  (-0.305,   -0.200)   &  & \mbox{\scriptsize SLCO4C1} &  (-0.696,   -0.537) &   \\
   & \mbox{\scriptsize PSORS1C1} &  (-0.238,   -0.048)  &  &  \mbox{\scriptsize LOC653344} &  (-0.263,  -0.204)   &  \\
       &\mbox{\scriptsize   C7orf45} &  (-0.766,   -0.025)   &  & \mbox{\scriptsize ADHFE1} &  (-0.346,  -0.162) &  \\
          & \mbox{\scriptsize   HS.578925} &  (-0.578,   -0.477)  &  &  \mbox{\scriptsize FCGR3A} &  (-0.566,  -0.011)  &   \\
        &\mbox{\scriptsize   HS.130424} &  (-1.341,   -0.160)& & \mbox{\scriptsize MARK3} &  (-0.407,  -0.072) &   \\
         & \mbox{\scriptsize   HS.147787} &  (-0.285,   -0.194)&   &    \mbox{\scriptsize POLR1C} &  (-0.385,   -0.209)  &   \\
       & \mbox{\scriptsize   GNL3} &  (-1.111,   -0.353)    &   & \mbox{\scriptsize UBE2L6} &  ( 0.351, -0.816) &  \\
      \bottomrule
    \end{tabular}
  \end{table}

58 acute HIV patients were recruited from locations in Africa (n=43) and the United States (n=15).
We analyze  the original data set containing  subjects both from Africa and the United States, with 186 males and
 females, whose Viral Loads are measured over a period of 24 weeks. 
Patient samples were collected at study enrollment (confirmed acute) for all patients and at weeks 1, 2, 4, 12 and 24.
Subjects are from 18  to 66
years old. The current data set also contains genetic information   of each
participant over     BreadArray expression values of around 6000 genes on different chromosomes.
 Weekly populations are analyzed separately. The sample size $n$ of each weekly data is around $20$. 
 We successfully applied our methodology to this data, despite the computational burden occurring with the extremely large amount of parameters.

 Table  \ref{tab:4a} summarizes   the confidence intervals concerning the treatment group. We found confidence intervals for all $48803$ genes with only $20$ samples in weekly data.
 Therefore, our method enables  the discovery of a genetic biological  pathways associated with the ARV treatment of HIV positive patients.  Censoring level was   2\% in Week 1, 5\% in Week 2, 10\% in Week 4,  70\% in Week 8, 40\% in Week 12 and 50\% in Week 24.
For illustration purposes, we present the results only    for the genes whose intervals did not contain zero, indicating  their  strong association with the Viral loads measurements. 
   
We observe that a number of the genes with large significance have been associated with HIV in previous studies; some, only very recently. MKL1 (megakaryoblastic leukemia (translocation) 1) gene is known to play an important role in the expansion and/or persistence of HIV infected cells in patients \citep{maldarelli_2014}. Similarly, from table \ref{tab:4a}, Week 1, we observe that MKL1 has a confidence interval far way from zero. Our findings of Week 1 also confirm that gene PKD1L1 has a significant confidence interval. The association of polycystic kidney disease 1 like 1 (PKD1L1) with kidney disease makes the gene expression a possible indicator of HIV associated nephropathy. In fact, kidney disease is often a sign of accelerated HIV disease progression \citep{bruggeman_2009}. In addition, as a member of ATP-binding cassette (ABC) drug transporters family, the gene ABCD4 we identified in Week 4 data has a potential important role in infectious diseases such as HIV-1 \citep{Crawford_2009}. Moreover, the gene expression GPBP1L1 is a kind of GC-rich promoter binding protein, which is a region important for HIV-1 transcription and thereby its propagation \citep{sawaya_1998}. The above showcase the parallel discovery of our method to the newly established results in  medicine, and provides evidence that our methods can be used to discover scientific findings in applications involving high-dimensional datasets.


\appendix

In Appendix A, we present proofs of the Theorems 1-9.   The rest of the supplementary material  contains proofs of the Lemmas 1-6.
Referenced citations are matching those of the main document.

\section{Proofs of  Main Theorems}\label{Proofs_theorem}

\begin{proof}[Proof of Theorem \ref{cor:fixed:ci}]
The proof of the theorem follows from the bounding residual terms in the Bahadur representation \eqref{delta_1}  with the help of Lemma \ref{cor:2} - \ref{lemma4}. 

Recall in Lemma \ref{lemma4}, we showed that 
$$\| I_4\|_\infty = \ocal_P\left( (r_n^{1/2} \vee r_n K_1) t^{1/2} (\log p)^{1/2} \bigvee t \log p / n^{1/2} \right).$$
For the  term  $I_3$, we have that
\begin{align*}
&\left\| \frac{1}{2f(0)}\left(\Omegab(\hat{\betab})-\Sigmab^{-1}(\betab^*)\right)\frac{1}{\sqrt{n}}\sum_{i=1}^n\psi_i(\betab^*) \right\|_\infty \\
&\qquad \leq \ocal_P\left( K^2 K_\gamma^2 s_\Omega^3 (1 \vee K_\gamma \vee K_\gamma \lambda_j) (\lambda_j \vee r_n^{1/2} t^{3/4} (\log p / n)^{1/2} \vee t \log p / n) \right),
\end{align*}
by applying H\"older's inequality and Hoeffding's inequality along with Lemma \ref{lemma2}. 

For the   term $I_2$, we have 
\begin{align*}
&\left\| \frac{1}{2f(0)}\Omegab(\hat{\betab})\sqrt{n}\cdot \mathcal{O}(\|\hat{\betab}-\betab^*\|_1)(\hat{\betab}-\betab^*) \right\|_\infty \\
& \qquad \leq \frac{\sqrt{n}}{2f(0)} \left( \left \| \Omegab(\hat{\betab}) - \Sigmab^{-1}(\betab^*)\right\|_{\max} + \left\| \Sigmab^{-1}(\betab^*) \right\|_{\max} \right) \mathcal{O}(\|\hat{\betab}-\betab^*\|_1^2) \\
& \qquad \leq \frac{\sqrt{n}}{2f(0)} \left( \left \| \Omegab(\hat{\betab}) - \Sigmab^{-1}(\betab^*)\right\|_{\infty} + \left\| \Sigmab^{-1}(\betab^*) \right\|_{\max} \right) \mathcal{O}(\|\hat{\betab}-\betab^*\|_1^2) \\
&\qquad \leq \ocal_P\left( K^2 K_\gamma^2 s_\Omega^3 d_n^2 n^{1/2} (1 \vee K_\gamma \vee K_\gamma \lambda_j) (\lambda_j \vee r_n^{1/2} t^{3/4} (\log p / n)^{1/2} \vee t \log p / n) \right. \\
&\qquad \qquad \left. \bigvee K_\gamma n^{1/2} d_n^2 \right),
\end{align*}
by H\"older's inequality and Lemma \ref{lemma2}, where $\| A \|_\infty$ denotes the max row sum of matrix $A$ and $\| A \|_{\max}$ denotes the maximum element in the matrix $A$.

Lastly, for the only remainder term in \eqref{delta_1}, $I_1$,  we use  H\"older's inequality and Lemma \ref{lemma2},
\begin{align*}
&\sqrt{n}\left(I-\Omegab(\hat{\betab})   \Sigmab(\betab^*)  \right)\left(\hat{\betab}-\betab^*\right) \\ 
& \qquad = \sqrt{n}\left(\Sigmab^{-1}(\betab^*)  -\Omegab(\hat{\betab}) \right) \Sigmab(\betab^*)  \left(\hat{\betab}-\betab^*\right)\\
& \qquad \leq \ocal_P\left( K^2 K_\gamma^3 s_\Omega^3 d_n n^{1/2} (1 \vee K_\gamma \vee K_\gamma \lambda_j) (\lambda_j \vee r_n^{1/2} t^{3/4} (\log p / n)^{1/2} \vee t \log p / n)  \right).
\end{align*}
\end{proof}

\begin{proof}[Proof of Theorem \ref{normality}]

We begin the proof by noticing that 
\begin{align*}
\psi_i(\betab^*) &= \mbox{sign}(y_i-\max\{0,x_i\betab^*\})(w_i(\betab^*))^\top \\ 
&= \mbox{sign}(\max\{0,x_i\betab^*+\varepsilon_i\}-\max\{0,x_i\betab^*\})(w_i(\betab^*))^\top  .
\end{align*}
Recollect that by Condition {\bf (E)}, $\PP(\varepsilon_i \geq 0 ) =1/2$.
Additionally, we observe that in distribution, the term on the right hand side is equal to $ w_i^\top (\betab^*) R_i $,  with 
$\{ R_i\}_{i=1}^n$ denoting an i.i.d. Rademarcher sequence defined as  $R_i = \mbox{sign}(-\varepsilon_i)$.
Hence, it suffices to analyze the distributional properties of $ w_i^\top (\betab^*) R_i.$ Moreover, Rademacher random variables  are  independent in distribution from $w_i(\betab^*)$.  
Thus, we provide asymptotics of 
\begin{align*}
  \frac{1}{2f(0)}\Sigmab^{-1}(\betab^*)\frac{1}{\sqrt{n}}\sum_{i=1}^n w_i^\top (\betab^*)R_i.
\end{align*} 
We begin by defining 
$$V_i := \frac{1}{\sqrt{n}}W_{ij}\ind(x_i\betab^*>0)R_i = \frac{1}{\sqrt{n}}X_{ij}\ind(x_i\betab^*>0)R_i$$   and we also define $T_n := \sum_{i=1}^n V_i.$ 
Notice that $V_i$'s are independent from each other, since we assumed that each observation  is independent in our design. We have
\begin{align}\label{EV_i}
& \sum_{i=1}^n \EE |V_i|^{2+\delta} 
= \left(\frac{1}{\sqrt{n}}\right)^{2+\delta}\EE \sum_{i=1}^n |X_{ij}\ind(x_i\betab^*>0)|^{2+\delta} \leq  n^{-1-\delta/2}\EE \sum_{i=1}^n |X_{ij}|^{2+\delta} \leq n^{-\delta/2}K .
\end{align}
Moreover,
$
\var T_n  
= \frac{1}{n}\sum_{i=1}^n \EE \left(X_{ij}\ind(x_i\betab^*>0)R_i\right)^2- \left(\EE X_{ij}\ind(x_i\betab^*>0)R_i\right)^2.
$
Since $R_i$  are independent from $X$, $$\EE X_{ij}\ind(x_i\betab^*>0)R_i = \EE X_{ij}\ind(x_i\betab^*>0)\cdot \EE R_i = 0.$$ In addition, also due to this fact, $V_i$ follows a symmetric distribution about $0$. Thus,
\begin{align*}
\var T_n &= \frac{1}{n}\EE \sum_{i=1}^n \left(X_{ij}\ind(x_i\betab^*>0)R_i\right)^2 = \frac{1}{n}\EE \left(\sum_{i=1}^n X_{ij}\ind(x_i\betab^*>0)R_i\right)^2 \geq \frac{1}{n}\int_{-n}^n t_n^2f(t_n) dt_n,
\end{align*}
where with a little abuse in notation we denote the density and distribution of $T_n$ to be $f(t_n)$ and $F(t_n)$.
 Observe  that $$\frac{1}{n}\EE \left(\sum_{i=1}^n X_{ij}\ind(x_i\betab^*>0)R_i\right)^2 = \frac{1}{n}\int_{-\infty}^\infty t_n^2f(t_n) dt_n\geq \frac{1}{n}\int_{-n}^n t_n^2f(t_n) dt_n.$$ Thus,
 \begin{align}\label{varT_n}
\var T_n &\geq \frac{1}{n}\left(t_n^2F(t_n)\given[\big]_{-n}^n - 2\int_{-n}^n t_nF(t_n) dt_n\right) 
\\
&\geq \frac{1}{n}\left(n^2F(n)-n^2F(-n) - 2\int_{-n}^n t_n dt_n\right)\nonumber\\
&=  \frac{1}{n}\left(2n^2F(n) - n^2\right) = n\left( 2F(n) - 1\right) \nonumber
\end{align} 
Now combining \eqref{EV_i} and \eqref{varT_n}, we have
$\lim_{n\rightarrow\infty} \frac{\sum_{i=1}^n \EE |V_i|^{2+\delta}}{(\var T_n)^{1+\frac{\delta}{2}}} = 0.$
Thereby, we arrive at the result
$$\frac{1}{\sqrt{n}}\left(\sum_{i=1}^n w_i^\top (\betab^*)R_i\right)_j \xrightarrow{d} \mathcal{N}\left(0,\var T_n\right),$$
with the fact that 
$ \var T_n = \frac{1}{n}\EE \sum_{i=1}^n W_{ij}(\betab^*)^2 = \frac{1}{n}\EE W_j^\top (\betab^*)W_j(\betab^*) = \Sigmab(\betab^*)_{jj}.$ 
Also, the covariance 
$$\EE\left[\frac{1}{\sqrt{n}}\left(\sum_{i=1}^n w_i^\top (\betab^*)R_i\right)_{j_1}\frac{1}{\sqrt{n}}\left(\sum_{i=1}^n w_i^\top (\betab^*)R_i\right)_{j_2}\right] = \EE \left[\frac{1}{n}\sum_{i=1}^nW_{ij_1}(\betab^*)W_{ij_2}(\betab^*)\right] = \Sigmab(\betab^*)_{j_1j_2}.$$

Therefore, we have the following conclusion, 
$$\left[\frac{1}{2f(0)}\Sigmab^{-1}(\betab^*)\frac{1}{\sqrt{n}}\sum_{i=1}^n\psi_i(\betab^*)\right]_j\xrightarrow{d} \mathcal{N}\left(0, \frac{1}{4f(0)^2}\left[\Sigmab^{-1}(\betab^*)\Sigmab(\betab^*)\left(\Sigmab^{-1}(\betab^*)\right)^\top \right]_{jj}\right),$$
where $j=1,\cdots,p$. This gives 
\begin{align} \label{invsigma_normality}
\left[ \Sigmab^{-1}(\betab^*)_{jj} \right]^{-\frac{1}{2}} \left[\frac{1}{2f(0)}\Sigmab^{-1}(\betab^*)\frac{1}{\sqrt{n}}\sum_{i=1}^n\psi_i(\betab^*)\right]_j\xrightarrow{d} \mathcal{N}\left(0, \frac{1}{4f(0)^2}\right)
\end{align}

Notice that for two nonnegative real numbers $a$ and $b$, it holds that
\begin{align*}
\frac{1}{\sqrt{a}} - \frac{1}{\sqrt{b}} = \frac{\sqrt{b} - \sqrt{a}}{\sqrt{ab}} = \frac{b - a}{\sqrt{ab}(\sqrt{b} + \sqrt{a})}.
\end{align*}
We first make note of a result in the proof of Theorem \ref{cor:ci}, that 
\begin{align} \label{eqn:oso_bound}
\left\| \hat \Omegab(\hat\betab) \Sigmab (\hat \betab) \hat \Omegab(\hat\betab) - \Sigmab^{-1} (\betab^*)\right\|_{\max} = \smallocal_P(1)
\end{align}
Let $a = \left[ \hat \Omegab(\hat\betab) \Sigmab (\hat \betab) \hat \Omegab(\hat\betab) \right]_{jj}$ and $b = \Sigmab^{-1} (\betab^*)_{jj}$. By Condition {\bf{(CC)}}, we have $\sqrt{b}$ is bounded away from zero. Then, $\sqrt{a}$ is also bounded away from zero by \eqref{eqn:oso_bound}, and so is  $\sqrt{ab}(\sqrt{b} + \sqrt{a})$, since we have
\begin{align*}
\left[ \Sigmab^{-1} (\betab^*)\right]_{jj}  - \left[ \hat \Omegab(\hat\betab) \Sigmab (\hat \betab) \hat \Omegab(\hat\betab) \right]_{jj} \leq
\left\| \hat \Omegab(\hat\betab) \Sigmab (\hat \betab) \hat \Omegab(\hat\betab) - \Sigmab^{-1} (\betab^*)\right\|_{\max}
= \smallocal_P \left( 1 \right).
\end{align*}
The rate above follows from \eqref{oso_invsig_max_bound} in the proof of Theorem \ref{cor:ci}. Notice the rate is of order smaller than the rate assumption in Theorem \ref{cor:fixed:ci}.

Thus, we can deduce that
\begin{align*}
\left[\Omegab(\hat{\betab}) \hat\Sigmab(\hat{\betab})\Omegab (\hat{\betab})\right]_{jj}^{-\frac{1}{2}} - \left[ \Sigmab^{-1}(\betab^*)_{jj} \right]^{-\frac{1}{2}} \leq 
C \left\| \hat \Omegab(\hat\betab) \Sigmab (\hat \betab) \hat \Omegab(\hat\betab) - \Sigmab^{-1} (\betab^*)\right\|_{\max}.
\end{align*}
for some finite constant $C$. Applying Slutsky theorem on \eqref{invsigma_normality} with the inequality above, the desired result is obtained.
\end{proof}

\begin{proof}[Proof of Theorem \ref{thm:f}]
We can rewrite the expression $\hat f(0)$ in \eqref{eq:density} as
\begin{align*}
\hat f(0) &= \hat h_n^{-1}\frac
{\sum_{i=1}^n \ind(x_i\hat{\betab}>0)\ind(0\leq y_i-x_i\hat{\betab}\leq \hat h_n)}
{\sum_{i=1}^n\ind(x_i\hat{\betab}>0)} \\
&= \hat h_n^{-1}\frac
{n^{-1} \sum_{i=1}^n \ind(x_i\hat{\betab}>0)\ind(0\leq y_i-x_i\hat{\betab}\leq \hat h_n)}
{n^{-1} \sum_{i=1}^n \PP \{ x_i \betab^* >0 \}} 
\cdot \frac{n^{-1} \sum_{i=1}^n \PP \{ x_i \betab^* >0 \}} {n^{-1}\sum_{i=1}^n\ind(x_i\hat{\betab}>0)}.
\end{align*}
Since $\left|  n^{-1} \sum_{i=1}^n \left[ \ind\{x_i \hat \betab >0 \} - \PP \{ x_i \betab^* >0\}\right] \right| = \smallocal_P(1)$, we have 
\begin{align*}
\hat f(0) \xrightarrow{d} \frac
{(\hat h_n n)^{-1} \sum_{i=1}^n \ind(x_i\hat{\betab}>0)\ind(0\leq y_i-x_i\hat{\betab}\leq \hat h_n)}
{n^{-1} \sum_{i=1}^n \PP \{ x_i \betab^* >0 \}}.
\end{align*}
Using a similar argument and the fact that $\lim_{n \rightarrow \infty} \hat h_n / h_n = 1$, we have
\begin{align*}
\hat f(0) \xrightarrow{d} \frac
{(h_n n)^{-1} \sum_{i=1}^n \ind(x_i\hat{\betab}>0)\ind(0\leq y_i-x_i\hat{\betab}\leq \hat h_n)}
{n^{-1} \sum_{i=1}^n \PP \{ x_i \betab^* >0 \}}.
\end{align*}

Now we work on the numerator of right hand side. Specifically, let $\eta_i = y_i - x_i\betab^*$ and $\hat \eta_i = y_i - x_i \hat\betab$, we look at the difference of the quantities below,
\begin{align*}
& (h_n n)^{-1}   \left| \sum_{i=1}^n \ind\{ x_i \hat \betab >0\} \ind\{0 \leq \hat \eta_i \leq \hat h_n\}  - \sum_{i=1}^n \ind\{ x_i \betab^* >0\} \ind\{0 \leq \eta_i \leq h_n\} \right| \\
&\leq (h_n n)^{-1} \left| \sum_{i=1}^n \ind\{ x_i \hat \betab >0\} \ind\{0 \leq \hat \eta_i \leq \hat h_n\}  
- \sum_{i=1}^n \ind\{ x_i \betab^* >0\} \ind\{0 \leq \hat \eta_i \leq \hat h_n\} \right| \\
&\qquad + 2(h_n n)^{-1} \left| \sum_{i=1}^n \ind\{ x_i \hat \betab >0\} \ind\{0 \leq \eta_i \leq h_n\} - \sum_{i=1}^n \ind\{ x_i \betab^* >0\} \ind\{0 \leq \eta_i \leq h_n\} \right| \\
&\qquad + (h_n n)^{-1} \left| \sum_{i=1}^n \ind\{ x_i \betab^* >0\} \ind\{0 \leq \hat \eta_i \leq \hat h_n\}  
- \sum_{i=1}^n \ind\{ x_i \betab^* >0\} \ind\{0 \leq \eta_i \leq h_n\} \right| \\
& \leq \underbrace{3(h_n n)^{-1} \sum_{i=1}^n  \ind \{ x_i \betab^* \leq x_i  (\hat \betab - \betab^*) \}}_{T_1}
 + \underbrace{(h_n n)^{-1} \left| \sum_{i=1}^n \left( \ind\{0 \leq \hat \eta_i \leq \hat h_n\}  - \ind\{0 \leq \eta_i \leq h_n\} \right) \right|}_{T_2}.
\end{align*}

We begin with term $T_1$. By Condition {\bf{(E)}}, we have $\EE T_1 = \smallocal (h_n^{-1}  \| \hat \betab - \betab^*\|_1)$. By Corollary \ref{lemma0}, we have 
$$T_1 - \EE T_1 \leq |T_1 - \EE T_1| = \smallocal_P \left( h_n^{-1} (r_n^{1/2} t^{3/4} (\log p / n)^{1/2} \vee t\log p / n) \right),$$
which then brings us that $T_1$ is of order $\smallocal_P(1)$. For term $T_2$, we work out the expression
\begin{align*}
&\ind\{0 \leq \hat \eta_i \leq \hat h_n\}  - \ind\{0 \leq \eta_i \leq h_n\}  = \ind\{0 \leq \hat \eta_i\}\ind(\hat \eta_i \leq \hat h_n\}  - \ind\{0 \leq \eta_i\} \ind\{\eta_i \leq h_n\} \\
& \qquad = \ind\{0 \leq \hat \eta_i\} \left( \ind(\hat \eta_i \leq \hat h_n\} - \ind(\eta_i \leq h_n\} \right) + \left( \ind\{0 \leq \hat \eta_i\} - \ind\{0 \leq \eta_i\} \right) \ind\{\eta_i \leq h_n\} \\
&\qquad \leq \ind\{\hat \eta_i \leq \hat h_n\} - \ind\{\eta_i \leq h_n\} + \ind\{0 \leq \hat \eta_i\} - \ind\{0 \leq \eta_i\}.
\end{align*}
Next, we notice that for real numbers $a$ and $b$, we have  $\ind(a > 0) - \ind(b > 0) \leq \ind(|b| \leq |a - b|).$ 
Thus, we have 
\begin{align*}
T_2 &\leq (h_n n)^{-1} \left| \sum_{i=1}^n \left\{ \ind(\hat \eta_i \leq \hat h_n\} - \ind\{\eta_i \leq h_n\} + \ind\{0 \leq \hat \eta_i\} - \ind\{0 \leq \eta_i\} \right) \right| \\
&\leq h_n^{-1} n^{-1} \sum_{i=1}^n \ind\{ | h_n - \eta_i | \leq | \hat h_n - h_n | +  | \eta_i - \hat \eta_i | \} + h_n^{-1} n^{-1} \sum_{i=1}^n \ind\{ | \eta_i | \leq | \hat \eta_i - \eta_i |\} \\
&\leq \underbrace{h_n^{-1} n^{-1} \sum_{i=1}^n \ind\{ | h_n - \eta_i | \leq | \hat h_n - h_n | +   \| x_i \|_\infty \|\hat \betab - \betab^*\|_1 \}}_{T_{21}} \\
&\qquad + \underbrace{h_n^{-1} n^{-1} \sum_{i=1}^n \ind\{ | \eta_i | \leq \| x_i \|_\infty \|\betab^* - \hat \betab\|_1 \}}_{T_{22}}
\end{align*}
To bound $T_{21}$, we use similar techniques as with $T_1$. Notice that 
\begin{align*}
\EE T_{21} = h_n^{-1} \PP \left( | h_n - \eta_i | \leq | \hat h_n - h_n | +   \| x_i \|_\infty \|\hat \betab - \betab^*\|_1 \right)
\end{align*}
It is easy to see that $| h_n - \eta_i |$ shares the nice property of the density of $\varepsilon_i$. Thus, $\EE T_{21}$ is bounded by $\smallocal_P(1)$. Then by Hoeffding's inequality, we have that with probability approaching $1$ that $T_{21}$ is of $\smallocal_P(1)$. $T_{22}$ can be bounded in exactly the same steps.

Finally, we are ready to put everything together that
$$(h_n n)^{-1} \left| \sum_{i=1}^n \ind\{ x_i \hat \betab >0\} \ind\{0 \leq \hat \eta_i \leq \hat h_n\}  - \sum_{i=1}^n \ind\{ x_i \betab^* >0\} \ind\{0 \leq \eta_i \leq h_n\} \right| = \smallocal_P(1).$$
By applying Slutsky theorem, the result follows directly,
\begin{align*}
\hat f(0) \xrightarrow{d} \frac
{\sum_{i=1}^n \ind\{ x_i \betab^* >0\} \ind\{0 \leq \eta_i \leq h_n\}}
{n^{-1} \sum_{i=1}^n \PP \{ x_i \betab^* >0 \}}.
\end{align*}
\end{proof}

\begin{proof}[Proof of Corollary \ref{cor1}]
By multiplying and dividing the term $f(0)$, we can rewrite the term on the left hand side as 
\begin{align*}
\left[\Omegab(\hat{\betab}) \hat\Sigmab(\hat{\betab})\Omegab (\hat{\betab})\right]_{jj}^\frac{1}{2}U_j \cdot 2 \widehat f(0)
&= \left[\Omegab(\hat{\betab}) \hat\Sigmab(\hat{\betab})\Omegab (\hat{\betab})\right]_{jj}^\frac{1}{2}U_j \cdot 2 f(0) \frac{\widehat f(0)}{f(0)}.
\end{align*}
Also, as a result of theorem \ref{thm:f}, we have
\begin{align*}
\frac{|\widehat f(0) - {f(0)}|}{f(0)} = |\widehat f(0)/f(0) - 1| = \smallocal_P(1),
\end{align*}
with Condition {\bf{(E)}} guarantees that $f(0)$ is bounded away from $0$. It also indicates that $\widehat f(0) / f(0) \xrightarrow{d} 1$. Finally, we apply Slutsky's Theorem and Theorem \ref{normality}, we have
\begin{align*}
\left[\Omegab(\hat{\betab}) \hat\Sigmab(\hat{\betab})\Omegab (\hat{\betab})\right]_{jj}^\frac{1}{2}U_j \cdot 2 \widehat f(0) \xrightarrow[n,p,s_{\betab^*}\rightarrow\infty]{d} \mathcal{N}\left(0,1\right).
\end{align*}
\end{proof}

\begin{proof}[Proof of Theorem \ref{cor:ci}]
 The result of Theorem \ref{cor:ci} is a simple consequence of Wald's device and results of Corollary \ref{cor1}.
 The only missing link is an upper bound on 
 \begin{align} \label{oso_invsig_max}
 \left\| \Omegab(\hat\betab) \Sigmab (\hat \betab) \Omegab(\hat\betab) - \Sigmab^{-1} (\betab^*)\right\|_{\max}.
 \end{align}
 First, observe that 
 \[
  \Omegab(\hat\betab) \Sigmab (\hat \betab) \Omegab(\hat\betab) - \Sigmab^{-1} (\betab^*) 
 =
 \underbrace{\left( \Omegab(\hat\betab) - \Sigmab^{-1} (\betab^*) \right)  \Sigmab (\hat \betab) \Omegab(\hat\betab)}_{T_1}   + \underbrace{\Sigmab^{-1} (\betab^*) \left( \Sigmab (\hat \betab) \Omegab(\hat\betab) - \mathbb{I} \right)}_{T_2} .
 \]
 Regarding term $T_1$, observe that by Lemma \ref{lemma2}  it is equal to $\smallocal_P(1)$ whenever $\|\Sigmab (\hat \betab) \Omegab(\hat\betab)\|_{\max}$ is $\mathcal{O}_P(1)$. This can be seen from the decomposition of $\Sigmab (\hat \betab)\Omegab(\hat\betab) - \mathbb{I}$, which reads,
 \begin{align*}
 \left\| \Sigmab (\hat \betab)\Omegab(\hat\betab) - \mathbb{I} \right\|_{\max} & =
\underbrace{\left\| \Sigmab^{-1} (\betab^*) \left(\hat \Sigmab  (\hat\betab )- \Sigmab  (\betab^*) \right) \right\|_{\max}}_{T_{21}} \\
&+ \underbrace{\left\| \left( \Omegab(\hat\betab) - \Sigmab^{-1} (\betab^*) \right)   \left(\hat \Sigmab  (\hat\betab )-  \Sigmab  (\betab^*) \right) \right\|_{\max} }_{T_{22}}
+ \underbrace{ \left\| \Sigmab  (\betab^*) \left( \Omegab(\hat\betab) - \Sigmab^{-1} (\betab^*) \right) \right\|_{\max}}_{T_{23}}
 \end{align*}
 
 We notice that
 \begin{align}
 T_{21} &= \left\| \Sigmab^{-1} (\betab^*) \left( n^{-1} \sum_{i=1}^n w_i^\top(\hat \betab) w_i(\hat \betab) - n^{-1} \sum_{i=1}^n w_i^\top(\betab^*) w_i(\betab^*) \right. \right. \nonumber\\ 
  &\qquad \left. \left. + n^{-1} \sum_{i=1}^n w_i^\top(\betab^*) w_i(\betab^*) - n^{-1} \EE \sum_{i=1}^n w_i^\top(\betab^*) w_i(\betab^*) \right) \right\|_{\max}  \nonumber \\
 &\leq \left\| \Sigmab^{-1} (\betab^*) \left( n^{-1} \sum_{i=1}^n \left(w_i(\hat \betab) + w_i(\betab^*)\right)^\top \left(w_i(\hat \betab) - w_i(\betab^*)\right) \right) \right\|_{\max} \label{T211} \\ 
 &\qquad + \left\| \Sigmab^{-1} (\betab^*) \left( n^{-1} \sum_{i=1}^n \left( w_i^\top(\betab^*) w_i(\betab^*) - \EE w_i^\top(\betab^*) w_i(\betab^*) \right) \right) \right\|_{\max}. \label{T212}
 \end{align}
 For \eqref{T211}, we have the following bound
 \begin{align*}
 \eqref{T211} &\leq \left\| \Sigmab^{-1} (\betab^*) \right\|_\infty  \left\| n^{-1} \sum_{i=1}^n \left(w_i(\hat \betab) + w_i(\betab^*)\right)^\top \left(w_i(\hat \betab) - w_i(\betab^*)\right) \right\|_{\max} \\
 &\leq K_\gamma s_\Omega n^{-1} \sum_{i=1}^n 2K^2 \left( \ind(x_i \hat \betab > 0) - \ind(x_i \betab^*) \right),
 \end{align*}
 where $\| A \|_\infty$ denotes the max row sum of matrix $A$ and $\| A \|_{\max}$ denotes the maximum element in the matrix $A$. By Lemma \ref{lemma0}, we can easily bound the term above with $\ocal_P \left( K^2 K_\gamma s_\Omega (r_n^{1/2} t^{3/4} (\log p / n)^{1/2} \vee t \log p / n)\right)$.
For \eqref{T212}, we start with the following term,
 \begin{align*}
 n^{-1} \sum_{i=1}^n \left( W_{ij}(\betab^*) W_{ik}(\betab^*) - \EE W_{ij}(\betab^*) W_{ik}(\betab^*) \right).
 \end{align*}
 Applying Hoeffding's inequality on this term, we have that with probability approaches $1$, the term is bounded by $\ocal_P(n^{-1/2})$. Then we bound term \eqref{T212} as following,
 \begin{align*}
 \eqref{T212} &\leq \left\| \Sigmab^{-1} (\betab^*) \right\|_\infty  \left\| n^{-1} \sum_{i=1}^n \left( w_i^\top(\betab^*) w_i(\betab^*) - \EE w_i^\top(\betab^*) w_i(\betab^*) \right) \right\|_{\max} \\
 &\leq K_\gamma s_\Omega \max_{j, k} \left\{ n^{-1} \sum_{i=1}^n \left( W_{ij}(\betab^*) W_{ik}(\betab^*) - \EE W_{ij}(\betab^*) W_{ik}(\betab^*) \right) \right\} = \smallocal_P(1)
 \end{align*}
 
Term $T_{22}$ can be bounded using Lemma \ref{lemma2} and the results from term $T_{21}$, and turns out to be of order
\begin{align*}
&\ocal_P\left( K^4 K_\gamma^2 s_\Omega^3 (1 \vee K_\gamma \vee K_\gamma \lambda_j) (r_n^{1/2} t^{3/4} (\log p / n)^{1/2} \vee t \log p / n) (\lambda_j \vee r_n^{1/2} t^{3/4} (\log p / n)^{1/2} \vee t \log p / n) \right).
\end{align*}

Lastly, by Lemma \ref{lemma2}, term $T_{23}$ is of order 
$$\ocal_P\left( K^2 K_\gamma^3 s_\Omega^3 (1 \vee K_\gamma \vee K_\gamma \lambda_j) (\lambda_j \vee r_n^{1/2} t^{3/4} (\log p / n)^{1/2} \vee t \log p / n) \right).$$
 
Putting the terms together, we have $\left\| \Sigmab (\hat \betab)\Omegab(\hat\betab) - \mathbb{I} \right\|_{\max}$ bounded by
\begin{align*}
&\ocal_P\left( s_\Omega (r_n^{1/2} t^{3/4} (\log p / n)^{1/2} \vee t \log p / n) \bigvee s_\Omega^3 (1 \vee \lambda_j)(\lambda_j \vee r_n^{1/2} t^{3/4} (\log p / n)^{1/2} \vee t \log p / n) \right. \\
&\qquad \left. \bigvee s_\Omega^3 (1 \vee \lambda_j) (r_n^{1/2} t^{3/4} (\log p / n)^{1/2} \vee t \log p / n) (\lambda_j \vee r_n^{1/2} t^{3/4} (\log p / n)^{1/2} \vee t \log p / n) \right).
\end{align*}
Thus, $\|\Sigmab (\hat \betab) \Omegab(\hat\betab)\|_{\max}$ is $\mathcal{O}_P(1)$, and so can $T_2$ be shown similarly. The expression \eqref{oso_invsig_max} is then bounded as, 
\begin{align} \label{oso_invsig_max_bound}
&\left\| \hat \Omegab(\hat\betab) \Sigmab (\hat \betab) \hat \Omegab(\hat\betab) - \Sigmab^{-1} (\betab^*)\right\|_{\max}  \\
&\qquad = \ocal_P\left( s_\Omega (r_n^{1/2} t^{3/4} (\log p / n)^{1/2} \vee t \log p / n) \bigvee s_\Omega^3 (1 \vee \lambda_j)(\lambda_j \vee r_n^{1/2} t^{3/4} (\log p / n)^{1/2} \vee t \log p / n) \right. \nonumber \\
&\qquad \left. \bigvee s_\Omega^3 (1 \vee \lambda_j) (r_n^{1/2} t^{3/4} (\log p / n)^{1/2} \vee t \log p / n) (\lambda_j \vee r_n^{1/2} t^{3/4} (\log p / n)^{1/2} \vee t \log p / n) \right)\nonumber,
\end{align}
which then completes the proof.

 \end{proof}

\begin{proof}[Proof of Theorem \ref{cor:uniform:ci}]
 The result of Theorem \ref{cor:uniform:ci}  holds by observing that Bahadur representations \eqref{delta_1}   remain  accurate uniformly in the sparse vectors $\betab \in \mathcal{B}$; hence, all the steps of Theorem \ref{cor:fixed:ci}
 apply in this case as well. 
 \end{proof}

\begin{proof}[Proof of Theorem \ref{thm:clad_temp}]
The proof for the result with initial estimator chosen as the penalized CLAD estimator of \cite{muller_vdgeer_2014} follows directly from Lemma \ref{lemma0}-\ref{lemma4} and Theorem \ref{cor:fixed:ci}-\ref{cor:ci} with $r_n = s_{\betab^*}^{1/2} (\log p / n)^{1/2}$, $t = s_{\betab^*}$ and $d_n = s_{\betab^*}(\log p / n)^{1/2}$.
\end{proof}

\begin{proof}[Proof of Theorem \ref{normality_r}, \ref{cor:fixed:ci:b} and \ref{cor:uniform:ci:b}]

Due to the limit of space, we follow the line of the proof of Theorem   \ref{normality} but only give necessary details when the proof is different.
First, we observe that with a little abuse in notation
\[
\psi_i(\betab) = w_i^\top (\betab) R_i^r, \qquad R_i^r = q_i \psi (- v_i \varepsilon_i)
\]
thus it suffices to provide the asymptotic of 
\[
T_n^r:= \frac{1}{\sqrt{n}} \sum_{i=1}^n V_i^r =  \frac{1}{\sqrt{n}} \sum_{i=1}^n x_1 \ind\{ x_i \betab  >0\}  R_i^r.
\]
Moreover, observe that $R_i^r$ are necessarily bounded random variables (see Condition (r$\boldsymbol\Gamma$). Following similar steps as in Theorem \ref{normality} we obtain 
\[
\mbox{Var}(T_n^r) \geq  n  -2 \exp\{-n^2/2\} 
\]
 where in the last step we utilized Hoeffding's inequality for bounded random variables.
 
 Next, we focus on establishing an equivalent of Lemma 2 but now for the doubly 
robust estimator. Observe that 
 \begin{equation}\label{eq:lambda}
  n^{-1} \sum_{i=1}^n \EE_\varepsilon[\psi_i^r(\betab)]=n^{-1} \sum_{i=1}^n x_i^\top \ind\{ x_i \betab >0\} q_i 
  \EE_{\varepsilon} \biggl[ \psi \Bigl(- v_i x_i (\betab^* - \betab) - v_i \varepsilon_i  \Bigl) \biggl].
 \end{equation}
 Moreover, whenever $\psi '$ exists we have 
 \[
   \EE_{\varepsilon} \biggl[ \psi \Bigl(- v_i x_i (\betab^* - \betab) - v_i \varepsilon_i  \Bigl) \biggl]
   =
 - v_i x_i (\betab^*-\betab)   \int_{-\infty}^{\infty} \psi'(\xi(u)) f(u) du.
 \]
 for $\xi (u)= \alpha (-v_i x_i (\betab^*-\betab)) + (1-\alpha) (- v_i u)$ for some $\alpha \in (0,1)$. When $\psi'$ doesn't exist we can decompose $\psi$ into a finite sum of step functions and then apply  exactly  the same technique on each of the step functions as in Lemma 2.  Hence, it suffices to discuss the differentiable case only.
 Let us denote the RHS of \eqref{eq:lambda} with $\Lambda_n^r(\betab) (\betab^* - \betab)$, i.e.
 \[
 \Lambda_n^r(\betab) = n^{-1} \sum_{i=1}^n -  \ind\{x_i\betab>0\}q_i v_i x_i^\top x_i 
 \int_{-\infty}^\infty \psi'(\xi(u)) f(u) du.
 \]
 Next, we observe that by 
 Condition (r$\boldsymbol\Gamma$),
 \[
 \left|  \int_{-\infty}^\infty \psi'(\xi(u)) f(u) du - \psi'(v_i \varepsilon_i) \right| \leq \sup_x |\psi'(x) | := C_1
 \]
 for a constant $C_1 <\infty$. With that the remaining steps of Lemma 2 can be completed with $\Sigmab$ replaced with $\Sigmab^r$.
 
 Next, by observing the proofs of Lemmas 3, 4  and 5 we see that the proofs remain to hold under 
  Condition (r$\boldsymbol\Gamma$),
and with $W$ replaced with $\tilde W$. The constants $KK_{\gamma}$ appearing in the simpler case will now be $K K_{\theta} M_1 M_2$. However, the rates remain the same up to these constant changes.

Next, we discuss Lemma 6. 
For the case of doubly robust estimator $\nu_n(\deltab)
$ of Lemma 6 takes the following form
\[
\tilde \nu_n(\deltab) = n^{-1} \sum_{i=1}^n \tilde \Omegab(\deltab + \betab^*)
[f_i(\deltab) \tilde g_i(\deltab) - f_i(0) \tilde g_i(0)]
\]
with $\tilde g_i(\deltab) = \psi(v_i (x_i \deltab +  \varepsilon_i))$. 
Moreover, $\EE_\varepsilon[ f_i(\deltab) \tilde g_i(\deltab)] = f_i(\deltab) \EE_{\varepsilon} [\psi(v_i (x_i \deltab +  \varepsilon_i))] := \tilde w_i(\deltab)$.
We consider the same covering sequence as in Lemma 6. Then, we observe
\[
\left| \tilde w_i(\tilde \deltab_k) - \tilde w_i(0) \right| \leq C_1 | x_i \tilde \deltab_k |.
\]
Furthermore, $\EE_X [f_i(\tilde \deltab_k) \EE_{\varepsilon} [\psi(v_i (x_i \tilde\deltab_k +  \varepsilon_i)) - f_i(0) \EE_{\varepsilon} [\psi(v_i (  \varepsilon_i))]^2$ $ \leq C_1 M_1 M_2  $ \phantom{JECAJECA} $ \left( G_i(\tilde \deltab_k,\betab^*,0) - G_i(0,\betab^*,0) \right) \Lambda_{\max} (\Sigmab(\betab^*))$, providing the bound of $T_1$ equivalent to that of Lemma 6. 

Term $T_2$ can be handled similarly as in Lemma 6. We illustrate the particular differences only in $T_{21}$ as others follows similarly.
 
 Observe that 
 \[
 f_i(\deltab) \tilde g_i(\deltab) = \ind\{x_i \deltab \geq - x_i \betab^*\} \psi(v (\varepsilon_i)) +\ind\{x_i \deltab \geq - x_i \betab^*\}  v_i x_i \deltab\psi'(\xi_{\deltab})
 \]
 for $\xi_{\deltab} =  v_i \varepsilon_i + (1-\alpha) v_i x_i \deltab$ for some $\alpha \in (0,1)$.
 Next, we consider the decomposition
 \[
  f_i(\deltab) \tilde g_i(\deltab)  - \EE\left[  f_i(\deltab) \tilde g_i(\deltab) \right] = T_{211}^r(\deltab) +T_{212}^r(\deltab)
 \]
 for $T_{211}^r(\deltab)  = T_{2111}^r(\deltab) + T_{2112}^r(\deltab)  $ and
 \[
 T_{2111}^r(\deltab)  = \left(  \ind\{x_i \deltab \geq - x_i \betab^*\} - \PP (x_i \deltab \geq - x_i \betab^*) \right) \psi(v_i \varepsilon_i)
 \]
  \[
 T_{2112}^r(\deltab)  =  \PP (x_i \deltab \geq - x_i \betab^*)   \psi(v_i \varepsilon_i)
 \]
 and
 \[
 T_{212}^r(\deltab) = \ind\{x_i \deltab \geq - x_i \betab^*\}  v_i x_i \deltab\psi'(\xi_{\deltab})
 - \EE\left[\ind\{x_i \deltab \geq - x_i \betab^*\}  v_i x_i \deltab\psi'(\xi_{\deltab}) \right]
 \]
 Furthermore, we observe that the same techniques developed in Lemma 6 apply to both of the terms of $T_{211}^r(\deltab) $ hence we only discuss the case of $T_{212}^r(\deltab) $. We begin by considering the decomposition $T_{212}^r(\deltab) =T_{2121}^r(\deltab) +T_{2122}^r(\deltab) $ with
 \[
 T_{2121}^r(\deltab) = \ind\{x_i \deltab \geq - x_i \betab^*\}  v_i x_i \deltab \left( \psi'(\xi_{\deltab}) - \EE_{\varepsilon} (\psi'(\xi_{\deltab}))\right)
 \]
 and
 \[
  T_{2122}^r(\deltab) = \ind\{x_i \deltab \geq - x_i \betab^*\}  v_i x_i \deltab  \EE_{\varepsilon} (\psi'(\xi_{\deltab})) 
  -
   \EE\left[\ind\{x_i \deltab \geq - x_i \betab^*\}  v_i x_i \deltab\EE_{\varepsilon}\psi'(\xi_{\deltab}) \right]
 \]
 Let us focus on the last expression as it is the most difficult one to analyze. 
 Observe that we are interested in the difference 
 $T_{2122}^r(\deltab)  - T_{2122}^r(\tilde \deltab_k)  $. We decompose this difference into four terms, two related to random variables and two related to the expectations. We handle them separately and observe that because of symmetry and monotonicity of the indicator functions once we can bound the difference of random variables we can repeat the arguments for the expectations. Hence, we focus on 
 \[
 I_1 = \ind\{x_i \deltab \geq - x_i \betab^*\}  v_i x_i \deltab  \EE_{\varepsilon} (\psi'(\xi_{\deltab})) - \ind\{x_i \tilde \deltab_k \geq - x_i \betab^*\}  v_i x_i \tilde \deltab_k  \EE_{\varepsilon} (\psi'(\xi_{\tilde \deltab_k})) .
 \]
 First due to monotonicity of indicators and \eqref{eq:bound1} we have 
 \[
 |I_1| \leq I_{11} + I_{12} + I_{13}
 \]
 with
 \begin{align*}
 I_{11} &= \left( \ind\{x_i \tilde \deltab_k + \tilde L_n  \geq - x_i \betab^*\} -\ind\{x_i \tilde \deltab_k \geq - x_i \betab^*\}  \right) v_i x_i \tilde \deltab_k   \EE_{\varepsilon} (\psi'(\xi_{\tilde \deltab_k})) 
\\
I_{12} &=
  \ind\{x_i \tilde \deltab_k + \tilde L_n  \geq - x_i \betab^*\} \tilde L_n \EE_{\varepsilon} (\psi'(\xi_{  \deltab })) 
  \\
I_{13} &= \ind\{x_i \tilde \deltab_k + \tilde L_n  \geq - x_i \betab^*\} v_i x_i \tilde \deltab_k \left(  \EE_{\varepsilon} (\psi'(\xi_{ \deltab }))  -\EE_{\varepsilon} (\psi'(\xi_{\tilde \deltab_k}))  \right)
 \end{align*}
 As $\sup \psi' < \infty$, $I_{11}$ can be handled in the same manner as $T_{21}$ of the proof of Lemma 6 whereas $I_{12} = \Ocal_P ( \tilde L_n )$. For $I_{13}$ it suffices to discuss the difference at the end of the right hand side of its expression.
 However, it is not difficult to see that 
 \[
  \EE_{\varepsilon} (\psi'(\xi_{ \deltab }))  -\EE_{\varepsilon} (\psi'(\xi_{\tilde \deltab_k}))  \leq 4 C v_i  \tilde L_n \leq 4 C M_1 \tilde L_n
 \]
 with $C = \sup_x|\psi''(x)|$ for the case of twice differentiable $\psi$, $C =  \sup_{y} \partial/ \partial y | \int_{-\infty} ^y \psi ' (x) dx |$ for the case of once differentiable $\psi$ and $C = f_{\max}$ for the case of non-differentiable functions $\psi$.
Combining all the things together we observe that the rate of Lemma 6 for the case of doubly robust estimators is of the order of 
\[
C \rbr{
\sqrt{\frac{  (r_n M_3 \vee r_n^2K_1^2 M_1 M_2) t 
\log(n p /\delta)}{n}} 
\bigvee
\frac{  t \log(2n p /\delta)}{n} }.
\]
with $M_3 = \sup _{x}|\psi'(x)|$
 for once differentiable $\psi$ and $M_3=f_{\max}$ for non-differentiable $\psi$.
 
Now, with equivalents of Lemmas 1-6 are established, 
 we can use them to bound successive terms in the Bahadur representation much like those of Theorem 1. Details are ommitted due to space considerations.

For Theorem \ref{cor:uniform:ci:b}, the same line of the proof of Theorem \ref{cor:uniform:ci} applies, but only replace the matrix $\Sigmab$ with the matrix $\Sigmab^{\mbox{r}}$. The result of the Theorem then follows from the arguments in Remark \ref{remark4}. Uniformity of the obtained results is not compromised as the weight functions $q_i$ and $v_i$ only depend on the design matrix.
\end{proof}

 \section{Proofs of Lemmas}
 
 \begin{proof}[Proof of Lemma \ref{lemma0}]

Let $\{\tilde \deltab_k\}_{k
  \in [N_\delta]}$ be the centers of the balls of radius $r_n\xi_n$ that
cover the set $\Ccal(r_n,t)$. Such a cover can be constructed with
$N_\delta \leq {p \choose t} (3/\xi_n)^t $ \citep[see, for example  ][]{vdvaart_1998}. 
Furthermore, let  $\mathbb{D}_n(\deltab)=n^{-1} \sum_{i=1}^n \left[ \mu_i(\deltab) -  \EE [\mu_i(\deltab)] \right]$ and let
\begin{equation*}
\mathcal{B} (\tilde \deltab_k, r) = 
\cbr{
\deltab \in \RR^p : 
\norm{\tilde \deltab_k-\deltab}_2 \leq r
\ , \ 
{\mbox{supp}}(\deltab) \subseteq {\mbox{supp}}(\tilde \deltab_k)
}
\end{equation*}
be a ball of radius $r$ centered at $\tilde \deltab_k$ with  elements
that have the same support as $\tilde \deltab_k$. In what follows, we
will bound $ \sup_{\deltab \in
  \Ccal(r_n, t)} |\mathbb{D}_n(\deltab)|$ using an
$\epsilon$-net argument.  In particular, using the above introduced
notation, we have the following decomposition
\begin{equation}
\label{eq:proof:lem0_lin:t1t2_1}
\begin{aligned}
\sup_{\deltab \in \Ccal(r_n, t)}
| \mathbb{D}_n(\deltab)| &=  \max_{k \in [N_\delta]}
\sup_{\deltab \in \mathcal{B} (\tilde \deltab_k, r_n\xi_n)}
| \mathbb{D}_n( \deltab)|  \\
& \leq
\underbrace{
  \max_{k \in [N_\delta]}
 |\mathbb{D}_n( \tilde \deltab_k)|
  }_{T_1}
 + 
\underbrace{
  \max_{k \in [N_\delta]}
  \sup_{\deltab \in \mathcal{B} (\tilde \deltab_k, r_n\xi_n)}
  | \mathbb{D}_n( \deltab) -  \mathbb{D}_n( \tilde \deltab_k)|
}_{T_2}.
\end{aligned}
\end{equation}
%
%

We first bound the term $T_1$ in \eqref{eq:proof:lem0_lin:t1t2_1}.
  To that end, let
$
  Z_{ik}  = 
  \rbr{
     \mu_{i}( \tilde\deltab_k)   - 
      \EE \left[ \mu_{i}( \tilde\deltab_k)    \right]}.  
$

With a  little abuse of notation we use $l$ to denote the density of $x_i\betab^*$ for all $i$. Observe,  
$$\EE \left[\mu_i(  \deltab)    \right] = \PP\biggl( x_i \betab^* \leq x_i \deltab       \biggl).$$ Let $w_{i}(\deltab)  $  denote the probability on the right hand side of the previous equation, as a function of $\deltab$.
Then
$
  T_1  = 
    \max_{k \in [N_\delta]}
    \abr{
      n^{-1}\sum_{i\in[n]}  Z_{ik}  }   .
$
Note that $\EE[Z_{ik} ] = 0$ and 
\begin{align*}
  \Var&[Z_{ik} ] 
    =   
    w_{i}( \tilde\deltab_k)
    - w_{i}^2( \tilde\deltab_k)      \stackrel{(i)}{\leq} 
    w_{i}( \tilde\deltab_k)
      \stackrel{(ii)}{\leq}  
    \abr{
  x_i\tilde\deltab_k  
  }
  l_i\rbr{   c_ix_i \tilde\deltab_k    }
  \stackrel{(iii)}{\leq} 
 \abr{
   x_i \tilde\deltab_k  
  }
  K_1\quad \rbr{c_i  \in [0,1]}
\end{align*}
where $(i)$ follows by dropping a negative term, $(ii)$ follows by
the mean value theorem  and  $(iii)$ from the  Condition (E).  Hence,   we have that almost surely,
$|Z_{ik}| \leq C \max_{i}\abr{
   x_i \tilde\deltab_k  
  }$ for a constant $C<\infty$.
 For a
fixed $k$, Bernstein's inequality \citep[see, for
example, Section 2.2.2 of][]{vanderVaart1996Weak} gives us
\begin{align*}
  \abr{n^{-1}\sum_{i\in[n]} Z_{ik}}
  \leq C\rbr{
  \sqrt{
    \frac{f_{\max}\log(2/\delta)}{n^2}
    \sum_{i\in[n]} 
     \abr{
  x_i\tilde\deltab_k 
  }
  }
  \bigvee \frac{ \log(2/\delta) }{n}
  }
\end{align*}
with probability $1-\delta$. 
Observe that $\tilde\deltab_k^\top W(\betab^*+\tilde\deltab_k)^\top W(\betab^*+\tilde\deltab_k)\tilde\deltab_k \leq \tilde\deltab_k^\top W(\betab^*)^\top W(\betab^*)\tilde\deltab_k + \tilde\deltab_k^\top X^\top [\ind(X (\betab^* +\tilde \delta_k) \geq 0) - \ind(X \betab^* \geq 0)] X \tilde\deltab_k 
\leq \tilde\deltab_k^\top W(\betab^*)^\top W(\betab^*)\tilde\deltab_k + 2 \tilde\deltab_k^\top X^\top X \tilde\deltab_k  $.
Hence,
\begin{align*}
    \sum_{i\in[n]} 
    \abr{
  x_i\tilde\deltab_k  
  }
    & \leq C  ^2 \sqrt{n}\sqrt{\tilde\deltab_k^\top  W^\top (\betab^*+\tilde\deltab_k) W (\betab^*+\tilde\deltab_k)  \tilde\deltab_k} \leq 2 C ^2 r_n \sqrt{n}\left(\Lambda_{\max}^{1/2}(\Sigmab (\betab^*)) \vee1\right),
\end{align*}
where the line follows using the Cauchy-Schwartz inequality and
inequality (58a) of \citet{wainwright_2009} and Lemma \ref{lemma2}.  
Hence,
with probability $1-2\delta$ we have for all $\lambda_j \geq A \sqrt{\log p /n}$ that
\begin{align*}
  \abr{n^{-1}\sum_{i\in[n]}Z_{ik}}
  \leq C\rbr{
  \sqrt{
    \frac{ r_n\log(2/\delta)}{n}
  }
  \bigvee \frac{ \log(2/\delta)}{n} }.
\end{align*} 
 
Using the union bound over 
$k\in[N_\delta]$, with probability $1-2\delta$, we have
$$
T_{1} \leq C \rbr{
  \sqrt{
    \frac{ r_n \sqrt{t} \log(2  
      N_\delta   /\delta)}{n}
  } 
  \bigvee \frac{ \log(2 
      N_\delta  /\delta)}{n} }.
$$
Let us now focus on bounding $T_2$ term.   Let 
$
Q_{i}( \deltab) =   
         \mu_{i}(  \deltab )    -  \EE \mu_i(\deltab) .
$
For a fixed   $k$ we have
\begin{align*}  
  \sup_{\deltab \in \mathcal{B} (\tilde \deltab_k, r_n\xi_n)}
  \abr{   
       \mathbb{D}_n(  \deltab) -  \mathbb{D}_n( 
      \tilde\deltab_k) 
  }
    &\leq  \sup_{\deltab \in \mathcal{B}(\tilde \deltab_k, r_n\xi_n)}
  \abr{
    n^{-1}\sum_{i\in[n]}
    {Q}_{i}(\deltab)
    - {Q}_{i}(\tilde\deltab_k)}: = T_{21}.
\end{align*}
 
Let $Z_i=x_i \betab^*$.  Observe that the density of $Z_i$ is  by Condition ({\bf E}) very close to the distribution of  $\varepsilon_{i}$.
 Moreover,
 \[\abr{x_i(\deltab-\tilde\deltab_k)} \leq K
\norm{\deltab-\tilde\deltab_k}_2
\sqrt{\abr{\mbox{supp}(\deltab-\tilde\deltab_k)}}   
\]
where $K$ is a constant such that    $\max_{i,j}|x_{ij}| \leq K$.
Hence,  
\begin{align*}
 &\max_{k \in [N_\delta]}\max_{i\in[n]}
  \sup_{\deltab \in \mathcal{B}(\tilde \deltab_k, r_n\xi_n)}
  \abr{
     x_i  \deltab
    -
      x_i  \tilde\deltab_k
  } \leq 
    r_n\xi_n\sqrt{t} \max_{i,j}|x_{ij}|
  \leq C
    r_n\xi_n
    \sqrt{ t  }
 =: \tilde L_n,
\end{align*}
For $T_{21}$, we will use
the fact that $\ind\{a < x\}$ and $\PP\{Z < x\}$ are monotone functions in $x$. 
Therefore,
\begin{align*}
T_{21} 
&\leq 
n^{-1}\sum_{i\in[n]}\bigg[   
  \ind\cbr{Z_i \geq   x_i  \tilde\deltab_k - \tilde L_n }  -
  \PP\sbr{Z_i \geq  x_i  \tilde\deltab_k +\tilde L_n }  \bigg] \\
&\leq
n^{-1}\sum_{i\in[n]}\bigg[   
  \ind\cbr{Z_i \geq   x_i  \tilde\deltab_k - \tilde L_n} -
  \PP\sbr{Z_i \geq  x_i  \tilde\deltab_k - \tilde L_n  }  
\bigg] \\
&\ +
n^{-1}\sum_{i\in[n]}\bigg[   
  \PP\sbr{Z_i \geq  x_i  \tilde\deltab_k - \tilde L_n } -
  \PP\sbr{Z_i \leq  x_i  \tilde\deltab_k+ \tilde L_n }
\bigg].
\end{align*}
The first term in the display above can be bounded in a similar way to
$T_1$ by applying Bernstein's inequality and hence the details are
omitted. For the second term, we have a bound
$C \tilde L_n  $, since
$  \PP\sbr{Z_i \geq  x_i  \tilde\deltab_k - \tilde L_n} - \PP\sbr{Z_i \leq  x_i  \tilde\deltab_k+ \tilde L_n }  \leq 2C f_{\max} \tilde L_n  $, per Condition ({\bf E}).  
Therefore, with probability $1-2\delta$,
\begin{align*}
  T_{21} 
  \leq C\rbr{
  \sqrt{
    \frac{ \tilde L_n \log(2/\delta)}{n}
  }
  \bigvee \frac{\log(2/\delta)}{n} 
  \bigvee     \tilde L_n }.
\end{align*} 
 A bound on $T_2$ now follows using a union bound over  $k\in[N_\delta]$.
We can choose   $\xi_n = n^{-1}$, which gives us
$N_\delta \lesssim \rbr{pn^2}^t $. With these choices, we obtain 
$
T  \leq C  \rbr{
\sqrt{\frac{   r_n   t   \sqrt{t} 
\log(n p /\delta)}{n}} 
\bigvee
\frac{  t \log(2n p /\delta)}{n} },
$
which completes the proof.

\end{proof}

\begin{proof}[Proof of Lemma \ref{lemma1}]
We begin by rewriting the term $n^{-1}\sum_{i=1}^n\psi_i(\betab)$, and  aim  to represent it through indicator functions.  
Observe that 
\begin{align}\label{psi_beta}
n^{-1}\sum_{i=1}^n\psi_i(\betab) = n^{-1}\sum_{i=1}^n x_i^\top \ind(x_i\betab>0)[1-2\cdot\ind(y_i-x_i\betab<0)].
\end{align}
Using the fundamental theorem of calculus, we notice that if $x_i\betab^*>0$,
$\int_{x_i(\betab-\betab^*)}^0 f(\epsilon_i)d\varepsilon_i = F(0)-F(x_i(\betab-\betab^*))
= \frac{1}{2}-P(y_i<x_i\betab)$,
where $F$ is the univariate distribution of $\varepsilon_i$.  Therefore,  with expectation on $\varepsilon$, we can obtain an expression without the $y_i$.
\begin{align*}
n^{-1}\sum_{i=1}^n\EE_\varepsilon \psi_i(\betab) 
&= \left[n^{-1}\sum_{i=1}^n x_i^\top \ind(x_i\betab>0)\cdot2\int_{x_i(\betab-\betab^*)}^0 f(u)du \right]\\
&= \left[n^{-1}\sum_{i=1}^n x_i^\top \ind(x_i\betab>0)\cdot2f(u^*)x_i(\betab^*-\betab) \right]
:= \Lambda_n(\betab)(\betab^*-\betab),
\end{align*} 
for some $u^*$ between 0 and $x_i(\betab^*-\betab)$, and where we have defined 
$$\Lambda_n(\betab)  = \left[n^{-1}\sum_{i=1}^n\ind(x_i\betab>0) x_i^\top  x_i\cdot2f(u^*)\right].$$

We then show a bound for $\Delta := \left|\left[ \EE_X \Lambda_n(\betab)-2f(0) \Sigmab( {\betab}^*)\right]_{jk}\right|$, where we recall $\Sigmab(\betab^*)$ is defined as earlier, $\Sigmab( {\betab}^*)=n^{-1}\sum_{i=1}^n \EE_X \ind(x_i {\betab}^*>0)x_i^\top x_i$. By triangular inequality,
\begin{align}
\Delta & \leq \left| n^{-1}\sum_{i=1}^n\EE_X \ind(x_i\betab>0)x_{ij} x_{ik}\cdot 2(f(u^*)-f(0)) \right|\label{Delta1}\\
&\qquad +\left|n^{-1}\sum_{i=1}^n\EE_X \ind(x_i\betab>0)x_{ij} x_{ik}\cdot 2f(0)-n^{-1}\sum_{i=1}^n \EE_X \ind(x_i\betab^*>0)x_{ij} x_{ik}\cdot2f(0)\right|.\label{Delta2}
\end{align}
Notice that $\ind(x_i\betab>0)-\ind(x_i {\betab}^*>0)\leq\ind(x_i\betab\geq 2x_i {\betab}^*)=\ind[x_i {\betab}^*\leq x_i(\betab- {\betab}^*)] $. Moreover, the original expresion is also smaller than or equal to $\ind\left(|x_i {\betab}^*|\leq|x_i(\betab- {\betab}^*)|\right)$.
 The   term \eqref{Delta2} can be bounded by Condition {\bf{(X)}} and {\bf{(E)}},
\begin{align*}
\left|n^{-1}\sum_{i=1}^n\EE_X \ind(x_i\betab>0)x_{ij} x_{ik}\cdot2f(0)-n^{-1}\sum_{i=1}^n \EE_X \ind(x_i\betab^*>0)x_{ij} x_{ik}\cdot2f(0)\right|\\
\leq 2f(0) K ^2 n^{-1}\sum_{i=1}^n  \EE_X \ind\left(|x_i {\betab}^*|\leq \|x_i \|_\infty \|(\betab- {\betab}^*)\|_1\right) \leq 2 f(0) K^2 \|(\betab- {\betab}^*)\|_1. 
\end{align*} 
With the help of H\"older's inequality, 
$\left| \eqref{Delta1}  \right| 
 {\leq} n^{-1}\sum_{i=1}^n\EE_X \ind(x_i\betab>0)\|x_i\|_\infty^2\cdot2\left|f(u^*)-f(0)\right|  .$
By triangular inequality and Condition {\bf{(E)}} we can further upper bound the right hand side with
\begin{align*}
   2\cdot n^{-1}\sum_{i=1}^n\EE_X \|x_i\|_\infty^2\cdot L_0\|x_i\|_\infty\|\betab-\betab^*\|_1 .
\end{align*}
 Then we are ready to put terms together and obtain a bound for $\Delta$. Additionally, by Condition {\bf{(X)}} we have
$$
\Delta \stackrel{\text{}}{\leq} (C  +2 f(0))K ^3  \|\betab-\betab^*\|_1,
$$
for $\|\betab-\betab^*\|_1<\xi$ and a constant $C$. 
 Essentially, this proves that $\Delta$ is not greater than a constant multiple of the difference between $\betab$ and $\betab^*$.   Thus, we have as $n\rightarrow\infty$
\begin{equation}\label{E_psi}
n^{-1}\sum_{i=1}^n\EE\psi_i(\betab) = n^{-1}\sum_{i=1}^n\EE_X \EE_\varepsilon \psi_i(\betab) = 2f(0) \Sigmab( {\betab}^*)(\betab^*-\betab) + \mathcal{O}(\|\betab-\betab^*\|_1)(\betab^*-\betab).
\end{equation}
\end{proof}

\begin{proof}[Proof of Lemma \ref{cor:2}]

For the simplicity in notation we fix $j=1$ and denote $\hat \gammab_{(1)}(\hat \betab)$ with $\hat \gammab (\hat \betab)$.
The proof is composed of two steps: the first establishes a cone set and an event set of interest whereas the second proves  the rate of the estimation error by certain approximation results.

{\bf Step 1}. Here we show that the estimation error $\hat \gammab - \gammab^*$ belongs to the appropriate cone set with high probability. 
We introduce the loss function  $l(\betab,\gammab) = n^{-1} \sum_{i=1}^n \left( W_{i,1} (\betab) - W_{i,-1} (\betab) \gammab \right)^2$.
The loss function above is convex in $\gammab$ hence 
\[
(\hat \gammab - \gammab^*) \left[ \nabla_{\gammab} l(\hat \betab, \gammab)|_{\gammab=\hat \gammab} - \nabla_{\gammab} l(\hat \betab,\gammab )|_{\gammab= \gammab^*} \right] \geq 0.
\]
  Let $h^*= \left \| \nabla_{\gammab} l(\hat \betab,  \gammab)|_{\gammab=  \gammab^*} \right  \|_\infty $.
Let $\deltab= \hat \gammab - \gammab^*$.
KKT conditions provide $ \left(\nabla_{\gammab} l(\hat \betab,  \gammab)|_{\gammab=\gammab^* + \deltab}\right)_j = - \lambda_1 \mbox{sgn}(\gammab^*_j + \deltab_j)$ for all $j \in S_{1}^c \cap \{\hat \gammab_j \neq 0\}$ with $S_1 = \{j: \gammab^* \neq 0\}$.
Moreover, observe that $\deltab_j=0$ for all $j \in S_1^c \cap \{\hat \gammab_j = 0\}$.
Then,
\begin{align*}
&(\hat \gammab - \gammab^*) \left[ \nabla_{\gammab} l(\hat \betab,  \gammab)|_{\gammab=\hat \gammab} - \nabla_{\gammab} l(\hat \betab,\gammab )|_{\gammab= \gammab^*} \right] 
\\
&=
\sum_{j \in S_1^c} \deltab_j (\nabla_{\gammab} l(\hat \betab,  \gammab)|_{\gammab=\gammab^* + \deltab} )_j 
+
\sum_{j \in S_1} \deltab_j (\nabla_{\gammab} l(\hat \betab, \gammab)|_{\gammab=\gammab^* + \deltab} )_j 
+
\deltab^\top (-  \nabla_{\gammab} l(\hat \betab,\gammab )|_{\gammab= \gammab^*})
\\
&\leq 
\sum_{j \in S_1^c} \deltab_j (- \lambda_1 \mbox{sgn}(\gammab^*_j + \deltab_j))
+
\lambda_1 \sum_{j \in S_1} |\deltab_j|
+
h^* \|\deltab \|_1
\\
&=
\sum_{j \in S_1^c}  - \lambda_1 |\deltab_j|
 +
 \sum_{j \in S_1}   \lambda_1 |\deltab_j|
 +
 h^* \|\deltab_{S_1} \|_1
 +
 h^* \|\deltab_{S_1^c} \|_1
 \\
 &=
 (h^* - \lambda_1)  \|\deltab_{S_1^c} \|_1 + (\lambda_1 + h^*) \|\deltab_{S_1} \|_1.
\end{align*}
 
Hence on the event $h^* \leq (a-1)/ (a+1) \lambda_1 $ for a constant $a>1$, the estimation error $\deltab$ belongs to the cone set 
\begin{equation}\label{eq:cone}
\mathcal{C}(a,S_1)= \{ \mathbf x \in \RR^{p-1}: \|\xb_{S_1^c} \|_1 \leq a \|\xb_{S_1} \|_1\}
\end{equation}

Next, we proceed to show that the event above holds with high probability for certain choice of the tuning parameter $\lambda_1$.
We begin by decomposing 
 \[
 h^*
 \leq  
 \left \| \nabla_{\gammab} l(\betab^*,  \gammab)|_{\gammab=  \gammab^*} \right  \|_\infty  
 +
  \left \| \nabla_{\gammab} l(\betab^*,  \gammab)|_{\gammab=  \gammab^*}   
-  \nabla_{\gammab} l(\hat \betab,  \gammab)|_{\gammab=  \gammab^*} \right  \|_\infty  
 \]

Let $H_1= \nabla_{\gammab} l(\betab^*,  \gammab)|_{\gammab=  \gammab^*}$
and let $H_2= \nabla_{\gammab} l(\betab^*,  \gammab)|_{\gammab=  \gammab^*}   
-  \nabla_{\gammab} l(\hat \betab,  \gammab)|_{\gammab=  \gammab^*} $
We begin by observing that
$
\nabla_{\gammab} l(\hat \betab,  \gammab)|_{\gammab=  \gammab^*} 
=
  \nabla_{\gammab} l( \betab^*,  \gammab)|_{\gammab=  \gammab^*} 
  + \Delta_1 + \Delta_2 + \Delta_3 + \Delta_4
$, for
\begin{align*}
\Delta_1 
&= -2 n^{-1}  \left( W_{-1}(\hat \betab) - W_{-1}(\betab^*)\right)^\top W_{1}(\hat \betab)
\\
\Delta_2 
&= -2 n^{-1} \left( W_{-1} (\betab^*) \right)^\top \left( W_{1}(\hat \betab) - W_{1}(\betab^*)\right)
\\
\Delta_3 
&= -2 n^{-1} \left( W_{-1}(\hat \betab) \right)^\top \left( W_{-1}(\hat \betab) - W_{-1}(\betab^*)\right) \gammab^*
\\
\Delta_4
&= 2 n^{-1} \left( W_{-1}(\hat \betab) - W_{-1}(\betab^*)\right)^\top  W_{-1}(\betab^*)\gammab^*
\end{align*}

Next, by Lemma \ref{lemma0} we observe 
\[
| \Delta_{1,j} | \leq   2 K^2 n^{-1}  \left| \sum_{i=1}^n   \mu_i(\betab^*-\hat \betab ) - \mu_i(0)\right | = \ocal_P \left( K^2 r_n^{1/2} t^{3/4} (\log p / n)^{1/2} \bigvee K^2 t \log p / n \right) 
\]
 and similarly $|\Delta_{2,j}| =\ocal_P \left( K^2 r_n^{1/2} t^{3/4} (\log p / n)^{1/2} \bigvee K^2 t \log p / n \right)$.  
Recall that $s_1 = \|\gammab ^* \|_0$.  Let $K_{\gamma}$ be defined as $\|\gammab^*\|_\infty \leq K_{\gamma}$. Then, by H\"older's inequality
\begin{align*}
&| \Delta_{3,j} |\leq 2 K^2 K_\gamma n^{-1}  \left| \sum_{i=1}^n   \sum_{k \in S_1} |X_{ik}| \left[ \mu_i(\betab^*-\hat \betab) - \mu_i(0)\right] \right | \\
&\qquad = \ocal_P \left( K^3 K_\gamma r_n^{1/2} t^{3/4} s_1 (\log p / n)^{1/2} \bigvee K^3 K_\gamma t s_1 \log p / n \right)
\end{align*}
and similarly $| \Delta_{4,j} |=  \ocal_P \left( K^3 K_\gamma r_n^{1/2} t^{3/4} s_1 (\log p / n)^{1/2} \bigvee K^3 K_\gamma t s_1 \log p / n \right) $.
Putting all the terms together we obtain $$H_2 =  \ocal_P \left( K^2 (1 \vee K K_\gamma) r_n^{1/2} t^{3/4} s_1 (\log p / n)^{1/2} \bigvee K^2 (1 \vee K K_\gamma) t s_1 \log p / n \right)$$.

Next, we focus on the term $H_1$. Simple computation shows that for all $k =2,\cdots p$, we have
\[
H_{1,k} = - 2 n^{-1} \sum_{i=1}^n u_i
\]
for $u_i =X_{ik}  \zetab_{1,i}^* \ind\{ x_i \betab^* >0\} $. Observe that the sequence $\{u_i\}$ across $i=1,\cdots, n$,  is a sequence of independent random variables. As $\varepsilon_i$ and $x_i$ are independent we have by the tower property $\EE[r_i]= \EE_X \left[X_{ik} \ind\{ x_i \betab^* >0\}   \EE_\varepsilon[ \zetab_{1,i}^*] \right] =0$. Moreover,  as $\zetab^*_1$ is sub-exponential random vector, by Bernstein's inequality and union bound we have 
\[
P \left( \|H_1 \|_\infty \geq c  \right) \leq p \exp \left\{ -\frac{n}{2} \left(\frac{c^2}{\tilde K ^2} \vee \frac{c}{\tilde K} \right)  \right\}
\]
where $\|u_i \|_{\psi_1} \leq K \| \zetab_{1,i}^* \|_{\psi_1}  := \tilde K< \infty$. We pick $c$ to be $(\log p / n)^{1/2}$, then we have with probability converging to $1$ that 
\begin{align*}
h^* &\leq \|H_1\|_\infty + \|H_2\|_\infty \leq (\log p / n)^{1/2} + C_1 r_n^{1/2} t^{3/4} s_1 (\log p / n)^{1/2} + C_2 t s_1 \log p / n \\
&\leq (a-1)/(a+1) \lambda_1,
\end{align*} 
for some constant $C_1$ and $C_2$. Thus, with $\lambda_1$ chosen as 
$$\lambda_1 = C \left( (\log p / n)^{1/2} \bigvee \left(r_n^{1/2} \bigvee t^{1/4} (\log p / n)^{1/2} \right) t^{3/4} s_1(\log p / n)^{1/2}\right),$$ 
for some constant $C > 1$, we have that $h^* \leq (a-1)/(a+1) \lambda_1$ with probability converging to $1$. More directly, with the condition on the penalty parameter $\lambda_1$, this implies that the event for the cone set \eqref{eq:cone} to be true holds with high probability.

\vskip 5pt 
{\bf Step 2}.
We begin by a basic inequality
\[
l(\hat \betab ,\hat\gammab) + \lambda_1 \|\hat\gammab \|_1 \leq  l( \hat \betab ,  \gammab^*)  + \lambda_1 \| \gammab^*\|_1
\]
guaranteed as  $\hat\gammab$ minimizes the penalized loss \eqref{nodewise_lasso}.
Here and below in the rest of the proof we suppress the subscript $1$ and $\betab$ in the notation of $W_1(\hat \betab)$ and $W_{-1} (\hat \betab)$ and use $\hat W$ and $\hat W^{-}$ instead and  similarly $  W^*:= W_1(\betab^*)$ and ${W^{-}}^* =W_{-1}(\betab^*)$.
Rewriting the inequality above we obtain
\begin{align*}
&-2 n^{-1} \hat W^\top \hat{W^{-}} \hat \gammab + n^{-1} \hat \gammab^\top {\hat{W^{-}} }^\top \hat{W^{-}} \hat \gammab
\\
&\leq 
-2 n^{-1} \hat W^\top \hat{W^{-}} \gammab^*
 + n^{-1} {  \gammab^*}^\top {\hat{W^{-}} }^\top  \hat{W^{-}} \gammab^*
-\lambda_1 \|\hat\gammab \|_1 +\lambda_1 \| \gammab^* \|_1
\end{align*}
 
Observe that $W_{ij}(\hat \betab) = W_{ij}(\betab^*) + X_{ij} [\mu_i(\betab^*-\hat\betab) - \mu_i(0)]$. Let $\alpha_{ij} = X_{ij} [\mu_i(\betab^*-\hat\betab) - \mu_i(0)]$. 
Let $\Ab$ be a matrix such that $\Ab = \{ \alpha_{ij}\}_{1\leq i \leq n, 1 \leq j \leq p}$. From now on we only consider $\Ab$ to mean $\Ab_1$ and $\Ab^{-}$ to mean $\Ab_{-1}$.
Next, note that $W_i^* = {W_i^{-} }^* \gammab^* + \zeta_i^*$ by the node-wise plug-in lasso problem \eqref{eq:optimization_pop}.
Together with the above, we observe that then 
$\hat W_ i = {W_i^{-} }^* \gammab^* + \zeta_i^* + \Ab_i : = {W_i^{-} }^* \gammab^* + \varepsilon^*_i$. Hence,   the basic inequality above becomes,
\begin{align*}
&-2 n^{-1} \left( {W^{-} }^*   \gammab^* + \boldsymbol\varepsilon^* \right)^\top ( {W^{-}}^*  + \Ab^{-}) \hat \gammab + n^{-1} \hat \gammab^\top ( {W^{-}}^*  + \Ab^{-})^\top ( {W^{-}}^*  + \Ab^{-}) \hat \gammab
\\
&\leq 
-2 n^{-1} \left( {W^{-} }^* \gammab^* +\boldsymbol\varepsilon^* \right)^\top ( {W^{-}}^*  + \Ab^{-}) \gammab^* + n^{-1} \gammab^{*\top} ( {W^{-}}^*  + \Ab^{-})^\top ( {W^{-}}^*  + \Ab^{-})  \gammab^* \\
 &
-\lambda_1 \|\hat\gammab \|_1 +\lambda_1 \| \gammab^* \|_1.
\end{align*}
With reordering the terms in the inequality above, we obtain
\begin{align*}
&  n^{-1} \left\| {W^{-} }^* \hat \gammab - {W^{-} }^* \gammab^*\right\|_2^2 
 \leq 
 \delta_1 + \delta_2 + \delta_3  
-\lambda_1 \|\hat\gammab \|_1 +\lambda_1 \| \gammab^* \|_1,
\\[2ex]
\mbox{for} \qquad &\delta_1=2 n^{-1} \varepsilon_1^{*\top} \left({W^{-}}^* + \Ab^- \right) \left( \hat \gammab - \gammab^*\right), &\\
&\delta_2 =  2 n^{-1} \gammab^{*\top}  {W^{-}}^{*\top} \Ab^- \left( \hat \gammab - \gammab^*\right),&\\
&\delta_3 = n^{-1}  \left( \gammab^* + \hat \gammab \right)^\top \left( \Ab^{-\top} \Ab^- + 2W^{-*\top} A^-\right)  \left( \gammab^* - \hat \gammab \right).&
\end{align*}
 
Next, we observe that $A_i$ are bounded, mean zero random variables and hence $n^{-1}| \sum_{i=1}^n A_i | = \Ocal_P(n^{-1/2})$. Moreover $\varepsilon_i^*$ is a sum of sub-exponential and bounded random variables, hence is sub-exponential. 
Thus, utilizing the above and results of Step 1 we obtain
$$
\delta_1 \leq K^2 (a+1) \| \hat \gammab_{S_1} - \gammab^*_{S_1} \|_{1} \Ocal_P(n^{-1/2}),
$$
$$ \delta_2 \leq K^2 (a+1)  \| \hat \gammab_{S_1} - \gammab^*_{S_1} \|_{1}  \|\gammab^*_{S_1} \|_{1} \Ocal_P(n^{-1/2}),$$ 
 
Lastly,  observe that 
\begin{align}
\delta_3 \leq n^{-1} {\gammab^*}^\top \left( \Ab^{-\top} \Ab^- + 2W^{-*\top} A^-\right)   \gammab^*  + n^{-1} {\hat \gammab }^\top \left( \Ab^{-\top} \Ab^- + 2W^{-*\top} A^-\right) \hat   \gammab 
\end{align}
Moreover, as $\hat\gammab - \gammab^*$ belongs to the cone $C(a,S_1)$ \eqref{eq:cone} by Step 1, by convexity arguments it is easy to see that $\hat \gammab$ belongs to the same cone. Together with H\"older's inequality we obtain
\[
\delta_3 \leq 3 K n^{-1} \sum_{i=1}^n  { W_{i,S_1}^{-*}} ^\top  {\Ab_{i,S_1}^{-}} \left[ \|  \gammab^* _{S_1}\|_2^2+ \|\hat\gammab_{S_1} \|_2^2\right]
\]
Utilizing Lemma \ref{lemma0} now provides 
\[
\delta_3 \leq \kappa\left[ \|  \gammab^* _{S_1}\|_2^2+ \|\hat\gammab_{S_1} \|_2^2\right]
\]
where $\kappa$ is such that $\kappa = \Ocal_P ( K^3 r_n^{1/2} t^{3/4} s_1 (\log p /n)^{1/2} )$.
Moreover, observe that if $\lambda_1$ is chosen to be larger than the upper bound of $\kappa$. Putting all the terms together we obtain
\begin{align*}
 n^{-1} \sum_{i=1}^n \left( {W_i^{-} }^* \hat \gammab - {W_i^{-} }^* \gammab^*\right)^2 
 &\leq 
2  \lambda_1 \| \hat \gammab_{S_1} - \gammab^*_{S_1} \|_{1}  
+ \lambda_1 \|  \gammab^*_{S_1} \|_2^2 +  \lambda_1 \| \hat  \gammab _{S_1} \|_2^2 -\lambda_1 \|\hat\gammab \|_1 +\lambda_1 \| \gammab^* \|_1\\
&\leq
3  \lambda_1 \| \hat \gammab_{S_1} - \gammab^*_{S_1} \|_{1}  
+ \lambda_1 \|  \gammab^*_{S_1} \|_2^2 +  \lambda_1 \| \hat  \gammab _{S_1} \|_2^2  
\end{align*}
where the last inequality holds as $|\hat \gamma_j - \gamma_j^* | + |\gamma_j^*| - |\hat \gamma_j|$  for $j \in S_{1}$.

Moreover, by Condition {\bf (C)} and Step 1 we have that the left hand side   is bigger than or equal to $C_2  n^{-1} \sum_{i=1}^n \left( {X_i^{-} }  \hat \gammab - {X_i^{-} }  \gammab^*\right)^2 
 $, allowing us to conclude
\begin{align}\label{eq:temp11}
& 
n^{-1} C_2  \left\| X   (\hat \gammab - \gammab^*) \right\|_2^2
 \leq 
3  \lambda_1 \| \hat \gammab_{S_1} - \gammab^*_{S_1} \|_{1}  
+2 \lambda_1 \|  \gammab^*_{S_1} \|_2^2 +  \lambda_1 \|  \hat \gammab_{S_1} - \gammab^*_{S_1}\|_2^2  
\end{align} 
holds with probability approaching one.
Let $S = S_{\betab^*}$ for short.
 Condition {\bf ($\Gammab$)} and {\bf(CC)} together imply that now we have 
 \[
 (\phi_0^2 C_2 -\lambda_1)\|\hat \gammab_{S_1} - \gammab^*_{S_1}\|_2^2 \leq 3 \sqrt{s_{1}}  \lambda_1 \| \hat \gammab_{S_1} - \gammab^*_{S_1} \|_{2} + 2\lambda_1 \|  \gammab^*_{S_1} \|_2^2.
 \]
 Solving for $\| \hat \gammab_{S_1} - \gammab^*_{S_1} \|_{2}$ in the above inequality we obtain
 \[
 \| \hat \gammab_{S_1} - \gammab^*_{S_1} \|_{2} \leq 3 \sqrt{s_{1}} \lambda_1 /  ( \phi_0^2 C_2 -\lambda_1) + 2\sqrt{2 s_1} \lambda_1 K_{\gamma} /  ( \phi_0^2 C_2 -\lambda_1).
 \]
  
  The result then follows from a simple norm inequality 
\begin{align*}
 \| \hat \gammab - \gammab^* \|_{1}   \leq (a + 1) \| \hat \gammab_{S_1}- \gammab^*_{S_1} \|_{1} \leq (a + 1)\sqrt{s_1} \| \hat \gammab_{S_1} - \gammab^*_{S_1} \|_{2}
\end{align*}
and considering an asymptotic regime with $n,p,s_{\betab^*},s_1 \to \infty$.

\end{proof}

\begin{proof} [Proof of Lemma \ref{lem:temp1} ]
Recall the definitions of $\hat{\zetab}_j$ and $\zetab_j^*$.
Observe that  we have the following inequality,
\begin{align*}
\left| \hat{\zetab}_j^\top \hat{\zetab}_j /n- \EE {\zetab_j^*}^\top  \zetab_j^* /n\right|  &\leq \left|n^{-1}\hat{\zetab}_j^\top \hat{\zetab}_j -n^{-1}{\zetab_j^*}^\top \zetab_j^*\right|+\left|n^{-1}{\zetab_j^*}^\top \zetab_j^*-n^{-1}\EE{\zetab_j^*}^\top \zetab_j^*\right|\\
&\leq n^{-1}\left\|{\hat{\zetab}}_j +\zetab_j^*\right\|_\infty\left\|\hat{\zetab}_j -\zetab_j^*\right\|_1+\left|n^{-1}{\zetab_j^*}^\top \zetab_j^*-n^{-1}\EE{\zetab_j^*}^\top \zetab_j^*\right|,
\end{align*}
using triangular inequality and H\"older's inequality. 

We proceed to upper bound all of the three terms on the right hand side of the previous inequality.
First, we observe 
\begin{align}\label{zetahat+zeta}
\left\|\hat{\zetab}_j +\zetab_j^*\right\|_\infty   
&\leq \left\|W_{j}(\betab^*) - W_{-j}(\betab^*)\gammab_{(j)}^*(\betab^*) \right\|_\infty +\left\|W_{j}(\hat{\betab}) - W_{-j}(\hat\betab) \hat \gammab_{(j)} (\hat{\betab}) \right\|_\infty .
\end{align}
Moreover, the conditions imply that $\| W_{j}(\hat{\betab})\|_\infty \leq K  $ (by the Condition {\bf (X)}),  

$$\|W_{-j}\hat\gammab_{(j)} (\hat{\betab}) \|_\infty \leq K \left( \| \hat \gammab_{(j)}  (\hat{\betab})  - \gammab_{(j)}^*( {\betab}^*) \|_1+ \|\gammab_{(j)}^*( {\betab}^*) \|_1  \right) $$
 and by Lemma  \ref{cor:2}, for $\lambda_j$ as defined, the right hand size is $\mathcal{O}_P \left( K K_\gamma s_j (\lambda_j \vee 1) \right)$.
 Thus, we conclude $\left\|\hat{\zetab}_j +\zetab_j^*\right\|_\infty=\mathcal{O}_P\biggl( K(1+s_j) \bigvee K K_\gamma s_j (\lambda_j \vee 1) \biggl) = \ocal_P \left( K (1 \vee K_\gamma \vee K_\gamma \lambda_j) s_j \right)$. 
 
Its multiplying term can be decomposed as following
\begin{align}\label{zetahat-zeta}
n^{-1}\left\|\hat{\zetab}_j -\zetab_j^*\right\|_1  &\leq \underbrace{n^{-1}\left\|X_j\circ\left(\ind(X\hat{\betab}>0)-\ind(X\betab^*>0)\right)\right\|_1}_{\rmnum{1}} \nonumber\\
&\qquad +\underbrace{n^{-1}\left\| W_{-j}(\hat\betab) \hat \gammab_{(j)}(\hat{\betab}) - W_{-j}(\betab^*)\gammab_{(j)}^*(\betab^*)\right\|_1}_{\rmnum{2}},
\end{align}
where $\circ$ denotes entry wise multiplication between two vectors. The reason we have to spend such a great effort in separating the terms to bound this quantity is that we are dealing with a $1$-norm here, rather than an infinity-norm, which is bounded easily.  

We start with term $\rmnum{1}$. Notice that
\begin{align*}
n^{-1}\left\|X_j\circ\left(\ind(X\hat{\betab}>0)-\ind(X\betab^*>0)\right)\right\|_1 \leq Kn^{-1} \sum_{i=1}^n \left |\ind(x_i\hat{\betab}>0)-\ind(x_i\betab^*>0)\right |,
\end{align*}
by H\"older's inequality and Condition {\bf (X)}. 
 Moreover, by Lemma \ref{lemma0} we can easily bound the term above with $\ocal_P \left( K r_n^{1/2} t^{3/4} (\log p / n)^{1/2} \bigvee K t \log p / n \right)$, with $r_n$ and $t$ as defined in Condition {\bf{(I)}}.

 For the term $\rmnum{2}$, we have  
\begin{align*}
 \rmnum{2} \leq &  n^{-1}\left\|X_{-j} \hat\gammab_{(j)}(\hat{\betab})\circ \ind(X\hat{\betab}>0) - X_{-j}\gammab_{(j)}^*(\betab^*)\circ \ind(X\hat{\betab}>0) \right\|_1 \\
 & + n^{-1}\left\|X_{-j}\gammab_{(j)}^*(\betab^*)\circ \ind(X\hat\betab>0) - X_{-j}\gammab_{(j)}^*(\betab^*)\circ\ind(X\betab^*>0)\right\|_1.
 \end{align*}
 Observe, that the right hand side is upper bounded with
\begin{align*}
& K \left\| \hat\gammab_{(j)}(\hat{\betab})-\gammab_{(j)}^*(\betab^*) \right\|_1 \left\| \ind(X\hat{\betab}>0)\right\|_\infty \\ 
 &+\left\|X_{-j}\gammab_{(j)}^*(\betab^*) \right\|_\infty  \left | n^{-1} \sum_{i=1}^n  \left[ \ind(x_i\hat{\betab}>0) -  \ind(x_i{\betab}^*>0) \right]\right |  \end{align*}
  by Condition {(\bf X)}. Utilizing Lemma \ref{lemma0}, Lemma \ref{cor:2} and Condition {($\boldsymbol \Gamma$)}
  together   we obtain 
  $$\rmnum{2} =   \ocal_P\left( K K_\gamma s_j \lambda_j \right) +  \ocal_P \left( K K_\gamma r_n^{1/2} t^{3/4} s_j (\log p / n)^{1/2} \bigvee K K_\gamma t s_j \log p / n \right), $$
 for the chosen $\lambda_j$. Combining bounds for the terms $\rmnum{1}$ and $\rmnum{2}$, we obtain
 \[
 n^{-1}\left\|\hat{\zetab}_j -\zetab_j^*\right\|_1 = \ocal_P\left( K K_\gamma s_j \lambda_j \bigvee K K_\gamma r_n^{1/2} t^{3/4} s_j (\log p / n)^{1/2} \bigvee K K_\gamma t s_j \log p / n \right)
 \]
 
Next, we bound $\left|n^{-1}{\zetab_j^*}^\top \zetab_j^*-n^{-1}\EE{\zetab_j^*}^\top \zetab_j^*\right|$. If we rewrite the inner product in summation form, we have
$
\left|n^{-1}{\zetab_j^*}^\top \zetab_j^*-n^{-1}\EE{\zetab_j^*}^\top \zetab_j^*\right| = n^{-1}\sum_{i=1}^{n}\left({\zeta_{ij}^*}^2 - \EE{\zeta_{ij}^*}^2\right).
$
Notice that $\zeta_{ij}^* = W_{ij}(\betab^*) - W_{i,-j}\gammab_{(j)}^*(\betab^*)$ is a bounded random variable and such that $|\zeta_{ij}^* | = \mathcal{O}_P (K  (1+s_j) )$. We then apply Hoeffding's inequality for bounded random variables,  to obtain  $\left|n^{-1}{\zetab_j^*}^\top \zetab_j^*-n^{-1}\EE{\zetab_j^*}^\top \zetab_j^*\right|= O_P( K^2 (1 + s_j)^2n^{-1/2}).$
\end{proof}

\begin{proof}[Proof of Lemma   \ref{lemma2} ]

We begin by first establishing that $\hat \tau_j^{-2} = \Ocal_P(1)$. 
 In the case when the penalty part $\lambda_j\left\|\hat{\gammab}_{(j)}(\hat{\betab})\right\|_1$ happens to be $0$,  which means $\hat{\gammab}_{(j)}(\hat{\betab})=0$, the worst case scenario is that the regression part,
 $ n^{-1}\left\|W_{j}(\hat{\betab}) - W_{-j}(\hat{\betab})\hat{\gammab}_{(j)}(\hat{\betab})\right\|_2^2,$ 
  also results in $0$, i.e.
\begin{align}\label{eq:tau_zero}
0 
&= W_{j}(\hat{\betab}) - W_{-j}(\hat{\betab})\hat{\gammab}_{(j)}(\hat{\betab})
\end{align}
We show that these terms cannot be equal to zero simultaneously, since this forces $W_j(\hat\betab) = 0$, which is not true. Thus, $\hat \tau_j^{-2}$ is bounded away from 0.

In order to show results about the matrices $\Omega(\hat{\betab})$ and $\Omega(\betab^*)$, we first  provide a bound on the $\hat{\tau}$ and $\tau$. This is critical, since the magnitude of $\Omega(\cdot)$ is determined by $\tau$. To derive the bound on the $\tau$'s, we have to decompose the terms very carefully and put a bound on each one of them. 

Recall definitions of $\hat{\zetab}_j$ and $\zetab_j^*$ in \eqref{eqn:zeta*} we have 
  \begin{align*}
\hat{\zetab}_j = 
 W_{j}(\hat{\betab}) - W_{-j}(\hat{\betab}) \hat \gammab_{(j)} (\hat{\betab}), \qquad \zetab_j^*=W_{j}(\betab^*) - W_{-j}(\betab^*)\gammab_{(j)}^*(\betab^*).
\end{align*}
 Moreover, by the Karush-Kuhn-Tucker conditions of problem \eqref{nodewise_lasso}  we have
 $\lambda_j \| \hat \gammab_{(j)} (\hat{\betab})\|_1 = n^{-1} {\hat\zetab_j}^\top W_{-j}(\hat{\betab}) \hat \gammab (\hat \betab)  $,
  which in turn enables  a representation 
 \[
\hat  \tau_j^2=n^{-1} {\hat\zetab_j}^\top \hat\zetab_j +n^{-1} {\hat\zetab_j}^\top W_{-j}(\hat{\betab}) \hat \gammab (\hat \betab) .
 \]
By definition  we have that 
$\tau_j^2=n^{-1}\EE{\zetab_j^*}^\top \zetab_j^* ,$
for which we have $\hat{\tau}_j^2$ as an estimate. 
The $\tau_j^2$ and $\hat{\tau}_j^2$ carry information about the magnitude of the values in $\Sigmab^{-1}(\betab^*)$ and $\Omega(\hat{\betab})$ respectively.  We  next   break down $\tau_j^2$ and $\hat{\tau}_j^2$ into parts related to difference between $\hat{\gammab}_{(j)}(\hat{\betab})$ and $\gammab_{(j)}^*({\betab}^*)$, which we know how to control.
Thus, we have the following decomposition,
 \begin{align*}
\left|\hat{\tau}_j^2-\tau_j^2\right| 
&\leq \underbrace{\left|n^{-1}\hat{\zetab}_j^\top \hat{\zetab}_j -\tau_j^2\right|}_{\Rmnum{1}}
+\underbrace{\left|n^{-1}\hat{\zetab}_j^\top W_{-j}(\hat{\betab}) \hat\gammab_{(j)} ( \hat{\betab} )\right|}_{\Rmnum{2}} .
\end{align*}

The task now boils down to bounding each one of the   terms $\Rmnum{1}$ and $\Rmnum{2}$ ,   independently.   
Term $\Rmnum{1}$ is now bounded  by Lemma  \ref{lem:temp1} and is in order of 
$$\ocal_P\left( K^2 K_\gamma s_j^2 (1 \vee K_\gamma \vee K_\gamma \lambda_j) (\lambda_j \vee r_n^{1/2} t^{3/4} (\log p / n)^{1/2} \vee t \log p / n) \right).$$  
Regarding term $\Rmnum{2}$, we first point out one result due to the Karush-Kuhn-Tucker conditions of (6), 
\begin{align}\label{nodewise_lasso_kkt}
\lambda_j\cdot 1^\top  &\geq \lambda_j \mbox{sign}\left(\hat \gammab_{(j)} (\hat{\betab})\right)^\top = n^{-1}\left(W_{j}(\hat{\betab}) -W_{-j}(\hat{\betab}) \hat\gammab_{(j)} (\hat{\betab})\right)^\top  W_{-j}(\hat{\betab}) = n^{-1}\hat{\zetab}_j^\top W_{-j}(\hat{\betab})\nonumber.
\end{align}
 For the term $\Rmnum{2}$, we then have 
\begin{align*}
\left|n^{-1}\hat{\zetab}_j^\top W_{-j}(\hat{\betab}) \hat \gammab_{(j)} (\hat{\betab})\right| &\leq \left\|n^{-1}\hat{\zetab}_j^\top W_{-j}(\hat{\betab})\right\|_\infty \left\|\hat \gammab_{(j)} (\hat{\betab})\right\|_1 = \ocal_P \left( s_j \lambda_j \vee s_j \lambda_j^2 \right),
\end{align*}
since by Lemma \ref{cor:2} we have 
$$\left\| \hat \gammab_{(j)} (\hat  \betab)\right\|_1  \leq   \left\|\gammab_{(j)}^*(\betab^*)\right\|_1 +  \left\|\hat \gammab_{(j)} (\hat  \betab)-\gammab_{(j)}^*(\betab^*)\right\|_1=\mathcal{O}_P(s_j) +  \ocal_P(s_j \lambda_j).$$
 Putting all the pieces together, we have shown that rate
\begin{align*}
&\left|\hat{\tau}_j^2-\tau_j^2\right| = \ocal_P\left( K^2 K_\gamma s_j^2 (1 \vee K_\gamma \vee K_\gamma \lambda_j) (\lambda_j \vee r_n^{1/2} t^{3/4} (\log p / n)^{1/2} \vee t \log p / n) \right).
\end{align*}    

As $\hat \tau_j^{-2} = \Ocal_P(1)$ we have 
$ \left|\frac{1}{\hat{\tau}_j^2}-\frac{1}{\tau_j^2}\right| =  \ocal_P \left( \left|{\tau_j^2-\hat{\tau}_j^2} \right| \right) .$ 
We then conclude  
\begin{align*}
\left\|\Omegab(\hat{\betab})_j-\Sigmab^{-1}(\betab^*)_j\right\|_1  
&\leq \hat{\tau}_j^{-2}{\left\|\hat{\gammab}_{(j)}(\hat{\betab})-\gammab_{(j)}^*( {\betab}^*)\right\|_1}{} +\left\|\gammab_{(j)}^*(\betab^*)\right\|_1\left|\frac{1}{\hat{\tau}_j^2}-\frac{1}{\tau_j^2}\right|
\\
&= \ocal_P\left( K^2 K_\gamma^2 s_j^3 (1 \vee K_\gamma \vee K_\gamma \lambda_j) (\lambda_j \vee r_n^{1/2} t^{3/4} (\log p / n)^{1/2} \vee t \log p / n) \right).
\end{align*}
\end{proof}
 
\begin{proof} [Proof of Lemma \ref{lemma4}]
For the simplicity of the proof we  introduce some additional notation.
Let $\deltab = \hat \betab - \betab^*$, and
\[
\nu_n(\deltab) = n^{-1} \sum_{i=1}^n \Omegab(\hat \betab) \left[ \psi _i (\hat \betab) -  \psi_i(\betab^*)\right].
\]
Observe that
$
\ind \left\{ y_i - x_i \hat \betab \leq  0\right\} = \ind \left\{ x_i \deltab  \geq  \varepsilon_i\right\}
$
and hence 
$1 -2 \ind\{y_i - x_i \hat \betab > 0 \} = 2\ind\left\{y_i - x_i \hat \betab \leq 0 \right\} -1$.
The term we wish to bound then can be expressed as 
\[
\mathbb{V}_n(\deltab)=\nu_n(\deltab) - \EE \nu_n(\deltab)
\]
for $\nu_n(\deltab)$ denoting the following quantity
\[
\nu_n(\deltab) = n^{-1} \sum_{i=1}^n \Omegab(\deltab +   \betab^*)  X_i \left[ f_i(\deltab) g_i(\deltab) - f_i(\zero) g_i(\zero)\right]
\]
and 
\[
f_i(\deltab) = \ind  \left\{ x_i \deltab \geq - x_i \betab^*\right\} , \qquad g_i(\deltab) = 2 \ind  \left\{    x_i \deltab \geq \varepsilon_i\right\}  -1.
\]

Let $\{\tilde \deltab_k\}_{k
  \in [N_\delta]}$ be centers of the balls of radius $r_n\xi_n$ that
cover the set $\Ccal(r_n,t)$. Such a cover can be constructed with
$N_\delta \leq {p \choose t} (3/\xi_n)^t $ \citep[see, for example  ][]{vdvaart_1998}. 
Furthermore, let 
\begin{equation*}
\mathcal{B} (\tilde \deltab_k, r) = 
\cbr{
\deltab \in \RR^p : 
\norm{\tilde \deltab_k-\deltab}_2 \leq r
\ , \ 
{\mbox{supp}}(\deltab) \subseteq {\mbox{supp}}(\tilde \deltab_k)
}
\end{equation*}
be a ball of radius $r$ centered at $\tilde \deltab_k$ with  elements
that have the same support as $\tilde \deltab_k$. In what follows, we
will bound $ \sup_{\deltab \in
  \Ccal(r_n, t)} \norm{ \mathbb{V}_n(\deltab)}_\infty$ using an
$\epsilon$-net argument.  In particular, using the above introduced
notation, we have the following decomposition
\begin{equation}
\label{eq:proof:lem_lin:t1t2}
\begin{aligned}
&
\sup_{\deltab \in \Ccal(r_n, t)}
\norm{ \mathbb{V}_n(\deltab)}_\infty  =  \max_{k \in [N_\delta]}
\sup_{\deltab \in \mathcal{B} (\tilde \deltab_k, r_n\xi_n)}
\norm{ \mathbb{V}_n( \deltab)}_\infty \\
& \leq
\underbrace{
  \max_{k \in [N_\delta]}
  \norm{\mathbb{V}_n( \tilde \deltab_k)}_\infty 
  }_{T_1}
 + 
\underbrace{
  \max_{k \in [N_\delta]}
  \sup_{\deltab \in \mathcal{B} (\tilde \deltab_k, r_n\xi_n)}
  \norm{ \mathbb{V}_n( \deltab) -  \mathbb{V}_n( \tilde \deltab_k)}_\infty
}_{T_2}.
\end{aligned}
\end{equation}

Observe that the term $T_1$ arises from discretization of the sets
 $ \Ccal(r_n, t)$. To control it, we will
apply the tail bounds for each fixed $l$ and $k$. The term $T_2$
captures the deviation of the process in a small neighborhood around
the fixed center   $\tilde\deltab_k$. For those deviations we will provide  covering number arguments. In the
remainder of the proof, we provide details for bounding $T_1$ and
$T_2$.

We first bound the term $T_1$ in \eqref{eq:proof:lem_lin:t1t2}.
Let $a_{ij}(\betab) = \eb_j^\top  \Omegab(  \betab) \Xb_i$.
We are going to decouple dependence on $X_i$ and $\varepsilon_i$. To that end, let
\begin{align*}
  Z_{ijk}  &= a_{ij}(\betab^* + \tilde\deltab_k)
  \rbr{
    \rbr{
      f_i( \tilde\deltab_k) g_i(\tilde\deltab_k) - 
      \EE \left[ f_i( \tilde\deltab_k) g_i(\tilde\deltab_k)  | X_i \right]}
    -     \rbr{
       f_i( \zero) g_i(\zero) - 
       \EE \left[ f_i( \zero) g_i(\zero)  | X_i \right]}}
  \\
\intertext{and}
  \tilde Z_{ijk}  &= 
    a_{ij}(\betab^* + \tilde\deltab_k)
    \rbr{
     \EE \left[ f_i( \tilde\deltab_k) g_i(\tilde\deltab_k)  | X_i \right]
      -
     \EE \left[ f_i( \zero) g_i(\zero)  | X_i \right]
    }
    \\
    &-
    \EE\sbr{
         a_{ij}(\betab^* + \tilde\deltab_k)
    \rbr{
      f_i( \tilde\deltab_k) g_i(\tilde\deltab_k) -
     f_i( \zero) g_i(\zero)  }
    }.
\end{align*}

With a  little abuse of notation we use $f$ to denote the density of $\varepsilon_i$ for all $i$. Observe that   $\EE \left[ f_i(  \deltab) g_i( \deltab)  | X_i \right] =f_i(  \deltab) \PP(\varepsilon_i \leq X_i  \deltab)$. We use $w_i(\deltab)$ to denote the right hand side of the previous equation.

Then
\begin{align*}
  T_1 &= 
    \max_{k \in [N_\delta]}
    \max_{j \in [p]}
    \abr{
      n^{-1}\sum_{i\in[n]} \rbr{Z_{ijk} + \tilde Z_{ijk}}} \leq 
    \underbrace{
    \max_{k \in [N_\delta]}
    \max_{j \in [p]}
    \abr{
      n^{-1}\sum_{i\in[n]} Z_{ijk}
    } }_{T_{11}}  +
    \underbrace{
    \max_{k \in [N_\delta]}
    \max_{j \in [p]}    
    \abr{       n^{-1}\sum_{i\in[n]}  \tilde Z_{ijk} 
    }}_{T_{12}}.
\end{align*}
Note that $\EE[Z_{ijk} \mid \{\Xb_i\}_{i\in[n]}] = 0$ and 
\begin{align*}
  \Var&[Z_{ijk} \mid \{\Xb_i\}_{i\in[n]} ]\\
   &=      a_{ij}^2(\betab^* + \tilde\deltab_k)
  \Big(
    w_i( \tilde\deltab_k)
    - w_i^2( \tilde\deltab_k)  + 
    w_i(  \zero)
    - w_i^2(  \zero)    
    \\
    & \qquad\qquad\qquad\qquad\qquad\qquad-
    2 \rbr{w_i( \zero) \vee w_i(\tilde\deltab_k)}
    + 
    2 w_i\rbr{  \tilde\deltab_k}
      w_i\rbr{ \zero}    
  \Big) \\
  & \stackrel{(i)}{\leq}
a_{ij}^2(\betab^* + \tilde\deltab_k)
  \rbr{
    w_i( \tilde\deltab_k)
    +
    w_i( \zero)
    -
    2 \rbr{w_i(  \zero) \vee w_i( \tilde\deltab_k)}
  }
 \\
  & \stackrel{(ii)}{\leq}
a_{ij}^2(\betab^* + \tilde\deltab_k)
  f_i( \tilde \deltab_k)   \abr{
  x_i  \tilde\deltab_k
  }
  f\rbr{  \eta_ix_i  \tilde\deltab_k}
  \quad \rbr{\eta_i \in [0,1]} \\
  & \stackrel{(iii)}{\leq}
a_{ij}^2(\betab^* + \tilde\deltab_k)
   f_i( \tilde \deltab_k)   \abr{
   x_i  \tilde\deltab_k
  }
  f_{\max}
\end{align*}
where $(i)$ follows by dropping a negative term, $(ii)$ follows by
the mean value theorem, and  $(iii)$ from the assumption that the
conditional density is bounded stated in Condition~{\bf(E)}.

Furthermore, conditional on $\{\Xb_i\}_{i\in[n]}$ we have that almost surely.
$|Z_{ijk}| \leq 4 \max_{ij}|a_{ij}(\betab^* + \tilde\deltab_k)|$.
We will work on the event 
\begin{align}  \label{eq:proof:lem_lin:event1}
  \Acal = \cbr{  \max_{i \in [n], j \in [p]}
    |a_{ij}(\betab^* + \tilde \deltab_k) -\Sigmab^{-1}_{ij} (\betab^*) | \leq C_n},  
\end{align}
which holds with probability at $1-\delta$ using
Lemma \ref{lemma2}.   For a
fixed $j$ and $k$ Bernstein's inequality \citep[see, for
example, Section 2.2.2 of][]{vanderVaart1996Weak} gives us
\begin{align*}
  \abr{n^{-1}\sum_{i\in[n]} Z_{ijk}}
  \leq C\rbr{
  \sqrt{
    \frac{f_{\max}\log(2/\delta)}{n^2}
    \sum_{i\in[n]}a_{ij}^2(\betab^* + \tilde\deltab_k)
    \abr{x_i  \tilde\deltab_k}
  }
  \bigvee \frac{\max_{i\in[n],j\in[p]}|a_{ij}(\betab^* + \tilde\deltab_k)|}{n}
  \log(2/\delta)
  }
\end{align*}
with probability $1-\delta$. On the event $\mathcal{A}$
\begin{align*}
    \sum_{i\in[n]}a_{ij}^2(\betab^* + \tilde\deltab_k)
    \abr{x_i  \tilde\deltab_k} 
    & \leq C_n^2\sqrt{\tilde\deltab_k^\top  W(\betab^*+\tilde\deltab_k) W^\top(\betab^*+\tilde\deltab_k)\tilde\deltab_k} \leq (1+o_P(1))C_n^2 r_n\Lambda_{\max}^{1/2}(\Sigmab (\betab^*)),
\end{align*}
where the line follows using the Cauchy-Schwartz inequality and
inequality (58a) of \citet{wainwright_2009} and Lemma \ref{lemma2}.  Combining all of the results above,
with probability $1-2\delta$ we have that
\begin{align*}
  \abr{n^{-1}\sum_{i\in[n]}Z_{ijk}}
  \leq C\rbr{
  \sqrt{
    \frac{C_n^2r_n\log(2/\delta)}{n}
  }
  \bigvee \frac{C_n\log(2/\delta)}{n} }.
\end{align*} 
Using the union bound over $j\in[p]$ and
$k\in[N_\delta]$, with probability $1-2\delta$, we have
\[
T_{11} \leq C \rbr{
  \sqrt{
    \frac{C_nr_n\log(2  
      N_\delta p /\delta)}{n}
  }
  \bigvee \frac{C_n\log(2 
      N_\delta p/\delta)}{n} }.
\]
We deal with the term $T_{12}$ in a similar way. 
 For a fixed $k$ and $j$, conditional on the event $\mathcal{A}$ we apply
Bernstein's inequality to obtain
\[
\abr{n^{-1}\sum_{i\in[n]}\tilde Z_{ijk}} \leq C \rbr{
\sqrt{\frac{C_n^2r_n^2\log(2/\delta)}{n}} 
\bigvee
\frac{C_n\log(2/\delta)}{n}}
\]
with probability $1-\delta$, since
 on the event $\Acal$ in \eqref{eq:proof:lem_lin:event1} we have that
$\abr{
  \tilde Z_{ijk}
} \leq 4C_n
$
and
\begin{align*}
\Var
\sbr{
\tilde Z_{ijk}
} 
& \leq 
\EE
\sbr{
        a_{ij}^2(\betab^* + \tilde\deltab_k)
        \rbr{
          f_i( \tilde\deltab_k) \PP(\varepsilon_i \leq X_i  \tilde\deltab_k)
          -
          f_i(  0) \PP(\varepsilon_i \leq 0)
        }^2
} \\
& \leq C_n^2 f_{\max} \left( G_i(  \tilde\deltab_k,  \betab^* , 0) - G_i( 0,  \betab^* , 0) \right)^2 r_n^2\Lambda_{\max}(\Sigmab (\betab^*)) \leq C C_n^2 K_1^2r_n^2 .
\end{align*}
where in the last step we utilized Condition ({\bf E}) with $z =r_n$.
The union bound over  $k\in[N_\delta]$, and $j\in[p]$,
gives us
\begin{equation*}
T_{12} \leq C \rbr{
\sqrt{\frac{C_n^2 K_1^2 r_n^2
\log(2  N_\delta p /\delta)}{n}} 
\bigvee
\frac{C_n\log(2 N_\delta p /\delta)}{n} }  
\end{equation*}
with probability at least $1-2\delta$. Combining the bounds on
$T_{11}$ and $T_{12}$, with probability $1-4\delta$,  we have
\begin{align*}
T_{1} & \leq C \rbr{
\sqrt{\frac{C_n^2 (r_n \vee r_n^2K_1^2)
\log(2  N_\delta p /\delta)}{n}} 
\bigvee
\frac{C_n\log(2 N_\delta p /\delta)}{n} },
\end{align*}
since $ r_n = \Ocal_P(1)$.  Let us now focus on bounding $T_2$ term.  Note that $a_{ij}(\betab^* + \deltab_k) =
a_{ij}(\betab^*) + a'_{ij}(\bar\betab_k)\deltab_k$ for
some $\bar \betab_k$ between $\betab^* + \deltab_k$ and $\betab^*$. Let 
\[
W_{ij}( \deltab) = 
      a'_{ij}(\bar\betab_k)\deltab_k
      \rbr{ 
         f_i(  \deltab ) g_i( \deltab )  -  f_i( \zero) g_i(\zero) 
      } ,
\]
and 
\[
Q_{ij}( \deltab) = 
      a_{ij}(\betab^*)
      \rbr{ 
         f_i(  \deltab ) g_i( \deltab )  -  f_i( \zero) g_i(\zero) 
      } .
\]
Let $\mathbb{Q}(\deltab) = Q (\deltab)- \EE[Q(\deltab)]$.
For a fixed $j$, and $k$ we have $  \sup_{\deltab \in \mathcal{B} (\tilde \deltab_k, r_n\xi_n)}
  \abr{ \eb_j^\top \rbr{
       \mathbb{V}_n(  \deltab) -  \mathbb{V}_n( 
      \tilde\deltab_k)}
  }$ 
  is upper bounded with
\begin{align*}  
  \underbrace{
  \sup_{\deltab \in \mathcal{B}(\tilde \deltab_k, r_n\xi_n)}
  \abr{
    n^{-1}\sum_{i\in[n]}
    \mathbb{Q}_{ij}(\deltab)
    -
        \mathbb{Q}_{ij}(\tilde\deltab_k)}}
   _{T_{21}}   + 
  \underbrace{
  \sup_{\deltab \in \mathcal{B}(\tilde \deltab_k, r_n\xi_n)}
  \abr{
    n^{-1}\sum_{i\in[n]}
    W_{ij}(\deltab)
    -
      \EE\sbr{
        W_{ij}(\deltab)
      }
  }
  }_{T_{22}}.
\end{align*}

We will deal with the two terms separately. 
Let $Z_i =  \max\{ \varepsilon_i, - X_i \betab^*\}$
\begin{align*}
f_i(\deltab) g_i(\deltab) = \ind\{X_i \deltab \geq Z_i\} - \ind\left\{ X_i \deltab \geq - X_i \betab^*\right\}.
 \end{align*}
 Observe that the distribution of $Z_i$ is the same as the distribution of $|\varepsilon_i|$ due to the Condition ({\bf E}).
 Moreover,
 \[\abr{x_i  (\deltab-\tilde\deltab_k)} \leq K
\norm{\deltab-\tilde\deltab_k}_2
\sqrt{\abr{\mbox{supp}(\deltab-\tilde\deltab_k)}}   
\]
where $K$ is a constant such that    $\max_{i,j}|x_{ij}| \leq K$.
Hence,

\begin{align}\label{eq:bound1}
 &\max_{k \in [N_\delta]}\max_{i\in[n]}
  \sup_{\deltab \in \mathcal{B}(\tilde \deltab_k, r_n\xi_n)}
  \abr{
     x_i  \deltab
    -
      x_i  \tilde\deltab_k
  } \leq 
    r_n\xi_n\sqrt{t} \max_{i,j}|x_{ij}|
  \leq C
    r_n\xi_n
    \sqrt{ t  }
 =: \tilde L_n.
\end{align}
For $T_{21}$, we will use
the fact that $\ind\{a < x\}$ and $\PP\{Z < x\}$ are monotone function in $x$. 
Therefore,
\begin{align*}
T_{21} 
&\leq
n^{-1}\sum_{i\in[n]}\bigg[  
\abr{a_{ij}(\betab^*)}
\Big(
  \ind\cbr{Z_i \leq   x_i  \tilde\deltab_k + \tilde L_n}-
  \ind\cbr{-X_i\betab^* \leq  x_i  \tilde\deltab_k  - \tilde L_n}-
  \ind\cbr{Z_i \leq  x_i  \tilde\deltab_k } \\
  &   +
  \ind\cbr{-X_i\betab^* \leq  x_i  \tilde\deltab_k  }   -
  \PP\sbr{Z_i \leq  x_i  \tilde\deltab_k - \tilde L_n}+
  \PP\sbr{-X_i \betab^* \leq x_i  \tilde\deltab_k+ \tilde L_n} \\
  &  +
  \PP\sbr{Z_i \leq  x_i  \tilde\deltab_k }-
  \PP\sbr{-X_i \betab^*\leq  x_i  \tilde\deltab_k  }
\Big)
\bigg] 
\end{align*}
Furthermore, by adding and substracting appropriate terms we can decompose the right hand side above into two terms. The first,
\begin{align*}
& 
n^{-1}\sum_{i\in[n]}\bigg[  
\abr{a_{ij}(\betab^*)}
\Big(
  \ind\cbr{Z_i \leq   x_i  \tilde\deltab_k + \tilde L_n}-
  \ind\cbr{ -Z_i \betab^*\leq  x_i  \tilde\deltab_k - \tilde L_n}  -
   \ind\cbr{Z_i \leq  x_i  \tilde\deltab_k }\\
  &  +
  \ind\cbr{-X_i\betab^* \leq  x_i  \tilde\deltab_k  }    -
  \PP\sbr{Z_i \leq  x_i  \tilde\deltab_k + \tilde L_n}+
  \PP\sbr{-X_i \betab^* \leq x_i  \tilde\deltab_k- \tilde L_n}\\
  &  +
  \PP\sbr{Z_i \leq  x_i  \tilde\deltab_k }-
  \PP\sbr{-X_i \betab^*\leq  x_i  \tilde\deltab_k  }
\Big)
\bigg] 
\end{align*}
and the second
\begin{align*}
&
n^{-1}\sum_{i\in[n]}\bigg[  
\abr{a_{ij}(\betab^*)}
\Big(
  \PP\sbr{Z_i \leq  x_i  \tilde\deltab_k + \tilde L_n}-
  \PP\sbr{-X_i \betab^* \leq x_i  \tilde\deltab_k- \tilde L_n} \\
  &  -
  \PP\sbr{Z_i \leq  x_i  \tilde\deltab_k - \tilde L_n}+
  \PP\sbr{-X_i \betab^* \leq x_i  \tilde\deltab_k+ \tilde L_n}
\Big)
\bigg].
\end{align*}
The first term in the display above can be bounded in a similar way to
$T_1$ by applying Bernstein's inequality and hence the details are
omitted. For the second term we have a bound
$CC_n\tilde L_n$, since
$\abr{a_{ij}(\betab^*)} \leq C_n $ by the definition of $a_{ij}$ and Lemma \ref{lemma2} and 
$  \PP\sbr{Z_i \leq  x_i  \tilde\deltab_k + \tilde L_n} - \PP\sbr{Z_i \leq  x_i  \tilde\deltab_k - \tilde L_n} \leq C \|f_{|\varepsilon_i|}\| _\infty\tilde L_n \leq 2C f_{\max} \tilde L_n$. In the last inequality we used the fact that $\|f_{|\varepsilon_i|}\| _\infty \leq 2\|f_{ \varepsilon_i }\| _\infty$.
Therefore, with probability $1-2\delta$,
\begin{align*}
  T_{21} 
  \leq C\rbr{
  \sqrt{
    \frac{f_{\max}C_n^2\tilde L_n\log(2/\delta)}{n}
  }
  \bigvee \frac{C_n\log(2/\delta)}{n} 
  \bigvee  f_{\max} \tilde L_n}.
\end{align*} 
A bound on $T_{22}$ is obtain similarly to that on $T_{21}$. The only
difference is that we need to bound $a'_{ij}(\bar\betab_k)\deltab_k$, for $\bar \betab_k = \alpha \betab^* + (1-\alpha) (\betab^* + \tilde\deltab_k) $ and $\alpha \in (0,1)$,   instead of $|a_{ij}(\betab^*)|$. Observe that 
$a_{ij}(\betab) {\hat \tau_j}^{2}= -  \hat \gamma_{(j),i}$.
Moreover, by construction $\hat \tau_j$ is a continuous, differentiable and convex function of $\betab$ and is bounded away from zero by Lemma \ref{lemma2}. Additionally, $\hat \gammab_{(j)}$ is a convex function of $\betab$ as a set of solutions of a minimization of a convex function over a convex constraint is a convex set. Moreover, $\hat \gamma_{j}$ is a bounded random variable according to Lemma \ref{lemma2}.
Hence,
$|a'_{ij}(\betab^*)| \leq K' $,  for a large enough constant $K'$.
 Therefore, for a large enough constant $C$ we have
\begin{align*}
  T_{22} 
  \leq C\Bigg(&
  \sqrt{
    \frac{f_{\max} r_n^2 \zeta_n^2   \tilde L_n
      \log(2/\delta)}{n}
  }  
  \bigvee \frac{ \tilde L_n \log(2/\delta)}{n}  
  \bigvee  f_{\max} C_n\tilde L_n \Bigg).
\end{align*} 
A bound on $T_2$ now follows using a union bound over $j\in[p]$  and $k\in[N_\delta]$.

We can choose   $\xi_n = n^{-1}$, which gives us
$N_\delta \lesssim \rbr{pn^2}^t $. With these choices, the term $T_2$
is negligible compared to $T_1$ and we obtain 
\begin{align*}
T & \leq C  \rbr{
\sqrt{\frac{C_n^2 (r_n \vee r_n^2K_1^2) t 
\log(n p /\delta)}{n}} 
\bigvee
\frac{C_n t \log(2n p /\delta)}{n} },
\end{align*}
which completes the proof.

\end{proof}

{

}

\end{document}